\documentclass{siamltex}
\usepackage{amssymb}
\usepackage{lmodern}
\usepackage{bm}
\usepackage{graphicx}
\usepackage{graphics}
\usepackage{caption2}
\usepackage{epstopdf}
\usepackage{psfrag}
\usepackage{amsmath}
\usepackage{lineno}
\usepackage{amscd}
\usepackage{amsfonts}
\usepackage{float}
\usepackage{latexsym}
\usepackage{multirow}\usepackage{bbm}
\usepackage{lscape}
\usepackage{comment}
\usepackage{color}
\usepackage[ruled,linesnumbered]{algorithm2e}

\newcommand{\ds}{\displaystyle}
\newcommand{\f}{\frac}

\newcommand{\om}{\Omega}
\newcommand{\p}{\partial}

\newtheorem{remark}{Remark}[section]
\newtheorem{thm}{Theorem}[section]

\newtheorem{prop}[thm]{Proposition}
\numberwithin{equation}{section}

\usepackage[marginal]{footmisc}

\begin{document}
\title{A time-fractional optimal transport and  mean-field planning:   Formulation and  algorithm}

\author{Yiqun Li\thanks{Department of Mathematics, University of South Carolina, Columbia, SC 29208, USA}
\and
Hong Wang\thanks{Department of Mathematics, University of South Carolina, Columbia, SC 29208, USA}
\and
Wuchen Li\thanks{Department of Mathematics, University of South Carolina, Columbia, SC 29208, USA (Corresponding author, email: wuchen@mailbox.sc.edu)}
}

\maketitle
\begin{abstract}
The time-fractional optimal transport (OT) and mean-field planning (MFP) models are developed to describe the anomalous transport of the agents  in a heterogeneous environment such that  their densities are transported from the initial density distribution to the terminal one with the minimal cost. We derive a strongly coupled nonlinear system of a  time-fractional transport equation and a backward time-fractional Hamilton-Jacobi equation based on the first-order optimality condition. The general-proximal primal-dual hybrid gradient (G-prox PDHG) algorithm is applied to discretize the OT and MFP formulations, in which  a  preconditioner induced by the numerical approximation to the time-fractional PDE is derived to accelerate the convergence of the algorithm for both problems.  Numerical experiments for OT and MFP problems between Gaussian distributions and between image densities   are carried out to investigate the performance of the OT and MFP formulations. Those numerical experiments also demonstrate the effectiveness and flexibility of our proposed algorithm.
\end{abstract}
	
	
	
	\pagestyle{myheadings}
	\thispagestyle{plain}
	
\section{Introduction}\label{Model:intro}

Mean-field planning (MFP) problems model the interaction of the  infinite number of agents with the  given initial and terminal density distributions in a mean-field manner, describing the equilibrium state of the system \cite{HuaMal,MicLio}.
In addition, MFP are  generalizations of optimal transport (OT) problems \cite{BenBre1,BenBre2,GlaKop,Lav,LavCla,LeeLai,PapPey,PeyCut,Por,Vil,YuLaiJCP}. OT and MFP problems have demonstrated  wide applications in diverse fields, e.g., pedestrian motion, robotics path planning, reaction-diffusions, reinforced learning as well as   finance  problems \cite{BenCar,ChaPoc1,Gomes,GueLas,HakZhu,HuaMal,LinFun,LiLee,LiLu,LiuJac,RutOsh}.

Note that most works on OT and MFP problems in the literature are governed by the integer-order partial differential equations  (PDEs) \cite{AchCap1,AchCap2,DinLi,FuOsh2,YuLai}.
This issue may be elaborated in the context of OT. Let $\Omega \subset \mathbb{R}^d$ $ (d = 1, 2, 3)$ be a bounded domain, and let $\rho$ : $ \Omega  \times [0,T] \rightarrow \mathbb{R}^{+}$ be the density of agents, $\bm v : \Omega  \times [0,T] \rightarrow \mathbb{R}^d$ be the velocity field and $\bm m : = \rho(\bm x, t) \bm v(\bm x, t)$ be the  mass flux  which models strategies (control) of the agents.
Integer-order OT is reformulated as the following: Find the density $\rho$ and the mass flux $\bm m$ such that
\begin{equation}\label{OMT:e2}
\inf_{\forall \rho, \bm m} ~~\int_0^T \int_{\Omega} \f{ | \bm m(\bm x, t) |^2}{2 \rho(\bm x, t)}~ d\bm x dt,
\end{equation}
which is constrained by
\begin{equation}\begin{array}{cl}\label{Model:Int}
\ds \p_t \rho+\nabla \cdot \bm{m}=0, ~&\ds \hspace{-0.1in}~ \mbox{in}~~ \Omega \times (0, T], \\[0.05in]
\ds \rho(\bm x, 0)=\rho_0(\bm x), \quad \rho(\bm x, T)=\rho_1(\bm x), ~&\ds \hspace{-0.1in}~ \mbox{on} ~~ \Omega,
\\[0.05in]
\ds  \bm {m}(\bm x, t) \cdot \bm{n}(\bm x)=0, ~&\ds \hspace{-0.1in}~ \mbox{on}\; \partial \Omega \times [0, T].
\end{array}
\end{equation}
The derivation of the OT problem reveals that this formulation describes transport taking place in a homogeneous medium \cite{Vil}. This explains why the constraining PDE in \eqref{Model:Int} is integer-order. However, in a heterogeneous porous medium, the interaction between the agents and/or agents and environment dominates the transport process. Consequently, the travel time of the agents may deviate significantly from the transport of the agents in a homogeneous environment \cite{ZhoStr}, leading to anomalous transport that is characterized by nonlocal memory effect described by a time-fractional transport PDE. The corresponding OT formulation \eqref{OMT:e2} constrained by transport in a heterogeneous environment can be described by
\begin{equation}\begin{array}{cl}\label{Model:set}
\ds \p_t^\alpha \rho+\nabla \cdot \bm{m}=0, ~&\ds \hspace{-0.1in}~ \mbox{in}~~ \Omega \times (0, T], \\[0.05in]
\ds \rho(\bm x, 0)=\rho_0(\bm x), \quad \rho(\bm x, T)=\rho_1(\bm x), ~&\ds \hspace{-0.1in}~ \mbox{on} ~~ \Omega ,\\[0.05in]
\ds  \bm {m}(\bm x, t) \cdot \bm{n}(\bm x)=0, ~&\ds \hspace{-0.1in}~ \mbox{on}\; \partial \Omega \times [0, T].
\end{array}
\end{equation}
Here $\p_t^\alpha$ represents the Caputo fractional derivative operator with $\Gamma(\cdot)$ being the Gamma function  defined by \cite{Pod}
\begin{equation}\label{frac}
  \ds  \p_t^\alpha g := {}_0 I_t^{1-\alpha}\partial_t g, \quad {}_0I_t^{1-\alpha}g :=\frac{1}{\Gamma(1-\alpha)}\int_0^t \frac{g(s)}{(t-s)^{\alpha}}ds, \quad \ds 0 < \alpha < 1.
\end{equation}

In this work, we introduce  a Lagrangian multiplier
 and utilize the first-order optimality condition to derive a strongly coupled nonlinear system \eqref{Model:KKT} of a  time-fractional  transport PDE  in terms of the density and a backward time-fractional Hamilton-Jacobi PDE in terms of the  Lagrangian multiplier  $\phi$. Numerically, we discretize the generalized Lagrangian formulations of the OT and  MFP systems with the time-fractional differential operator \eqref{frac} approximated by the widely used L1 scheme. The discretizations of the adjoint operators are developed such that the duality properties of the continuous problem is preserved on the discrete level, based on which we further derive a discrete system \eqref{dis:KKT} approximating the nonlinear system \eqref{Model:KKT} on the first-order optimality condition.
The discrete saddle-point system is solved via the general-proximal primal-dual hybrid gradient (G-prox PDHG) algorithm, in which we apply a  preconditioner  generated by the numerical approximation to the time-fractional PDE to implement this algorithm with larger step sizes independent of the grid sizes to  accelerate the convergence of the algorithm for  both problems.
We carry out numerical experiments to investigate the performance of the OT and  MFP  formulation: \textbf{(i)} We investigate the convergence of the numerical approximation. \textbf{(ii)} With the numerically examined convergence, we investigate the performance of OT and MFP problems between Gaussian distributions and between image densities. Those experiments further demonstrate the effectiveness and flexibility of our proposed algorithm.

 The rest of the paper is organized as follows. In Section \S \ref{frac:MFP}, we formulate the time-fractional OT and MFP problems.
 We then derive a strongly coupled nonlinear system of a  time-fractional  transport PDE in terms of the density and a backward time-fractional Hamilton-Jacobi PDE in terms of the  Lagrangian multiplier  $\phi$ based on the first-order optimality condition.
In \S \ref{Disc:MFP}, we develop a numerical approximation to the generalized Lagrangian formulations of the OT and  MFP problems. We further derive a discrete system approximating the nonlinear system arising from the continuous problem. We then apply the G-prox PDHG algorithm to solve  our problem in \S \ref{PDHG:MFP}. The computational details are provided at the end of this section.
In \S \ref{Nume:MFP},  we carry out numerical experiments to investigate the performance of the OT and MFP  formulation.

\section{Fractional OT and  MFP}\label{frac:MFP}
\subsection{Problem formulation}
 We first investigate some basic  properties of the time-fractional transport PDE   \eqref{Model:set}, governed by which we then formulate the time-fractional OT and MFP problems. We assume that both $\rho$ and $\bm m$ exhibit appropriate smoothness  such that the subsequent derivations are well-defined.

\begin{prop}\label{mass} If problem \eqref{Model:set} is solvable then it is mass-conservative, i.e.,
\begin{equation}\label{thm:mass}
\int_{\Omega} \rho(\bm x, t)\, d \bm x = \int_{\Omega} \rho_0(\bm x)\, d \bm x, \quad \forall t \in [0, T].
\end{equation}
Conversely, if $\int_{\Omega} \rho_0\, d \bm x = \int_{\Omega} \rho_1 \,d \bm x$ then problem \eqref{Model:set} is solvable.
\end{prop}
\begin{proof} Suppose \eqref{Model:set} is solvable. Integrate the governing PDE on $\Omega$ and incorporate the no-flow boundary condition in \eqref{Model:set} to obtain
\begin{equation*}
{}_0 I_t^{1-\alpha}\partial_t \int_{\Omega}  \rho(\bm x, t) \,d \bm x =\p_t^\alpha  \int_{\Omega} \rho(\bm x, t) \,d \bm x  = \int_{\Omega} \big (\p_t^\alpha \rho(\bm x, t) + \nabla \cdot \bm m \big ) \,d \bm x =  0.
\end{equation*}
Apply ${}_0 I_t^{\alpha}$ on both sides of the equation and use ${}_0 I_t^{\alpha} \,{}_0 I_t^{1-\alpha} = {}_0 I_t^{1}$ \cite{Pod} to arrive at \eqref{thm:mass}.

Conversely, $\rho(\bm x, t) := \f{(T-t)}{T} \rho_0 (\bm x) + \f{t}{T} \rho_1(\bm x)$ satisfies \eqref{thm:mass}. Direct calculation gives
$$ \p_t^\alpha \rho (\bm x, t)  = \f{ t^{1-\alpha}}{T\Gamma(2-\alpha)} \big(\rho_1(\bm x) - \rho_0(\bm x)\big).$$
Recall that the Neumann boundary-value problem of the Poisson PDE
\begin{equation*}
-\Delta \psi = - \f{ t^{1-\alpha} (\rho_1 (\bm x)- \rho_0(\bm x))}{T\Gamma(2-\alpha)}~~\mathrm{in}~\Omega \times [0,T], \quad
\nabla \psi(\bm x, t) \cdot \bm n(\bm x) = 0 ~~\p \Omega \times [0,T]
\end{equation*}
admits the unique solution up to an additive constant. Hence, $\bm m=- \nabla \psi $ satisfies
\begin{equation*}
 \nabla \cdot \bm{m}  = - \f{ t^{1-\alpha}}{T\Gamma(2-\alpha)} (\rho_1 - \rho_0) = - \p_t^\alpha \rho, \quad \bm m \cdot \bm n = 0.
\end{equation*}
Namely, problem \eqref{Model:set} is solvable.
\end{proof}
\begin{remark} Proposition \ref{mass} extends the well-known result for the integer-order transport PDE \cite{FuOsh1,YuLai} to the current time-fractional transport PDE.
\end{remark}

Let $\mathcal C(\rho_0, \rho_1)$ denote the non-empty set of the solutions $(\rho,\bm m)$ to problem \eqref{Model:set}.
Let $L: \mathbb{R}^{+} \times \mathbb{R}^d \rightarrow \overline{\mathbb{R}}:=\mathbb{R} \cup\{\infty\}$ be the dynamic cost function and $F: \mathbb{R} \rightarrow \overline{\mathbb{R}}$ models interaction cost. The goal of dynamic MFP problem is to minimize the total cost among all feasible $(\rho, \bm m) \in \mathcal C(\rho_0, \rho_1)$
\begin{equation}\begin{array}{l}\label{Model:MFP}
\ds \min _{(\rho, \bm m) \in \mathcal{C} (\rho_0, \rho_1)} \int_0^T \int_{\Omega} L(\rho( \bm x, t), \bm m(\bm x, t)) d \bm x d t+ \int_0^T \int_{\Omega} F(\rho(\bm x, t)) d \bm xd t,
\end{array}
\end{equation}
where we consider a typical dynamic cost function
\begin{equation}\label{OT:L}
  L(\beta_0, \bm{\beta}):=
   \begin{cases}\f{\|\bm{\beta}\|^2}{2 \beta_0} & \text { if } \beta_0>0, \\
  0 & \text { if } \beta_0=0, \bm{\beta}=\mathbf{0}, \\
  +\infty & \text { if } \beta_0=0, \bm{\beta} \neq \mathbf{0},
  \end{cases}
\end{equation}
and
\begin{equation}\label{F}
\ds F (\rho(\bm x, t)):=  \lambda_R  R(\rho(\bm {x}, t))+\lambda_Q  \rho( \bm{x}, t) Q(\bm {x}).
\end{equation}
Since $\bm{m}=\rho \bm{v}$,   the  dynamic cost function $L$ defined in \eqref{OT:L} makes sure that $\bm{m}=\mathbf{0}$ if $\rho=0$.
In \eqref{F}, $\lambda_R$, $\lambda_Q \ge 0$ are fixed parameters,
$R: \mathbb{R}^{+} \rightarrow \mathbb{R}$ is a function to regularize $\rho$, and $Q(\bm{x}): \Omega \rightarrow \overline{\mathbb{R}}$ gives a moving preference for the density $\rho$. Consider $\Omega_0 \subset \Omega$ and define
$Q(\bm{x}):=
\left\{\begin{array}{ll}+\infty, & \bm{x} \in \Omega_0 \\
 0, & \bm{x} \notin \Omega_0
 \end{array}\right.
 $ in \eqref{F}, then the mass will move away from $\Omega_0$ to keep the  interaction cost finite.
 In general, the density $\rho( \bm{x}, t)$ tends to be smaller at the location where $Q(\bm{x})$ is larger and vice versa.

If the interaction cost $ F =0$, the MFP \eqref{Model:MFP} becomes the dynamic formulation of OT problem
\begin{equation}\begin{array}{l}\label{Model:OT}
\ds \min _{(\rho, \bm m) \in \mathcal{C} (\rho_0, \rho_1)} \int_0^T \int_{\Omega} L(\rho( \bm x, t), \bm m(\bm x, t)) d \bm x d t.
\end{array}
\end{equation}
 OT can be considered as a special case of MFP where masses move freely in $\Omega$ through $t \in[0, T]$.

\subsection{A coupled time-fractional PDE system from the optimality condition}

Let $\phi(\bm x,t)$ be a Lagrangian multiplier that is subject to the no-flux boundary condition \vspace{-0.1in}
\begin{equation}\label{phi:BC}
\nabla \phi(\bm x,t) \cdot \bm n(\bm x) = 0 \quad \mathrm{on}~ \p \om \times [0,T].
\end{equation}
We incorporate \eqref{OT:L} to define the generalized Lagrangian as follows:
\begin{equation}\label{Model:L}\begin{array}{rl}
\ds 	\mathcal L(\rho,\bm m,\phi)  \ds := \int_0^T \int_\Omega \f{| \bm m |^2}{2\rho} + F(\rho(\bm x, t))d\bm x dt + \int_0^T  \int_\Omega \phi \big (\p_t^{\alpha} \rho + \nabla \cdot {\bm  m} \big ) d\bm x dt.
\end{array}\end{equation}
Apply the Karush-Kuhn-Tucker approach to reformulate the constrained optimization problem \eqref{Model:MFP} as an unconstrained saddle-point problem: Find $(\rho, \bm m) \in \mathcal C (\rho_0, \rho_1)$ and $\phi$ to optimize the problem \vspace{-0.1in}
\begin{equation}\label{sad_L}
  \inf_{\rho, \bm m} \sup_{\phi}~ 	\mathcal L(\rho,\bm m,\phi).
\end{equation}
We incorporate the optimality condition of problem \eqref{Model:L}--\eqref{sad_L} to compute the first variation of the generalized Lagrangian $	\mathcal L(\rho,\bm m,\phi)$ in terms of all its arguments to arrive at the following proposition.

\begin{prop}\label{MFP}
The optimization problem \eqref{Model:L}--\eqref{sad_L} can be characterized by  the following coupled nonlinear time-fractional PDE system over $\Omega \times [0, T]$, which consists of a forward time-fractional transport equation describing the anomalous transport of the density and a  backward time-fractional Hamilton-Jacobi PDE describing the anomalous transport of the  Lagrangian multiplier  $\phi$:
\begin{equation}\label{Model:KKT}
\left\{\begin{array}{l}
\ds \p_t^{\alpha} \rho + \nabla \cdot {\bm  m} = 0, \\[0.1in]
\ds \hat \p_t^\alpha \phi -\f{|\bm m|^2}{2 \rho^2} + \mathcal \p_\rho F = 0,\\[0.1in]
\ds \f{\bm m}{\rho} =   \nabla \phi
\end{array}\right.
\end{equation}
with   initial,  terminal and boundary conditions
\begin{equation}\begin{array}{cl}\label{Model:KKT:ICBC}
\ds \rho(\bm x, 0) = \rho_0(\bm x), \quad  \rho(\bm x, T) = \rho_1(\bm x), \quad & \ds \mathrm{on} ~ \om,\\[0.05in]
\ds \bm m(\bm x,t) \cdot \bm n(\bm x) = \nabla \phi (\bm x,t) \cdot \bm n(\bm x) =0,  \quad & \ds \mathrm{on} ~ \p \om \times [0,T].
\end{array}
\end{equation}
Here the backward Riemann-Liouville fractional integral operator ${}_t\hat I_T^{1-\alpha}$ and  differential operator $\hat\p_t^{\alpha}$ are defined by \cite{Pod}
\begin{equation}\label{RiemannD}
	\hat\p_t^{\alpha}g:= - \p_t \,\big ( {}_t\hat I_T^{1-\alpha}g), \quad {}_t\hat I_T^{1-\alpha} := \frac{1}{\Gamma(1-\alpha)}\int_t^T\frac{g(s)ds}{(s-t)^{\alpha}}.
\end{equation}
\end{prop}

\begin{proof}
Let $\delta \phi$ be an admissible function that is subject to the boundary condition \eqref{phi:BC}, we use \eqref{Model:L} to compute the variation $\delta_\phi \mathcal L(\rho,\bm m,\phi)$ with respect to $\phi$
\begin{equation*}\label{Lphi:e1}\begin{array}{rl}
	\ds J_\phi(\theta) & \hspace{-0.1in}\ds := \mathcal L(\rho, \bm m, \phi + \theta \delta \phi) \\[0.1in]
	&  \hspace{-0.06in} \ds \ds = \int_0^T \int_\om \f{| \bm m |^2}{2\rho}  +  F(\rho(\bm x, t)) d\bm x dt + \int_0^T  \int_\om (\phi + \theta \delta \phi) \big (\p_t^{\alpha} \rho + \nabla \cdot {\bm  m} \big ) d\bm x dt.
\end{array}\end{equation*}
Differentiate $J_\phi(\theta)$ with respect to $\theta$ and then set $\theta = 0$ to obtain
\begin{equation*}\label{Lphi:e2}\begin{array}{l}
\ds	\delta_\phi \mathcal L(\rho,\bm m,\phi) = J_\phi'(0)  = \int_0^T \int_\om \delta \phi \big (\p_t^{\alpha} \rho + \nabla \cdot {\bm  m} \big ) d\bm x dt= 0.
\end{array}\end{equation*}
Since $\delta \phi$ is arbitrary, we obtain the time-fractional PDE for   the density $\rho$
\begin{equation}\label{Lphi:e3}
\ds \p_t^{\alpha} \rho + \nabla \cdot {\bm  m} = 0 \quad \mathrm{in} ~\Omega \times (0,T].
\end{equation}

To compute the variation $\delta_\rho \mathcal L(\rho,\bm m,\phi)$ with respect to $\rho$, we use the forward differential operator $\p_t^\alpha$ in \eqref{frac} and the backward differential operator $\hat\p_t^{\alpha}$ in \eqref{RiemannD} to integrate the $\phi\, \p_t^{\alpha} \rho$ term in the last integral on the right-hand side of \eqref{Model:L} to find
\begin{equation}\label{Lrho:e1}\begin{array}{cl}
	\ds &\ds  \int_\om \int_0^T \phi(\bm x,t) \, \p_t^{\alpha} \rho(\bm x,t) dt d\bm x \\[0.1in]
	\ds&\ds = \f1{\Gamma(1-\alpha)} \int_\om \bigg [\int_0^T \phi(\bm x,t) \int_0^t \f{\p_s \rho(\bm x,s)ds}{(t-s)^{\alpha}} dt \bigg ] d\bm x \\ [0.2in]
	\ds &\ds  = \f1{\Gamma(1-\alpha)} \int_\om \int_0^T \bigg [ \int_s^T
	\f{\phi(\bm x,t) \p_s \rho(\bm x,s)}{(t-s)^{\alpha}} dt \bigg] ds d\bm x\\[0.15in]
	\ds &\ds  = \int_\om \int_0^T \bigg [\p_s \rho(\bm x,s) \f1{\Gamma(1-\alpha)}\int_s^T \f{\phi(\bm x,t) }{(t-s)^{\alpha}} dt\bigg] ds d\bm x \\[0.15in]
	\ds &\ds  = \int_\om \int_0^T \rho(\bm x,s) \hat \p_s^\alpha \phi(\bm x,\cdot) ds d\bm x
	- \int_\om \rho_0(\bm x){}_0\hat I_T^{1-\alpha} \phi(\bm x,\cdot) d\bm x.
\end{array}\end{equation}

Integrate the remaining terms in the last integral of \eqref{Model:L} by parts and use \eqref{Lrho:e1} and the space-time boundary conditions in \eqref{Model:set} to find
\begin{equation}\label{Lrho:e3}\begin{array}{l}
	\ds \int_0^T  \int_\om \phi \big ( \p_t^{\alpha} \rho  + \nabla \cdot \bm m \big ) d\bm x dt = \int_0^T \int_\om  \rho \hat \p_t^\alpha \phi  - \bm m \cdot \nabla \phi \,d\bm x dt \\[0.125in]
\ds \qquad \qquad \qquad \qquad \qquad \qquad \qquad	- \int_\om \rho_0(\bm x){}_0\hat I_T^{1-\alpha} \phi(\bm x,\cdot) d\bm x.
\end{array}\end{equation}
Let $\delta \rho$ be any admissible function with the homogeneous initial condition
$\delta \rho(\bm x,0) = 0$ on $\om$. 
Combine \eqref{Model:L} and \eqref{Lrho:e3} and utilize  the above initial condition to obtain
\begin{equation*}\begin{array}{rl}
\hspace{-0.13in}	J_\rho(\theta) &\hspace{-0.13in} \ds := \mathcal L(\rho + \theta \delta \rho, \bm m, \phi) \\[0.1in]
	\hspace{-0.13in}& \hspace{-0.1in} \ds = \int_0^T \int_\om \f{| \bm m |^2}{2(\rho + \theta \delta \rho)} \! +\!  F(\rho  + \theta \delta \rho)d\bm x dt  + \int_0^T  \int_\om \phi \big ( \p_t^{\alpha} \rho  + \nabla \cdot \bm m \big ) \Big |_{\rho = \rho + \theta \delta \rho} d\bm x dt \\[0.15in]
\hspace{-0.13in}	& \ds \hspace{-0.1in} =  \int_0^T \int_\om \f{| \bm m |^2}{2(\rho + \theta \delta \rho)} \!+\!  F(\rho  + \theta \delta \rho) d\bm x dt  + \int_0^T \int_\om \Big [ (\rho +\theta \delta \rho)  \hat \p_t^\alpha  \phi  - \bm m \cdot \nabla \phi \Big ] d\bm x dt\\[0.15in]
\hspace{-0.13in}	& \ds - \int_\om \rho_0(\bm x) {}_0\hat I_T^{1-\alpha} \phi(\bm x,\cdot) d\bm x.
\end{array}\end{equation*}
Differentiate $J_\rho(\theta)$ with respect to $\theta$ yields
\begin{equation*}\begin{array}{rl}
	J_\rho'(\theta) &\hspace{-0.1in} \ds = \int_0^T \int_\om \delta \rho \bigg[- \f{| \bm m |^2 }{2(\rho + \theta \delta \rho)^2} +  \mathcal \p_\rho F \big|_{\rho = \rho + \theta \delta \rho } + \hat \p_t^\alpha  \phi \bigg]d\bm x dt.
\end{array}\end{equation*}
Then set $\theta = 0$ gives rise to \vspace{-0.05in}
\begin{equation*}\label{Lrho:e6}\begin{array}{rl}
	\ds \delta_\rho \mathcal L(\rho,\bm m,\phi) & \hspace{-0.1in} \ds = J_\rho'(0) = \int_0^T \int_\om \delta \rho \Big [ \hat \p_t^\alpha \phi -\f{| \bm m |^2}{2\rho^2} + \mathcal \p_\rho F  \Big ] d\bm x dt = 0.
\end{array}\end{equation*}
The arbitrary $\delta \rho$ yields the backward time-fractional Hamilton-Jacobi PDE for $\phi$\vspace{-0.02in}
\begin{equation}\label{Lrho:e7}
\hat \p_t^\alpha \phi -\f{| \bm m |^2}{2\rho^2} + \p_\rho  F =0, ~~~\mathrm{in} ~ \om \times [0,T].
\end{equation}

Finally, let $\delta \bm m$ be any admissible function satisfying the no-flow boundary condition in \eqref{Model:set}. Use \eqref{Model:L}, \eqref{Lrho:e3}, and the boundary condition \eqref{Model:set} for $\delta \bm m$ to deduce
\begin{equation*}\begin{array}{rl}
	J_{\bm m}(\theta) &\hspace{-0.1in} \ds := \mathcal L(\rho, \bm m + \theta \delta \bm m, \phi) \\[0.1in]
	& \hspace{-0.05in} \ds = \int_0^T \int_\om \f{| \bm m + \theta \delta \bm m|^2}{2\rho} +   F(\rho) d\bm x dt  + \int_0^T \int_\om \phi  \p_t^\alpha \rho  - (\bm m + \theta \delta \bm m) \cdot \nabla \phi  d\bm x dt.
\end{array}\end{equation*}
Differentiate $J_{\bm m}(\theta)$ with respect to $\theta$ and then set $\theta = 0$ to obtain
\begin{equation}\label{Lm:e2}
\delta_{\bm m} \mathcal L(\rho,\bm m,\phi) = J_{\bm m}'(0)
= \int_0^T \int_\om \f{{\bm m} \cdot \delta {\bm m}}{\rho} - \delta {\bm m} \cdot \nabla \phi d\bm x dt= 0.
\end{equation}
Since $\delta {\bm m}$ is arbitrary, we use  \eqref{Lm:e2} to derive
\begin{equation}\label{Lm:e3}
\f{\bm m}{\rho} = \nabla \phi,
\end{equation}
which completes the proof.
\end{proof}

\section{Discretization scheme}\label{Disc:MFP}
We develop  a numerical approximation to the  generalized Lagrangian \eqref{Model:L}, based on which  we further derive a discrete system  approximating the coupled  nonlinear system \eqref{Model:KKT}  on  the first-order optimality condition.

\subsection{Discretization both in space and time}
We consider the case of $\Omega \subseteq \mathbb R^2$ and $\Omega$ is square for the sake of simplicity.   Let $\Omega \times [0, T]=[0,1] ^3$ and  $\mathbf{m}:=\left(m_1, m_2\right)$. The  boundary conditions for $\bm m$ and $\phi$ in \eqref{Model:set} and \eqref{phi:BC}  imply that
\begin{equation}\label{2d:bd}
\hspace{-0.125in} \ds m_1(x, y, t)|_{x=0, 1}\! =\! m_2(x, y, t)|_{y=0, 1} =0, \, \p_x \phi(x, y, t)|_{x=0, 1} \!=\! \p_y \phi(x, y, t)|_{y=0, 1} =0.
\end{equation}
Let $\Delta x : = 1/N_x$, $\Delta y : = 1/N_y$,  and $\Delta t : ={1}/{N_t}$ be the sizes of spatial grids and time step, with $N_x$, $N_y$, and $N_t$ being positive integers. Let $\Delta V : =\Delta x \Delta y \Delta t$ for notational simplicity.
We define a spatial and temporal partition $x_i=(i-\frac{1}{2}) \Delta x$ for $1 \le i \le N_x$, $y_j=(j-\frac{1}{2}) \Delta y$ for $ 1 \le j \le N_y$, and $t_n= n \Delta t$ for $0 \le n \le N_t$. We further denote $x_{i+\f{1}{2}} = i \Delta x $ for $0 \le i \le N_x$ and $y_{j+\f{1}{2}}=  j \Delta y$ for $0 \le j \le N_y$.

We define  $\mathcal{G}_\rho$, $\mathcal{G}_{m_1}$, $\mathcal{G}_{m_2}$, and $\mathcal{G}_\phi$ as the sets
of grid point indices with $\mathcal{G}_{m_1}$ and $\mathcal{G}_{m_2}$ on  $x$- and $y$- staggered grids, respectively.
Here
$\mathcal{G}_\rho:=\{(i, j, n): \, 1 \le i \le N_x, \, 1 \le j \le N_y$, $ 0 \le n \le N_t\}$, $\mathcal{G}_{m_1}:=\{(i-\frac{1}{2}, j, n):  ~1 \le i \le N_x+1, ~1 \le j \le N_y$, $ ~ 1 \le n \le N_t\}$, $\mathcal{G}_{m_2}:=\{(i, j-\frac{1}{2}, n): ~1 \le i \le N_x, ~ 1 \le j \le N_y+1, ~ 1 \le n \le N_t\}$, and $\mathcal{G}_{\phi}:=\{(i, j, n): ~ 1 \le i \le N_x, ~ 1 \le j \le N_y, ~ 1 \le n \le N_t\}$.
Let $\rho_{i, j, n} := \rho(x_i, y_j, t_n)$, $[m_1]_{i-\f{1}{2},j,n} : = m_1(x_{i-\f{1}{2}},y_j,t_n)$, $[m_2]_{i,j-\f{1}{2},n}:=m_2(x_{i},y_{j-\f{1}{2}},t_n)$, and $\phi_{i, j, n} := \phi(x_i, y_j, t_n)$. Then we approximate  $\rho$, $m_1$,  $m_2$  and  $\phi$ by $P_{\mathcal{G}_\rho}, M_{\mathcal{G}_{m_1}}^x$,  $M_{\mathcal{G}_{m_2}}^y$ and $\Phi_{\mathcal{G}_\phi}$, respectively.
Figure \ref{mesh} illustrates the staggered grid and the corresponding spatial grid points for $\bm P_n$, $\bm M^x_n$, $\bm M^y_n$,  and $\bm \Phi_n$ defined below \eqref{Array:B} with $N_x =5$ and  $N_y=4$.
\begin{figure}[!htbp]
\setlength{\abovecaptionskip}{0pt}
\vspace{-0.15in}
\centering
\includegraphics[scale=0.3]{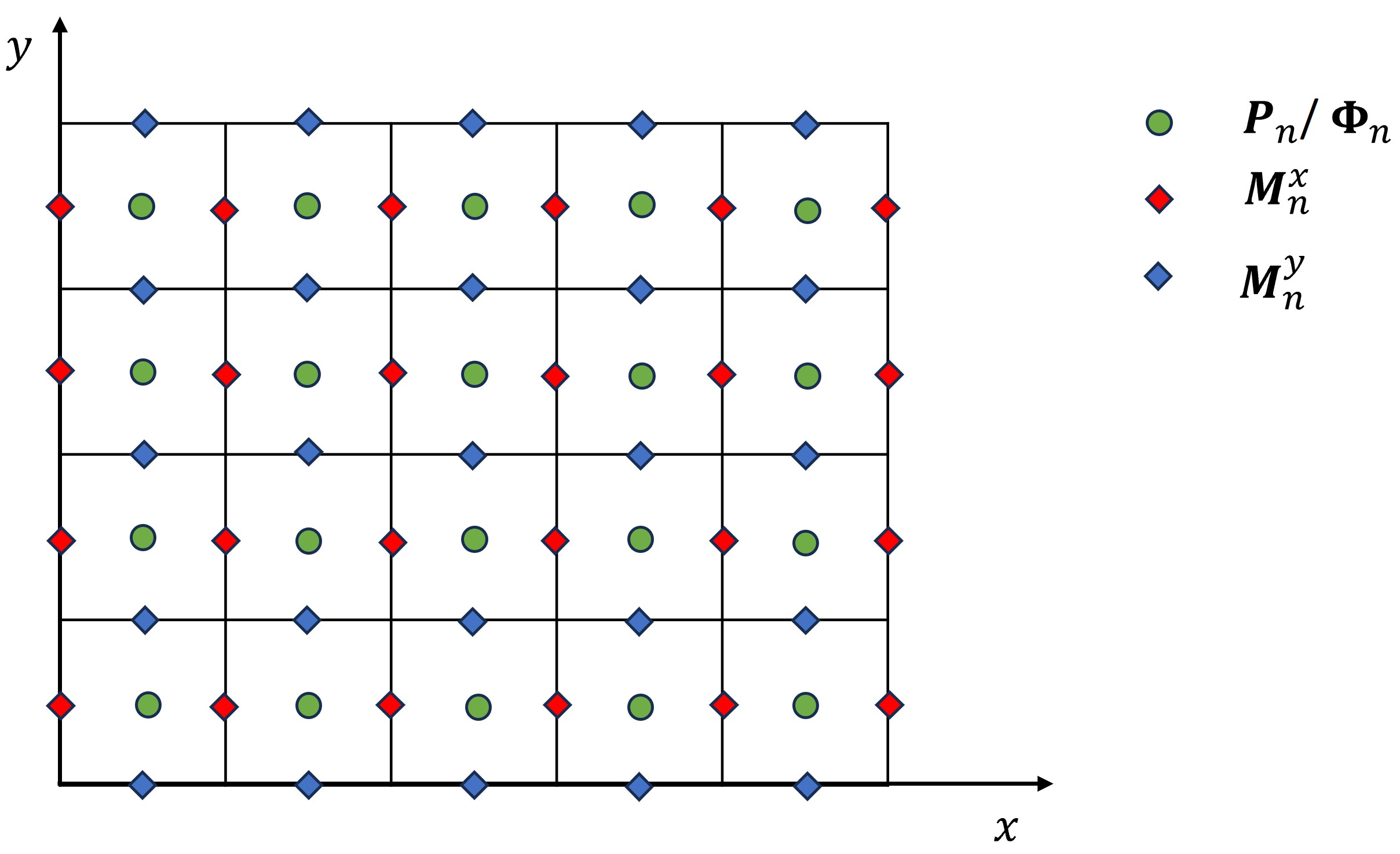}
\caption{Illustration of spatial grid points for $\bm P_n$, $\bm M^x_n$,  $\bm M^y_n$, and $\bm \Phi_n$.  }
\label{mesh}
\end{figure}
\paragraph{Discretization of the forward time-fractional transport PDE in \eqref{Model:KKT}} \vspace{-0.1in}
We employ  the $L1$ discretization to approximate $\p_t^{\alpha} \rho(x_i, y_j,t_n)$ for $1 \le i \le N_x$, $1 \le j \le N_y$, and $1 \le n \le N_t$ \cite{LinXu,LiuShen,SunWu}. We have \vspace{-0.05in}
\begin{equation}\label{Temp:e1}\begin{array}{rl}
	\hspace{-0.15in}\ds \p_t \rho(x_i, y_j,t_n) & \hspace{-0.1in} \ds \approx \delta_{t} \rho_{i,j,n}:= \frac{\rho_{i,j,n}
- \rho_{i,j,n-1}}{\Delta t} , \\[0.1in]
	\hspace{-0.15in}\ds \p_t^{\alpha} \rho(x_i, y_j,t_n) &\hspace{-0.1in} \ds =  \sum_{k=1}^n \int_{t_{k-1}}^{t_k}
\f{\p_s \rho(x_i, y_j, s)ds}{\Gamma(1-\alpha)(t_n-s)^{\alpha}}
\approx \sum_{k=1}^n \int_{t_{k-1}}^{t_k} \f{\delta_{t}\rho_{i,j,k}~ds}{\Gamma(1-\alpha)(t_n-s)^{\alpha}}  \\[0.15in]
	&\hspace{-0.1in} \ds = \sum_{k=1}^n  b_{n,k} \big [\rho_{i,j,k}-\rho_{i,j,k-1} \big ] \\[-0.05in]
	&\hspace{-0.1in} \ds  = b_{n,n}\rho_{i,j,n} + \sum_{k=1}^{n-1}(b_{n,k}-b_{n,k+1})\rho_{i,j,k} -b_{n,1}\rho_{i,j,0} =:  \delta_{t}^{\alpha} \rho_{i,j,n}, \\[0.17in]
	b_{n,k} &\hspace{-0.1in} \ds := \f{(t_n - t_{k-1})^{1-\alpha} - (t_n - t_k)^{1-\alpha}}{\Gamma(2-\alpha) \Delta t}, \quad 1 \le k \le n.
\end{array}\end{equation}

We similarly use the finite difference scheme to discretize $\p_x m_1(x_{i}, y_{j}, t_n)$ and $\p_y m_2(x_{i}, y_{j}, t_n )$  for $ 1 \le i \le N_x$, $1 \le j \le N_y$ and $1 \le n \le N_t$ \vspace{-0.05in}
\begin{equation}\label{Space:e1}\begin{array}{rl}
\hspace{-0.1in}\ds \p_x  m_1(x_{i}, y_{j}, t_n)  \approx  \delta_{x} [m_1]_{i,j,n} : = \f{  [m_1]_{i+\f{1}{2},j,n} - [m_1]_{i-\f{1}{2},j,n}}{\Delta x},\\[0.075in]
\hspace{-0.1in} \ds \p_y m_2(x_{i}, y_{j}, t_n)  \approx  \delta_{y} [m_2]_{i,j,n} := \f{  [m_2]_{{i},{j+\f{1}{2}}, n } - [m_2]_{{i},j-\f{1}{2}, n }}{\Delta y}.
\end{array}\end{equation}

Substitute   $\delta_{t}^{\alpha} \rho_{i, j, n}$, $\delta_x [m_1]_{i,j, n}$ and $\delta_y [m_2]_{i,j, n}$  for $\p_t^{\alpha} \rho(x_i, y_j, t_n)$, $\p_x m_1(x_i, y_j, t_n)$ and $\p_y m_2(x_i, y_j, t_n)$ in \eqref{Model:KKT} to arrive at a numerical approximation to the time-fractional transport equation in \eqref{Model:KKT}
\begin{equation}\label{Model:Dis}\begin{array}{l}
\ds  \delta_{t}^{\alpha} P_{i,j, n} + \delta_x M^x_{i,j, n} +  \delta_y M^y_{i,j, n} =0
\end{array}\end{equation}
for $  1 \le i \le N_x$, $ 1 \le j \le N_y$, and  $n= 1, 2, \cdots,  N_t$. We are now in the position to reformulate the derived discretization scheme \eqref{Model:Dis} into the matrix-vector multiplication form for future use.
Let $\bm P$ be a $(N_t+1)$-dimensional vector  and $\bm M^x$, $\bm M^y$, and  $\bm \Phi$ be $N_t$-dimensional vectors expressed in block form as follows:
\begin{equation}\label{Array:B}\begin{array}{l}
\bm P = \big[\bm P_0, \bm P_1, \cdots, \bm P_{N_t}]^\top, \quad \bm M^x=\big[ \bm  M^{x}_1,  \bm  M^{x}_2, \cdots, \bm  M^{x}_{N_t} \big]^\top,
\\[0.05in]
 \bm M^y=\big[ \bm  M^{y}_1, \bm  M^{y}_2, \cdots, \bm  M^{y}_{N_t} \big]^\top, \quad
\bm \Phi = \big[\bm  \Phi_1, \bm \Phi_2, \cdots,\bm  \Phi_{N_t}]^\top,
\end{array}
\end{equation}
where $\bm P_n$ and $\bm M^{x}_n$ are the $N_xN_y$-dimensional and $(N_x-1)N_y$-dimensional vectors, respectively,
\begin{equation*}\label{Array:S}{\footnotesize
 \begin{array}{l}
  \bm P_n := \big[P_{1,1,n}, \cdots, P_{N_x,1,n}, P_{1,2,n}, \cdots, P_{N_x,2,n}, \cdots, P_{1,N_y,n},\cdots, P_{N_x, N_y,n}\big], \,\, 0 \le n  \le N_t, \\[0.1in]
  \bm M^x_n := \big[M^x_{\f{3}{2},1,n}, \cdots, M^x_{N_x-\f{1}{2},1,n}, \cdots, M^x_{\f{3}{2},N_y,n},\cdots, M^x_{N_x-\f{1}{2}, N_y,n}\big], \quad 1 \le n  \le N_t,
 \end{array}}
 \end{equation*}
and  $\bm M^{y}_n$, and $\bm \Phi_{n}$  for $1 \le n  \le N_t$ could be defined accordingly.
Employ \eqref{Array:B} to define the $\tilde N $-dimensional vector $\bm U :=[\bm P^\top, {\bm M^x}^\top, {\bm M^y}^\top]^\top$ with
$\tilde N := N_x N_y (N_t +1)  + (N_x-1)N_y N_t + N_x (N_y-1) N_t,$
then the finite difference scheme \eqref{Model:Dis} with $\delta_t^\alpha$, $\delta_ x$, and $\delta_y$ defined in \eqref{Temp:e1}--\eqref{Space:e1} could be expressed as $\bm K  \bm U = \bm 0$ with the $(N_x N_y N_t)$-by-$\tilde N$ matrix $\bm K$ defined by
\begin{equation}\begin{array}{l}\label{Model:M}
\hspace{-0.175in}  \ds \bm K = \big[ \bm A \otimes \bm I_{N_x,N_x} \otimes \bm I_{N_y,N_y} ,\bm I_{N_t \times N_t} \otimes \bm I_{N_y \times N_y} \otimes \bm C^x , \bm I_{N_t \times N_t} \otimes \bm C^y  \otimes  \bm I_{N_x \times N_x}\big],
  \end{array}
\end{equation}
 where   $\otimes$ represents Kronecker product, $\bm I_{n,n}$ refers to the $n$-by-$n$ identity matrix, and $\bm A := \{a_{ij}\}_{i,j=1}^{N_t,N_t+1}$, $\bm C^x := \{c^x_{ij}\}_{i,j=1}^{N_x,N_x-1}$ are defined as follows
\begin{equation}\begin{array}{l}\label{M_A}
a_{ij}=\left\{
\begin{array}{ll}
-b_{i,1}, & j=1,\\[0.025in]
b_{i,j}-b_{i,j+1}, & 2 \le j \le i,\\[0.025in]
b_{i,i}, &  j = i+1,\\[0.025in]
0,& otherwise,
\end{array}
\right.
\end{array}
c_{ij}^x=\left\{
\begin{array}{ll}
\frac{1}{\Delta x}, & j=i,\\[0.025in]
-\frac{1}{\Delta x}, &  j=i-1 ,\\[0.025in]
0,& otherwise,
\end{array}
\right.
\end{equation}
and $\bm C^y = \{c^y_{ij}\}_{i,j=1}^{N_y,N_y-1}$ could be defined in an analogous manner as $\bm C^x$ in \eqref{M_A} with $\Delta x$ replaced by $\Delta y$.

Motivated by the integration  by parts in \eqref{Lrho:e1}--\eqref{Lrho:e3}, i.e., \vspace{-0.05in}
\begin{equation}\label{Int_adj}\begin{array}{l}
\ds \int_0^T  \int_\om \phi \p_t^{\alpha} \rho \,  d\bm x dt = \int_0^T \int_\om  \rho \hat \p_t^\alpha \phi \, d\bm x dt  	- \int_\om \rho_0(\bm x){}_0\hat I_T^{1-\alpha} \phi(\bm x,\cdot) d\bm x,\\[0.1in]
\ds \int_0^T  \int_\om \phi\, \nabla \cdot \bm m \, d\bm x dt = -\int_0^T \int_\om    \bm m \cdot \nabla \phi \,d\bm x dt,
\end{array}\end{equation}
we turn to discretize the operators $\hat\p_t^{\alpha} \phi (x_i, y_j, t_n)$ defined in \eqref{RiemannD} and $\nabla \cdot \phi(x_i, y_j, t_n)$ for $1 \le i \le N_x$, $1 \le j \le N_y$, and $1 \le n \le N_t$.
 We will prove in Proposition \ref{lem:adj:opr} that the discretizations in \eqref{BFEM:e1} and \eqref{Space:e2} could be deduced by those in \eqref{Temp:e1} and \eqref{Space:e1} such that the dual properties in the continuous problem in \eqref{Int_adj} still hold on the discrete level, which will be used subsequently  to derive the discrete system in Proposition \ref{discre:KKT:MFP}.
\paragraph{Discretization of the adjoint operators}
Similarly discretize $\hat\p_t^{\alpha}\phi(x_i, y_j ,t_n)$ backward in time  for $n=N_t,\ldots,1$ \vspace{-0.05in}
\begin{equation}\label{BFEM:e1}\begin{array}{l}
	\hspace{-0.135in}\ds \p_t \phi(x_i, y_j, t_{n}) \approx \delta_{t}\phi_{i,j,n},\\[0.1in]
	\hspace{-0.135in}\ds  \hat\p_t^{\alpha}\phi(x_i, y_j, t_{n}) \!=\!- \p_t \big ( {}_{t} \hat I_T^{1-\alpha}\phi \big )(x_i, y_j, t_n)
	\!\approx \! -\frac{1}{\Delta t} \bigg [{}_{t_n}\hat I_T^{1-\alpha}\phi- {}_{t_{n-1}}\hat I_T^{1-\alpha} \phi  \bigg ] (x_i, y_j) \\[0.125in]
	\hspace{-0.135in}\ds \qquad=\frac{1}{\Delta t}  \frac{1}{\Gamma(1-\alpha)}\left [ \sum_{k=n}^{N_t}\int_{t_{k-1}}^{t_k}\frac{\phi(x_i, y_j ,s)ds}{(s-t_{n-1})^{\alpha}}-\sum_{k=n}^{N_t-1}\int_{t_{k}}^{t_{k+1}}\frac{\phi(x_i, y_j, s) ds}{(s-t_{n})^{\alpha}} \right ]\\[0.175in]
	\hspace{-0.135in}\ds \qquad \approx \frac{1}{\Delta t}  \frac{1}{\Gamma(1-\alpha)}\left [\sum_{k=n}^{N_t}\int_{t_{k-1}}^{t_k}\frac{\phi(x_i, y_j ,t_{k})ds}{(s-t_{n-1})^{\alpha}}-\sum_{k=n}^{N_t-1}\int_{t_{k}}^{t_{k+1}}\frac{\phi(x_i, y_j,t_{k+1}) ds}{(s-t_{n})^{\alpha}} \right ]\\[0.15in]
	\hspace{-0.135in}\ds \qquad = b_{n,n} \phi_{i,j,n}+ \sum_{k=n+1}^{N_t}(b_{k,n}-b_{k,n+1}) \phi_{i,j,k}
	=: \! \hat \delta^{\alpha}_t \phi_{i,j,n},\\[0.18in]
	\hspace{-0.135in}\ds b_{k,n} := \f{(t_k - t_{n-1})^{1-\alpha} - (t_{k-1} - t_{n-1})^{1-\alpha}}{\Gamma(2-\alpha) \Delta t}, \quad n \le k \le N_t.
\end{array}\end{equation}
Here $\delta_{t} \phi_{i,j,n}$ is defined in \eqref{Temp:e1}.

We finally discrete $\p_x \phi(x_{i}, y_j, t_n )$ and $\p_y \phi(x_i, y_{j}, t_n )$ for  $1 \le n \le N_t$ \vspace{-0.08in}
\begin{equation}\label{Space:e2}\begin{array}{rl}
\hspace{-0.2in} \ds \p_x  \phi(x_{i}, y_j, t_n )  \approx  \hat \delta_{x} \phi _{i,j,n} := \f{ \phi_{i+1, j,n} - \phi_{i, j,n} }{\Delta x}, \; 1 \le i \le N_x-1, \, 1 \le j \le N_y,\\[0.1in]
\ds \p_y  \phi(x_i, y_j, t_n )  \approx  \hat \delta_{y} \phi _{i,j,n} := \f{ \phi_{i,j+1,n} - \phi_{i,j,n}}{\Delta y}, \; 1 \le i \le N_x, \, 1 \le j \le N_y-1.
\end{array}\end{equation}
\subsection{Discretization of the generalized Lagrangian \eqref{Model:L}}
With the notations and discretization schemes \eqref{Temp:e1}--\eqref{Space:e1} in hand, we are now in the position to discretize the generalized Lagrangian  $\mathcal L(\rho,\bm m,\phi)$ in \eqref{Model:L}. Since OT problem \eqref{Model:OT} is a simplified version of the MFP problem with $ F =0$ in \eqref{Model:MFP} and \eqref{Model:L}, we mainly focus on the derivations of the numerical discretization scheme to the MFP problem for illustration.  The numerical approximation to MFP  reads as follows: find $\bm P$, $\bm M^x$, $\bm M^y$, and $\bm \Phi$ defined in \eqref{Array:B}, such that
\begin{equation}\label{sad_DL}
  \inf_{\bm P, \bm M^x, \bm M^y } \sup_{\bm \Phi}~ \breve{ \mathcal L}(\bm P, \bm M^x, \bm M^y, \bm \Phi),
\end{equation}
where $\breve{ \mathcal L}$ is the numerical approximation to $\mathcal L(\rho,\bm m,\phi)$ in \eqref{Model:L} defined by
\begin{equation}\label{Model:DL}\begin{array}{l}
\hspace{-0.2in}\ds	\breve{\mathcal L}\ds := \Delta V \sum_{n=1}^{N_t} \sum_{i=1}^{N_x}\sum_{j=1}^{N_y} \bigg[{\breve L}_{i,j,n} \!+\!  F(P_{i, j, n})\!+\! \Phi_{i, j, n} \big ( \delta_{t}^{\alpha} P_{i,j, n}\! +\! \delta_x M^x_{i,j, n} \!+\! \delta_y M^y_{i,j, n} \big)\bigg]
\end{array}\end{equation}
with \vspace{-0.15in}
\begin{equation}\label{D:F}
  \ds {\breve L}_{i,j,n} : = \f{ \big(M_{i+\f{1}{2}, j, n}^{x}\big)^2}{P_{i,j, n} + P_{i+1, j, n}} + \f{\big( M_{i, j+\f{1}{2}, n}^{y}\big)^2}{P_{i,j, n} + P_{i, j+1, n}}.
\end{equation}

Similar to its continuous analogue, we employ the optimality condition of the problem \eqref{sad_DL}--\eqref{Model:DL} to compute the variation of  $\breve {\mathcal L}(\bm P, \bm M^x, \bm M^y, \bm \Phi)$ with respect to  all its arguments. To aim at this goal, we first prove that the adjoint properties of the continuous problem \eqref{Int_adj} is preserved on the discrete level in the following proposition.

\begin{prop}\label{lem:adj:opr}
For $ 1 \le i \le N_x$, $1 \le j \le N_y$ and $1 \le n \le N_t$,  $\delta^{\alpha}_t$ in \eqref{Temp:e1} and  $\hat \delta^{\alpha}_t$ in \eqref{BFEM:e1} satisfy
\begin{equation}\label{leftright:t}\begin{array}{l}
\hspace{-0.2in}\ds \sum_{n=1}^{N_t} \delta^{\alpha}_t P_{i,j,n}\,\Phi_{i,j,n} + P_{i,j,0}  \sum_{n=1}^{N_t}  b_{n,1}\,\Phi_{i,j,n} = \sum_{n=1}^{N_t}  P_{i,j,n} \,\hat\delta^{\alpha}_t \Phi_{i,j,n}.
\end{array}
\end{equation}
Similarly,  $\delta_x $ and $\delta_y $ in \eqref{Space:e1}, and  $\hat \delta_x $ and $\hat \delta_y $ in \eqref{Space:e2} satisfy
\begin{equation}\label{leftright:x}\begin{array}{l}
\hspace{-0.15in} \ds \sum_{i=1}^{N_x} \sum_{j=1}^{N_y} \delta_x M^x_{i,j, n} \Phi_{i,j,n} = -\sum_{i=1}^{N_x} \sum_{j=1}^{N_y}  M^x_{i+\f{1}{2},j, n} \hat \delta_{x} \Phi_{i,j,n},\\[0.1in]
\hspace{-0.15in} \ds \sum_{i=1}^{N_x} \sum_{j=1}^{N_y} \delta_y M^y_{i,j, n} \Phi_{i,j,n}  = -\sum_{i=1}^{N_x} \sum_{j=1}^{N_y}  M^y_{i,j+\f{1}{2}, n} \hat \delta_{y} \Phi_{i,j,n}.
\end{array}  \end{equation}
\end{prop}
\begin{proof}
Use   \eqref{Temp:e1} for $\delta^{\alpha}_t$ and  \eqref{BFEM:e1} for $\hat \delta^{\alpha}_t$ to obtain
\begin{equation}\label{leftright:e1}\begin{array}{cl}
&\hspace{-0.1in} \ds \sum_{n=1}^{N_t}  \delta^{\alpha}_t P_{i,j,n}\,\Phi_{i,j,n} + P_{i,j,0} \sum_{n=1}^{N_t}  b_{n,1} \,\Phi_{i,j,n} \\[0.1in]
&\hspace{-0.1in} \ds  =\sum_{n=1}^{N_t}  b_{n,n}P_{i,j,n}\,\Phi_{i,j,n}+\sum_{n=1}^{N_t}  \sum_{k=1}^{n-1}(b_{n,k}-b_{n,k+1})P_{i,j,k}\,\Phi_{i,j,n} \\[0.15in]
&\hspace{-0.1in} \ds =\sum_{n=1}^{N_t}  P_{i,j,n}\, b_{n,n}\Phi_{i,j,n}+\sum_{k=1}^{{N_t} -1} P_{i,j,k} \sum_{n=k+1}^{{N_t} }(b_{n,k}-b_{n,k+1})\Phi_{i,j,n}  \\[0.15in]
&\hspace{-0.1in} \ds =\sum_{n=1}^{N_t}  P_{i,j,n} \Big( b_{n,n}\Phi_{i,j,n} +\sum_{k=n+1}^{N_t} (b_{k,n}-b_{k,n+1})\Phi_{i,j,k}  \Big)  = \sum_{n=1}^{N_t}  P_{i,j,n} \,\hat\delta^{\alpha}_t \Phi_{i,j,n},
\end{array}  \end{equation}
which proves the first statement. We note that \eqref{leftright:t} multiplied by $\Delta t$ on both sides of the equality recovers the adjoint property in \eqref{Int_adj} by \eqref{Temp:e1}, \eqref{BFEM:e1}  and the proper discretization of ${}_0\hat I_T^{1-\alpha} \phi(x_i, y_j, t_n)$  in \eqref{Int_adj} as
\begin{equation}\begin{array}{cl}\label{adj:dis}
  \ds {}_0\hat I_T^{1-\alpha} \phi(x_i, y_j, t_n) & \hspace{-0.08in}\ds = \frac{1}{\Gamma(1-\alpha)} \sum_{n=1}^{N_t}\int_{t_{n-1}}^{t_n}\frac{\phi(x_i, y_j ,s)ds}{s^{\alpha}} \\[0.125in]
  \ds&\hspace{-0.08in}\ds  \approx   \frac{1}{\Gamma(1-\alpha)} \sum_{n=1}^{N_t}\int_{t_{n-1}}^{t_n}\frac{\phi(x_i, y_j ,t_n)ds}{s^{\alpha}} = \Delta t \sum_{n=1}^{N_t} b_{n,1} \Phi_{i,j,n}.
  \end{array}
\end{equation}

Similarly, we use  $\delta_x $ in \eqref{Space:e1} and  $\hat \delta_x $  in \eqref{Space:e2} to further derive
\begin{equation}\label{leftright:e2}\begin{array}{cl}
 & \ds \ds \sum_{i=1}^{N_x} \sum_{j=1}^{N_y} \delta_x M^x_{i,j, n} \Phi_{i,j,n}  =\sum_{i=1}^{N_x} \sum_{j=1}^{N_y} \bigg[\f{ M^x_{i+\f{1}{2},j, n}- M^x_{i-\f{1}{2},j, n}}{\Delta x} \bigg]\Phi_{i,j,n}  \\[0.15in]
 & \ds\ds = \sum_{i=1}^{N_x-1} \sum_{j=1}^{N_y}  M^x_{i+\f{1}{2},j, n} \bigg[\f{ \Phi_{i,j, n}- \Phi_{i+1,j, n}}{\Delta x}\bigg] = -\sum_{i=1}^{N_x-1} \sum_{j=1}^{N_y}  M^x_{i+\f{1}{2},j, n} \hat \delta_{x} \Phi_{i,j,n}\\[0.1in]
 & \ds \ds  = -\sum_{i=1}^{N_x} \sum_{j=1}^{N_y}  M^x_{i+\f{1}{2},j, n} \hat \delta_{x} \Phi_{i,j,n}.
\end{array}  \end{equation}
Here we have used the boundary conditions $M^x_{N_x+\f{1}{2},j, n} =0$ and $\hat \delta_{x} \Phi_{N_x,j,n} =0$ which approximates the no-flux boundary  condition in \eqref{2d:bd}.
The second relation in \eqref{leftright:x} could be proved similarly following the procedures in  \eqref{leftright:e2} combined with the expressions \eqref{Space:e1} for $\delta_y $ and \eqref{Space:e2} for  $\hat \delta_y $ and is thus omitted. \eqref{leftright:x} also preserves the adjoint property in \eqref{Int_adj} for the continuous problem.
\end{proof}

\begin{prop}\label{discre:KKT:MFP}
The optimality conditions for this discrete saddle-point system \eqref{sad_DL}--\eqref{Model:DL} read as follows: for $1 \le i \le N_x$, $1 \le j \le N_y$, and $1 \le n \le N_t$
\begin{equation}\begin{array}{l}\label{dis:KKT}
\ds \delta_{t}^{\alpha} P_{i,j, n} + \delta_x M^x_{i,j, n} +\delta_y M^y_{i,j, n} = 0, \\[0.1in]
\ds  \hat\delta_t^{\alpha} \Phi_{i, j, n}+ \p_{P_{i, j, n}} F(P_{i, j, n})-H_{i,j,n} =0,\\[0.1in]
\ds \f{ 2 M^x_{i+\f{1}{2}, j, n}}{P_{i,j,n} + P_{i+1,j,n}}  = \hat \delta_x \Phi_{i, j, n}, \quad \f{ 2 M^y_{i, j+\f{1}{2}, n}}{P_{i,j,n} + P_{i,j+1,n}}  = \hat \delta_y \Phi_{i, j, n}
\end{array}
\end{equation}
with
{\footnotesize\begin{equation}\begin{array}{l}\label{dis:KKT:ICBC}
\ds H_{i,j,n} :=  \f{\big( M_{i+\f{1}{2}, j, n}^{x}\big)^2}{(P_{i,j, n} + P_{i+1, j, n})^2} +  \f{ \big(M_{i-\f{1}{2}, j, n}^{x}\big)^2}{(P_{i-1,j, n} + P_{i, j, n})^2} \\[0.15in]
\ds \qquad \quad + \f{ \big(M_{i, j+\f{1}{2}, n}^{y}\big)^2}{(P_{i,j, n} + P_{i, j+1, n})^2} +  \f{ \big(M_{i, j-\f{1}{2}, n}^{y}\big)^2}{(P_{i,j-1, n} + P_{i, j, n})^2},\\[0.15in]
\ds P_{i, j, 0} = \rho_0(x_{i} , y_j), \quad P_{i,j, N_t} = \rho_1(x_{i} , y_j).
\end{array}
\end{equation}}
\end{prop}
\begin{proof}
The proof  could be carried out following that of Proposition \ref{Model:KKT} and thus we will provide only a brief outline of the proof here.

Differentiate \eqref{Model:DL} with respect to $\Phi_{i, j, n}$  to obtain
\begin{equation}\begin{array}{l}\label{dis:phi}
\ds \p_{\Phi_{i, j ,n}} \breve { \mathcal L}=  \Delta V   \big (\delta_{t}^{\alpha} P_{i,j, n} + \delta_x M^x_{i,j, n} +\delta_y M^y_{i,j, n} \big)=0,
\end{array}
\end{equation}
 which yields the first governing equation in \eqref{dis:KKT} and is the numerical approximation to the time-fractional transport PDE in \eqref{Lphi:e3}.

To  differentiate \eqref{Model:DL} with respect to $P_{i,j, n}$, we first utilize the adjoint properties of the discretization schemes \eqref{leftright:t}--\eqref{leftright:x} to reformulate \eqref{Model:DL} as
\begin{equation}\label{Model:DL:r}\begin{array}{l}
\hspace{-0.175in}\ds	\breve { \mathcal L} \ds
 \ds = \Delta V \sum_{n=1}^{N_t} \sum_{i=1}^{N_x}\sum_{j=1}^{N_y} \Big[\breve L_{i,j,n}+ F(P_{i, j, n}) + P_{i,j,n} \,\hat\delta^{\alpha}_t \Phi_{i,j,n} - b_{n,1}P_{i,j,0} \,\Phi_{i,j,n} \\[0.15in]
 \ds \qquad \qquad \qquad \qquad  -     M^x_{i+\f{1}{2},j, n} \hat \delta_x  \Phi_{i,j,n} -   M^y_{i,j+\f{1}{2}, n} \hat \delta_y \Phi_{i,j,n} \Big].
\end{array}\end{equation}
Differentiate the above expression with respect to $P_{i,j, n}$ and incorporate \eqref{D:F} and \eqref{dis:KKT:ICBC} to obtain
\begin{equation}\begin{array}{l}\label{dis:p}
\hspace{-0.15in}\ds \p_{P_{i,j, n}} \breve { \mathcal L} = \Delta V \Big[ -H_{i,j,n} + \p_{P_{i, j, n}} F(P_{i, j, n}) + \hat \delta_{t}^{\alpha} \Phi_{i, j, n}\Big]=0,
\end{array}
\end{equation}
this implies
$ \hat\delta_t^{\alpha} \Phi_{i, j, n}+ \p_{P_{i, j, n}}  F(P_{i, j, n})- H_{i,j,n}  =0$,
which might be considered as  a discretized version of the backward Hamilton-Jacobi PDE \eqref{Lrho:e7}.

Finally, differentiate \eqref{Model:DL:r} with respect to $M^x_{i+\f{1}{2}, j, n}$ and $M^y_{i, j+\f{1}{2}, j, n}$, respectively, and combine \eqref{D:F}  to obtain
\begin{equation}\begin{array}{l}\label{dis:m}
 \ds \p_{M^x_{i+\f{1}{2}, j, n}} \breve { \mathcal L} = \Delta V   \bigg[\f{ 2 M^x_{i+\f{1}{2}, j, n}}{P_{i,j,n} + P_{i+1,j,n}}  - \hat \delta_x \Phi_{i, j, n} \bigg]= 0,\\[0.15in]
 \ds \p_{M^y_{i, j+\f{1}{2}, n}} \breve { \mathcal L} = \Delta V   \bigg[\f{ 2 M^y_{i, j+\f{1}{2}, n}}{P_{i,j,n} + P_{i,j+1,n}}  - \hat \delta_y \Phi_{i, j, n} \bigg]= 0,
\end{array}
\end{equation}
thus we obtain the last two equations in \eqref{dis:KKT} which correspond to $\f{\bm m}{\rho} = \nabla \phi$ in \eqref{Lm:e3}, which  completes the proof.
\end{proof}

\section{Primal-dual formulations}\label{PDHG:MFP}
\subsection{PDHG algorithm}
We review  the primal-dual hybrid gradient (PDHG) algorithm \cite{ChaPoc1,ChaPoc2}, which is a widely used  optimization method to solve the saddle point problem \eqref{sad_DL}--\eqref{Model:DL}.
 Let $\sigma_{m}$ and  $\sigma_{\phi}$ be the step sizes.
At the $k$-th iteration, the update via PDHG algorithm contains the following steps \cite{JacLeg,LiLee}:
\begin{equation}\begin{array}{l}\label{PDHG:AL}
\hspace{-0.1in}  \ds [\bm M^{x,(k+1)}, \bm M^{y,(k+1)}]= \inf_{\bm M^x, \bm M^y}~ \breve{ \mathcal L}(\bm P^{(k)}, \bm M^x, \bm M^{y},  \bm {\bar  \Phi}^{(k)})\\[0.125in]
   \qquad \qquad  \qquad \qquad \qquad \qquad \qquad \ds + \f{\|\bm  M^x- \bm M^{x,(k)}\|_2^2}{2 \sigma_m} +  \f{\|\bm  M^y- \bm M^{y,(k)}\|_2^2}{2 \sigma_m},\\[0.125in]
  \hspace{-0.1in}  \ds \bm P^{(k+1)} = \inf_{\bm P}  ~ \breve{ \mathcal L}(\bm P, \bm M^{x, (k+1)}, \bm M^{y, (k+1)}, \bm {\bar  \Phi}^{(k)}) + \f{\|\bm  P- \bm P^{(k)}\|_2^2}{2 \sigma_m},\\[0.125in]
\hspace{-0.1in}  \ds \bm \Phi^{(k+1)}=\sup_{\bm \Phi}~ \breve{ \mathcal L}(\bm P^{(k+1)}, \bm M^{x, (k+1)}, \bm M^{y, (k+1)}, \bm  \Phi) - \f{\|\bm  \Phi- \bm \Phi^{(k)}\|_2^2}{2 \sigma_\phi},\\[0.125in]
\hspace{-0.1in}  \ds  \ds \bar{\bm  \Phi}^{(k+1)}=2 \bm \Phi^{(k+1)}-\bm  \Phi^{(k)},
\end{array}
\end{equation}
where  $\|\cdot\|_2$ represents the discrete $L^2$ norm defined by
\begin{equation}\begin{array}{l}\label{norm}
  \ds  \|\bm M^x\|_2^2 : = \Delta V  \sum_{\bm i \in \mathcal G_{{m_1}}}  \big(M^{x}_{\bm i}\big)^2, \quad \ds  \|\bm M^y\|_2^2 := \Delta V  \sum_{\bm i \in \mathcal G_{{m_2}}}  \big(M^{y}_{\bm i}\big)^2,\\[0.15in]
    \ds  \|\bm P\|_2^2 := \Delta V  \sum_{\bm i \in \mathcal G_\rho}  P_{\bm i}^2, \quad   \ds  \|\bm \Phi\|_2^2 := \Delta V  \sum_{\bm i \in \mathcal G_\phi}  \Phi_{\bm i}^2.
  \end{array}
\end{equation}
The main source of instability in the PDHG algorithm is the decoupling of the $\bm M^x$, $\bm M^y$, $\bm P$ and $\bm \Phi$ update steps. This algorithm converges if the step sizes $\sigma_{m}$ and $\sigma_{\phi}$ satisfy
$\sigma_{m}\sigma_{\phi} \big\|\bm K^\top \bm  K\big\|_{L^2} \le \sigma^* <1$
with $\bm K$ defined in \eqref{Model:M} being the preconditioner generated by the numerical approximation to the time-fractional
transport equation \eqref{Model:Dis}, and thus the step sizes of the discrete PDHG algorithm
must depend on the grid resolution.

To resolve this issue, the G-prox PDHG algorithm, a variation of the PDHG algorithm with a  preconditioner, was developed in \cite{JacLeg} to provide an appropriate choice of norms for the PDHG algorithm.
Specifically,  the G-prox PDHG modifies the step of updating $\bm \Phi^{(k+1)}$ in \eqref{PDHG:AL} as follows:
\begin{equation}\begin{array}{l}\label{GPDHG:AL}
\hspace{-0.1in}  \ds \bm \Phi^{(k+1)}=\sup_{\bm \Phi}~\breve{ \mathcal L}(\bm P^{(k+1)}, \bm M^{x, (k+1)}, \bm M^{y, (k+1)}, \bm  \Phi) - \f{\|\bm  \Phi- \bm \Phi^{(k)}\|_H^2}{2 \sigma_\phi},
\end{array}
\end{equation}
where  the norm in \eqref{GPDHG:AL} is augmented from $L^2$ to $H$ with the norm $\|\cdot\|_H$  defined as
\begin{equation}\begin{array}{l}\label{GPDHG:Precon}
  \ds \|\bm \Phi\|_H^2 : = \big\|\bm K^\top \bm \Phi\big\|_{2}^2.
\end{array}
\end{equation}
\begin{prop}\label{Gpdhg:Conv}
The step sizes of the G-prox PDHG algorithm only need to satisfy
  $\ds \sigma_{m}\sigma_{\phi} <1$.
\end{prop}
\begin{proof}
  The proof of this proposition follows from \cite[Theorem 3.1]{JacLeg} and is thus omitted.
\end{proof}

We note from  Proposition \ref{Gpdhg:Conv} that the step sizes of the G-prox PDHG algorithm are independent of the  preconditioner  $\bm K$ in \eqref{Model:M}, which allows us to choose larger step sizes independent of the grid sizes to reach faster convergence compared with the original PDHG algorithm.

\subsection{Implementations of the algorithm}
We illustrate some implementation details of the G-prox PDHG algorithm  for the discrete saddle point problem \eqref{sad_DL}--\eqref{Model:DL}.
\paragraph{Step 1: update $\bm M^{x, (k+1)}$ and $\bm M^{y, (k+1)}$}
Minimize the discrete function defined in \eqref{sad_DL} in terms of the second and the third components, respectively : Find the optimal $\bm M^{x, (k+1)}$ and $\bm M^{y, (k+1)}$ which are the solutions to
\begin{equation}\label{L:M}
  \ds \inf_{\bm M^x, \bm M^y} \breve{ \mathcal L}(\bm P^{(k)}, \bm M^x, \bm M^{y},  \bm {\bar  \Phi}^{(k)}) + \f{\|\bm  M^x- \bm M^{x,(k)}\|_2^2}{2 \sigma_m} +  \f{\|\bm  M^y- \bm M^{y,(k)}\|_2^2}{2 \sigma_m}.
\end{equation}
We note from the optimality conditions that the minimizers   $\bm M^{x, (k+1)}$ and $\bm M^{y, (k+1)}$   for \eqref{L:M} have explicit expressions, e.g., $\bm M^{x, (k+1)}$  satisfy the following expressions which are obtained by differentiating \eqref{L:M} with respect to $M^x_{i+\f{1}{2},j,n}$  combined with  $ \p_{M^x_{i+\f{1}{2}, j, n}} \breve{ \mathcal L} $ in \eqref{dis:m}  \vspace{-0.1in}
\begin{equation}\begin{array}{l}\label{update:m:e1}
  \ds \f{2 M_{i+\f{1}{2},j, n}^{x,(k+1)}}{P_{i, j, n}^{(k)} + P_{i+1,j, n}^{(k)} } -\hat \delta_x  \bar \Phi_{i,j,n}^{(k)}+ \f{1}{\sigma_m}\big[M_{i+\f{1}{2}, j, n}^{x,(k+1)} - M_{i+\f{1}{2},j, n}^{x,(k)}\big] = 0,\\[0.2in]
  \ds \qquad \qquad 1 \le i \le N_x -1, \quad 1 \le j \le N_y, \quad  n= 1, 2, \cdots, N_t,
  \end{array}
\end{equation}
by which we could update $M_{i+\f{1}{2},j, n}^{x, (k+1)}$ as follows:
\begin{equation}\begin{array}{l}\label{update:m:e2}
  \ds M_{i+\f{1}{2},j, n}^{x, (k+1)} ={\big( { M_{i+\f{1}{2}, j, n}^{x, (k)}} +\sigma_m  \hat \delta_x  \bar \Phi_{i,j, n}^{(k)} \big)}\bigg/{\bigg[\f{2\sigma_m}{P_{i, j, n}^{(k)} + P_{i+1,j, n}^{(k)}} + 1\bigg]}.
  \end{array}
\end{equation}
Similarly, we could update  $\bm M^{y,(k+1)}$ by
\begin{equation}\begin{array}{l}\label{update:m:e3}
  \ds M_{i,j+\f{1}{2}, n}^{y, (k+1)} ={\big( { M_{i, j+\f{1}{2}, n}^{y, (k)}}  + \sigma_m  \hat \delta_y  \bar \Phi_{i,j,n}^{(k)}\big)}\bigg/{\bigg[\f{2\sigma_m}{P_{i, j, n}^{(k)} + P_{i,j+1, n}^{(k)}} + 1\bigg]},\\[0.2in]
  \ds \quad 1 \le i \le N_x, \quad 1 \le j \le N_y -1, \quad n= 1, 2, \cdots,  N_t.
  \end{array}
\end{equation}
\paragraph{Step 2: update $\bm P^{(k+1)}$}
Minimize the discrete function defined in \eqref{sad_DL} in terms of the first component: Find the optimal $\bm P^{(k+1)}$ which is the solution to
\begin{equation}\label{L:rho}
  \ds \inf_{\bm P}~  \breve{ \mathcal L}(\bm P, \bm M^{x, (k+1)}, \bm M^{y, (k+1)}, \bm {\bar  \Phi}^{(k)}) + \f{\|\bm  P- \bm P^{(k)}\|_2^2}{2 \sigma_m}.
\end{equation}
We note from \eqref{dis:KKT:ICBC} that $P_{i,j,0}$ and $P_{i,j,N_t}$ are specified during the iteration process, and we update $ P_{i, j, n}^{(k+1)}$ for $1 \le i \le N_x$, $1 \le j \le N_y$ and $1 \le n \le N_t-1$  by the following linearized version of \eqref{L:rho}, which could be reached by combining \eqref{dis:p} as follows:
\begin{equation}\begin{array}{cl}\label{update:rho:e0}
  \ds P_{i,j, n}^{(k+1)} &\ds \hspace{-0.1in}=P_{i, j, n}^{(k)} - \f{\sigma_m}{ \Delta V} \p_{P_{i,j, n}^{(k)}} \breve{ \mathcal L}(\bm P^{(k)}, \bm M^{x, (k+1)}, \bm M^{y, (k+1)}, \bm {\bar  \Phi}^{(k)})\\[0.15in]
&\ds \hspace{-0.1in} = P_{i,j,n}^{(k)} + \sigma_m \big(H_{i,j,n}^{(k)} - \p_{P_{i, j, n}^{(k)}} F (P_{i, j, n}^{(k)} ) - \hat \delta_t^\alpha \bar \Phi_{i, j, n}^{(k)}\big).
\end{array}
\end{equation}
with {\footnotesize
\begin{equation}\begin{array}{l}\label{update:rho:DJ}
\hspace{-0.15in}\ds H_{i,j,n}^{(k)} : =\f{\big( M_{i+\f{1}{2}, j, n}^{x,(k)}\big)^2}{(P_{i,j, n}^{(k)} + P_{i+1, j, n}^{(k)})^2} +  \f{ \big(M_{i-\f{1}{2}, j, n}^{x,(k)}\big)^2}{(P_{i-1,j, n}^{(k)} + P_{i, j, n}^{(k)})^2}  \\[0.15in]
\ds \qquad  + \f{ \big(M_{i, j+\f{1}{2}, n}^{y,(k)}\big)^2}{(P_{i,j, n}^{(k)} + P_{i, j+1, n}^{(k)})^2} +  \f{ \big(M_{i, j-\f{1}{2}, n}^{y,(k)}\big)^2}{(P_{i,j-1, n}^{(k)} + P_{i, j, n}^{(k)})^2}.
\end{array}
\end{equation}}

\begin{algorithm}[h]
\caption{G-prox PDHG algorithm for the fractional MFP/OT problem \eqref{Model:MFP}.}
\label{algorithm}
Input: $\rho_0$, $\rho_1$, $\alpha$, $N_x$, $N_y$, $N_t$, $\sigma_m$, $\sigma_\phi$, $\lambda_E$, $\lambda_Q$, $F_E$, and $Q$.\\
Initialization: $ \bm P = \bm 1 \in \mathbb R^{N_x N_y (N_t+1)\times 1}$,  $ \bm M^x = \bm 1  \in \mathbb R^{(N_x-1) N_y  N_t \times 1}$, $\bm M^y = \bm 1 \in \mathbb R^{N_x (N_y-1) N_t \times 1}$, and  $ \bm \Phi = \bm 0 \in \mathbb R^{N_x N_yN_t \times 1}$.\\
Output: $\bm P $, $\bm M^x$, $\bm M^y$ and $\bm \Phi $.\\
\For {$k \in \mathbb N$}{
Update $M_{i+\f{1}{2}, j,  n}^{x, (k+1)}$ and $M_{i, j+\f{1}{2},  n}^{y, (k+1)}$ by \vspace{-0.05in} 
{\footnotesize
\begin{equation}\begin{array}{l}\label{AL:m1}
  \ds M_{i+\f{1}{2},j, n}^{x, (k+1)} ={\big( { M_{i+\f{1}{2}, j, n}^{x, (k)}} +\sigma_m  \hat \delta_x  \bar \Phi_{i,j, n}^{(k)} \big)}\bigg/{\bigg[\f{2\sigma_m}{P_{i, j, n}^{(k)} + P_{i+1,j, n}^{(k)}} + 1\bigg]},\\[0.1in]
\ds M_{i,j+\f{1}{2}, n}^{y, (k+1)} ={\big( { M_{i, j+\f{1}{2}, n}^{y, (k)}}  + \sigma_m  \hat \delta_y  \bar \Phi_{i,j,n}^{(k)}\big)}\bigg/{\bigg[\f{2\sigma_m}{P_{i, j, n}^{(k)} + P_{i,j+1, n}^{(k)}} + 1\bigg]}.
\end{array}
\end{equation}}
Update $ P_{i,j,n}$   by {\footnotesize
\begin{equation}\begin{array}{l}\label{AL:rho}
  \ds \ds P_{i,j, n}^{(k+1)}  =   P_{i,j,n}^{(k)} + \sigma_m \big(H_{i,j,n}^{(k)} - \p_{P_{i, j, n}^{(k)}}F (P_{i, j, n}^{(k)} ) - \hat \delta_t^\alpha \bar \Phi_{i, j, n}^{(k)}\big).
\end{array}
\end{equation}}
Here $H_{i,j,n}^{(k)}$ is defined in \eqref{update:rho:DJ}.

Update $\bm \Phi^{(k+1)}$  by {\footnotesize
\begin{equation}\begin{array}{l}\label{AL:phi}
  \ds \ds\ds \bm \Phi^{(k+1)}  = \bm \Phi^{(k)} + \sigma_\phi (\bm K \bm K^\top)^{-1} \big( \delta_t^\alpha \bm P^{(k+1)} + \delta_x \bm M^{x, (k+1)} +\delta_y \bm M^{y,(k+1)} \big),\\[0.1in]
  \bar{\bm \Phi}^{(k+1)}=2 \bm \Phi^{(k+1)}- \bm \Phi^{(k)}.
  \end{array}
\end{equation}}
}
\end{algorithm}

\paragraph{Step 3: update $\bm \Phi^{(k+1)}$}
Maximize the discrete function defined in \eqref{sad_DL} in terms of the last component: Find the optimal $\bm \Phi^{(k+1)}$ which is the solution to  \vspace{-0.05in}
\begin{equation}\label{L:phi}
 \ds \bm \Phi^{(k+1)}=\sup_{\bm \Phi}~ \breve{ \mathcal L}(\bm P^{(k+1)}, \bm M^{x, (k+1)}, \bm M^{y, (k+1)}, \bm  \Phi) - \f{\|\bm  \Phi- \bm \Phi^{(k)}\|_H^2}{2 \sigma_\phi}.
\end{equation}
Similar to \eqref{update:m:e1}, this maximization  has a closed form solution which can be further simplified by  \eqref{dis:phi} and \eqref{GPDHG:Precon} as follows:
\begin{equation}\begin{array}{l}\label{update:phi:e1}
  \ds \bm \Phi^{(k+1)}  = \bm \Phi^{(k)} + \sigma_\phi (\bm K \bm K^\top)^{-1} \big( \delta_t^\alpha \bm P^{(k+1)} + \delta_x \bm M^{x, (k+1)} +\delta_y \bm M^{y,(k+1)} \big).
  \end{array}
\end{equation}
We finally do the last step in the PDHG algorithm \eqref{PDHG:AL}, that is,
$\bm {  \bar \Phi}^{(k+1)}=2 \bm \Phi^{(k+1)}-\bm \Phi^{(k)}$.
The G-prox PDHG algorithm for the time-fractional MFP/OT  problem \eqref{Model:MFP} is summarized in Algorithm \ref{algorithm}.

\section{Numerical investigation}\label{Nume:MFP}
We conduct numerical experiments to test the convergence of the proposed numerical algorithm to  one-dimensional  integer-order and fractional OT examples.  We then perform some numerical experiments to show the efficiency and effectiveness of our proposed algorithm as well as the flexibility of our algorithm on diverse OT and MFP problems.
\subsection{Convergence behavior}\label{S:Converge}
The data are as follows: $\Omega \times [0, T] =[0, 1]^2$,  $\rho_0(x)=x+\frac{1}{2}$, and  $\rho_1(x)=1$. In   Algorithm \ref{algorithm},  we choose   $\sigma_m = 0.1$ and $\sigma_\phi = 0.05$ throughout this section.
\paragraph{Integer-order OT} For the integer-order OT problem,  the exact solutions $\rho$ and $ m$ are given as follows \cite{YuLai}:
\begin{equation}\label{Int_OT_rho}
\rho(x, t)=\left\{\begin{array}{ll}
\ds x+\frac{1}{2}, & t=0, \\[0.075in]
\ds \frac{\sqrt{2 t x+\left(\frac{t}{2}-1\right)^2}+t-1}{t \sqrt{2 t x+\left(\frac{t}{2}-1\right)^2}}, & 0<t \leq 1 .
\end{array}\right.\vspace{-0.08in}
\end{equation}
\begin{equation}\label{Int_OT_m}
m(x, t)=\left\{\begin{array}{ll}
\ds \frac{1}{4} x(x-1)(2 x+1), & t=0, \\[0.1in]
\ds \frac{x}{t^2}-\frac{3-t}{2 t^3} \sqrt{2 t x+\left(\frac{t}{2}-1\right)^2} \\[0.1in]
\ds -\frac{(t-1)\left(t^2-4\right)}{8 t^3} \frac{1}{\sqrt{2 t x+\left(\frac{t}{2}-1\right)^2}}-\frac{3 t-4}{2 t^3}, & 0<t \leq 1.
\end{array}\right.
\end{equation}

\begin{table}[!htbp]
\setlength{\abovecaptionskip}{0pt}
\centering
\caption{Convergence rates of the integer-order OT approximation.}	
\begin{tabular}{cccccccc} \hline
  $N_x \times N_t$   & $\| \rho-\bm P\|$   &  Order   &  $\| m-\bm M\|$  & Order \\ \hline

$8^2$ &1.37e-03  & --      & 2.30e-03  & --    \\
$10^2$ &1.10e-03 & 0.95    & 1.84e-03 & 0.96   \\
$20^2$ &5.30e-04 & 1.06    & 9.12e-04  &  1.01   \\
$25^2$ & 4.12e-04 & 1.13    & 7.27e-04  & 1.01    \\ \hline
\end{tabular}
\label{table1:OT}\end{table}
In Table \ref{table1:OT}, we compute the convergence rates of $\bm P$ and $\bm M$ measured in  the discrete $L^2$ norm \eqref{norm},
and the computational results indicate the convergence of  the approximation.

\paragraph{Time-fractional OT}  Since the analytical solutions for the fractional OT problem \eqref{Model:OT} are not available, we compute the reference solution with  fine temporal and spatial  mesh size $\Delta x = \Delta t = 1/200$  to measure the convergence rate. The numerical results are presented in Table \ref{table2:FOT}, which indicates the  convergence of the proposed Algorithm \ref{algorithm} for the fractional OT problem \eqref{Model:OT}.

%
%

\begin{table}[!htbp]
\setlength{\abovecaptionskip}{0pt}
\centering
\caption{Convergence rates of the fractional OT approximation.}	
\begin{tabular}{cccccccc} \hline
\multirow{5}{*}{$\| \rho-\bm P\|$} & $N_x \times N_t$   & $\alpha =0.6$   &     & $\alpha =0.8$   &  & $\alpha =0.9$   & \\ \hline

&$8^2$ &4.28E-02  & --      & 3.89E-02  & --     & 3.89E-02  & --\\
&$10^2$ &3.44E-02 & 0.98    & 3.03E-02  & 1.12   & 3.03E-02  & 1.11\\
&$20^2$ &1.81E-02 & 0.93    & 1.39E-02  &  1.13    & 1.39E-02  & 1.13\\
&$25^2$ &1.54E-02 &0.89    & 1.08E-02  & 1.15   & 1.07E-02  & 1.16 \\ \hline
\multirow{5}{*}{$\| m-\bm M\|$}
&$8^2$  &6.31E-02 & --      &2.99E-02 & --    &1.18E-02  & --\\
&$10^2$ &5.40E-02 &0.70   &2.57E-02 & 0.69   &1.01E-02  &0.71 \\
&$20^2$ &3.13E-02 & 0.79   &1.52E-02 & 0.75   &5.90E-03  &0.78 \\
&$25^2$ &2.66E-02 & 0.89   &1.27E-02 &0.83   &4.80E-03   &0.86 \\ \hline
\end{tabular}
\label{table2:FOT}\end{table}

\subsection{Performance of the one-dimensional OT}\label{S:Perf}
With the numerically tested convergence of the integer-order and fractional OT approximation, we then conduct numerical experiments to investigate the performance of the fractional OT problem for different  values of the fractional order $\alpha$ in comparison with the OT problem constrained with the integer-order transport PDE.

\begin{figure}[h]
	\setlength{\abovecaptionskip}{0pt}
	\centering
\includegraphics[width=1.5in,height=1.25in]{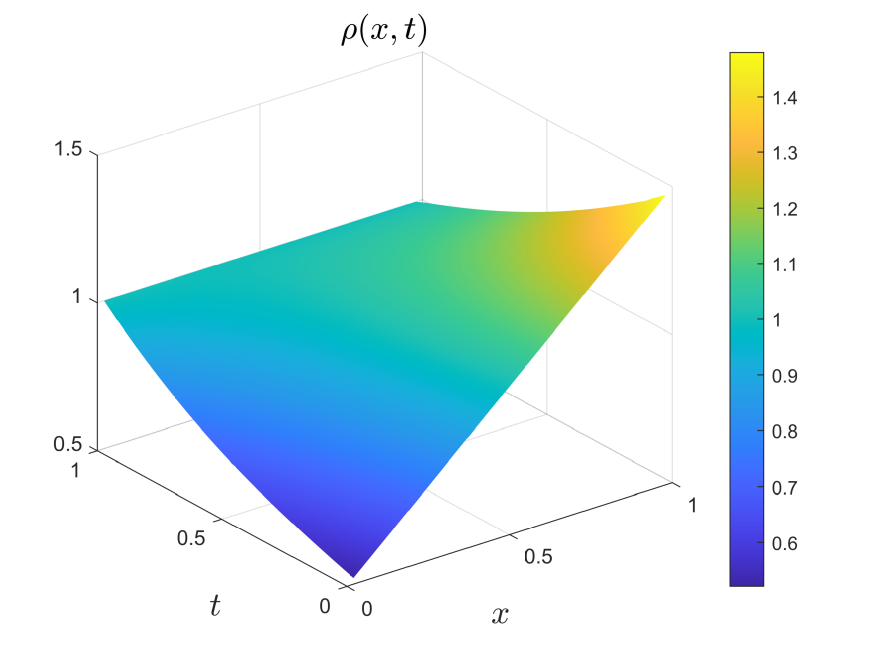}
\includegraphics[width=1.5in,height=1.25in]{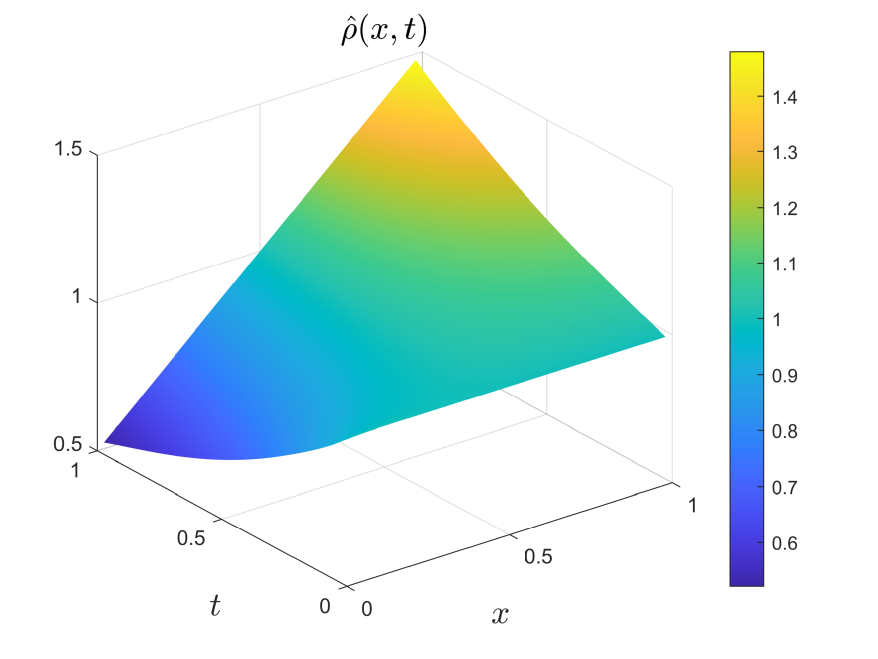}
\includegraphics[width=1.5in,height=1.25in]{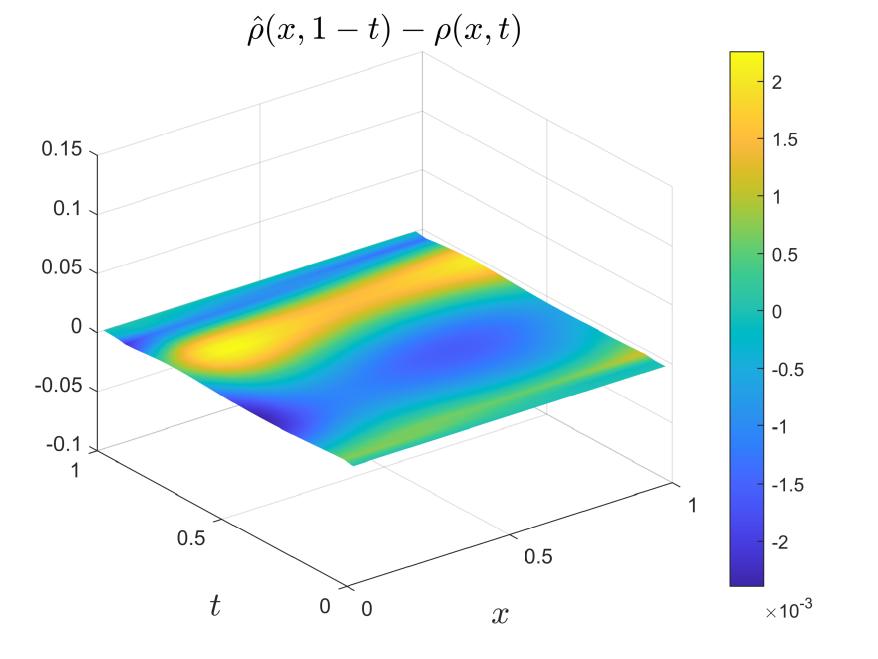}\\
\includegraphics[width=1.5in,height=1.25in]{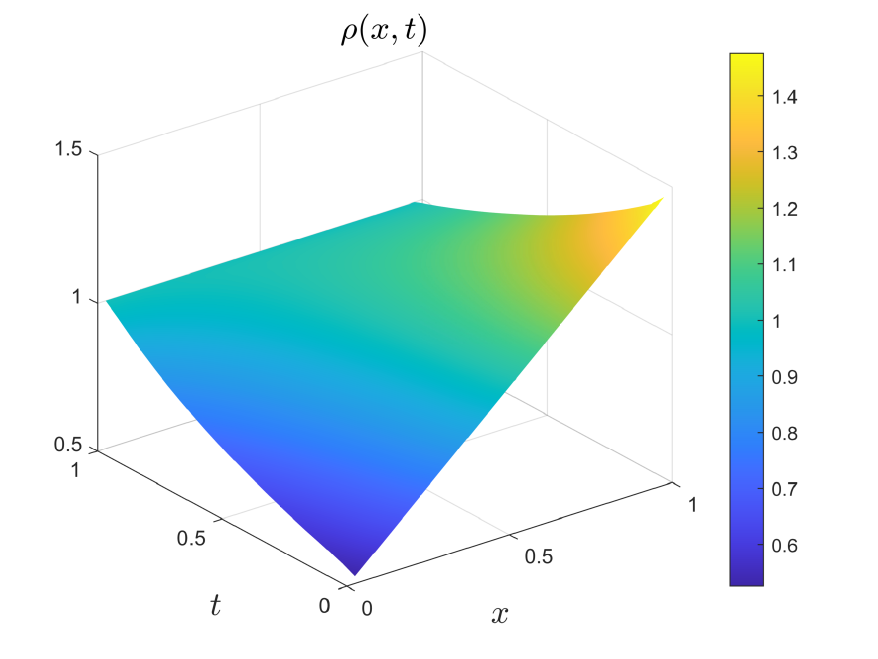}
\includegraphics[width=1.5in,height=1.25in]{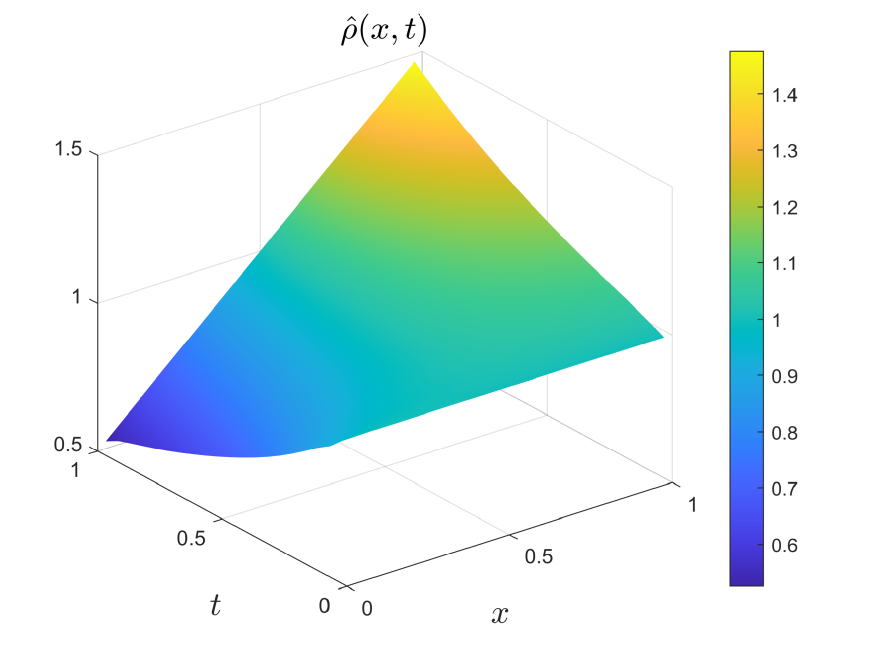}
\includegraphics[width=1.5in,height=1.25in]{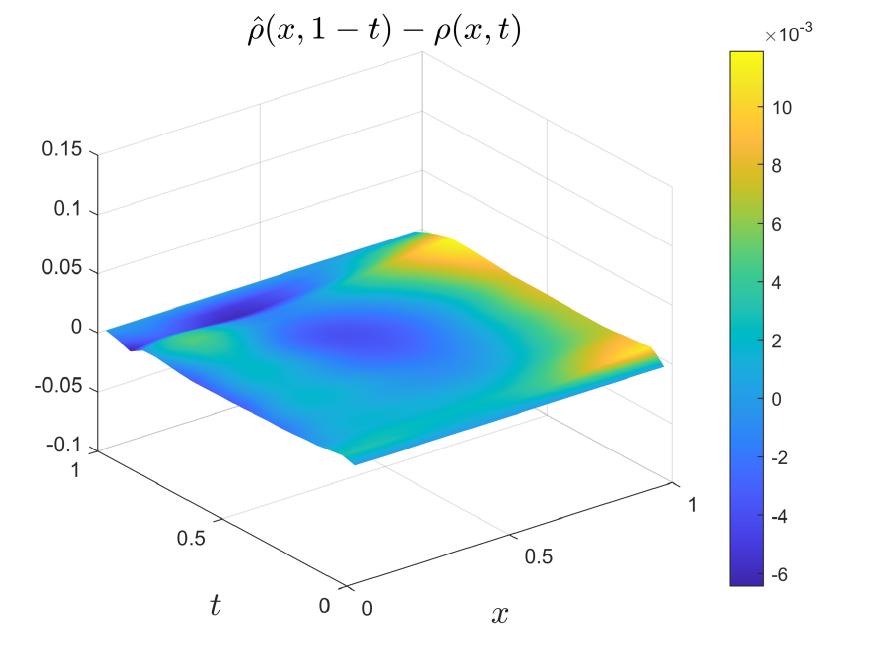}\\
\includegraphics[width=1.5in,height=1.25in]{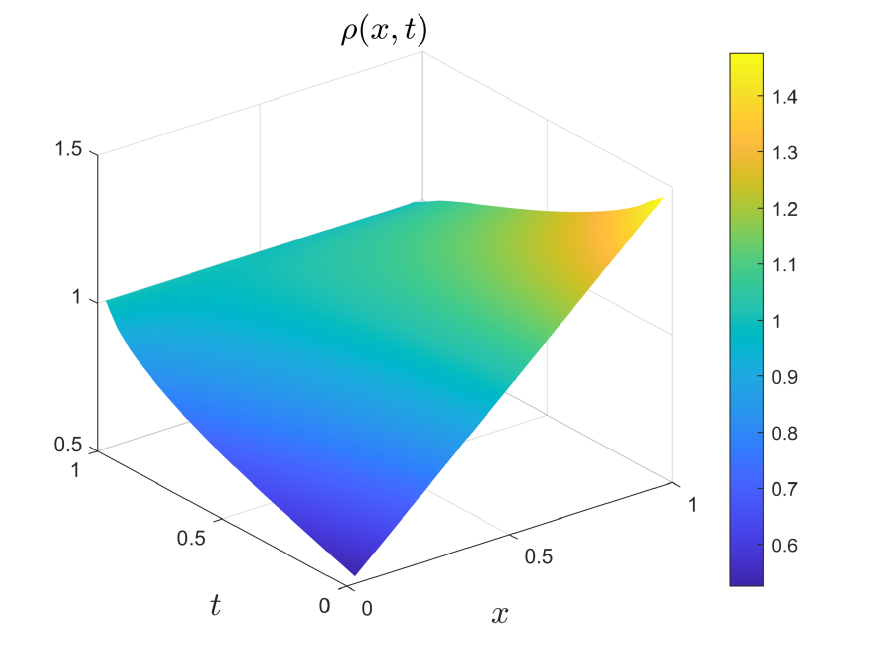}
\includegraphics[width=1.5in,height=1.25in]{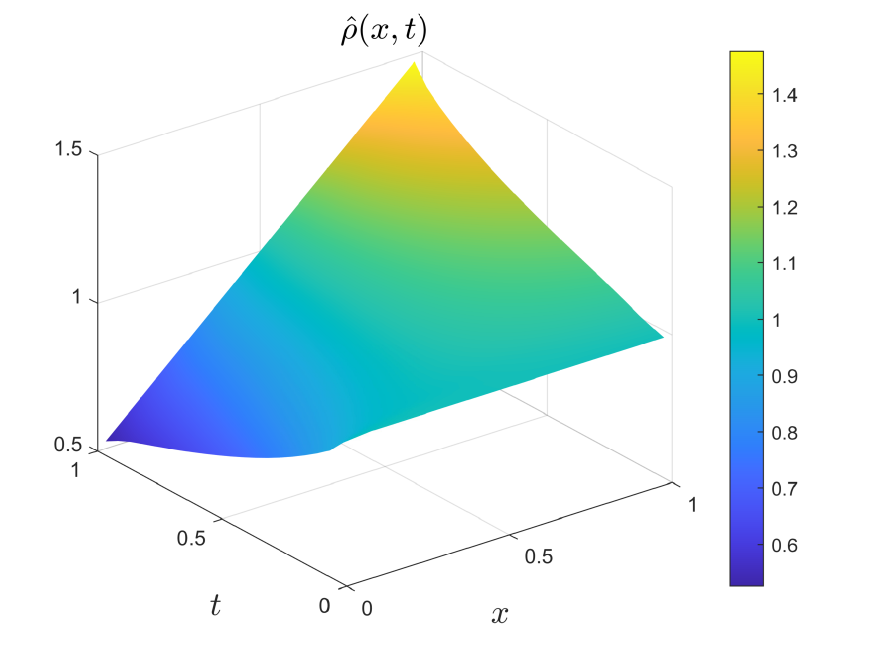}
\includegraphics[width=1.5in,height=1.25in]{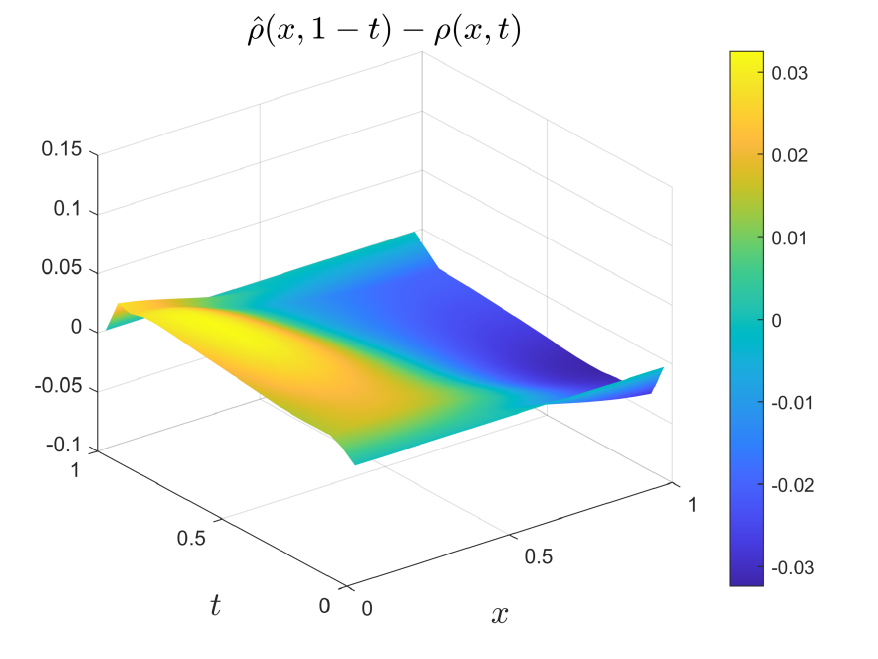}\\
\includegraphics[width=1.5in,height=1.25in]{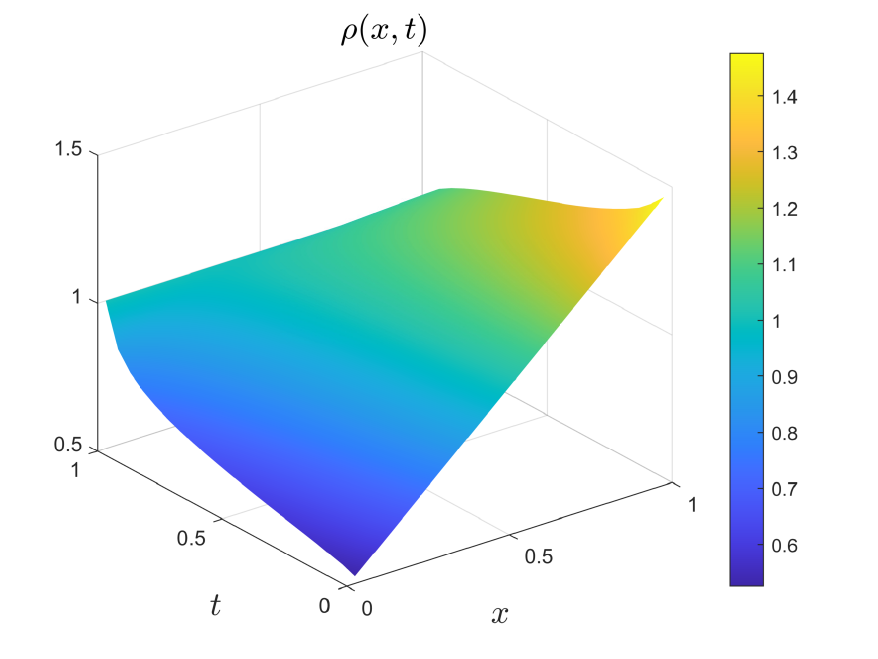}
\includegraphics[width=1.5in,height=1.25in]{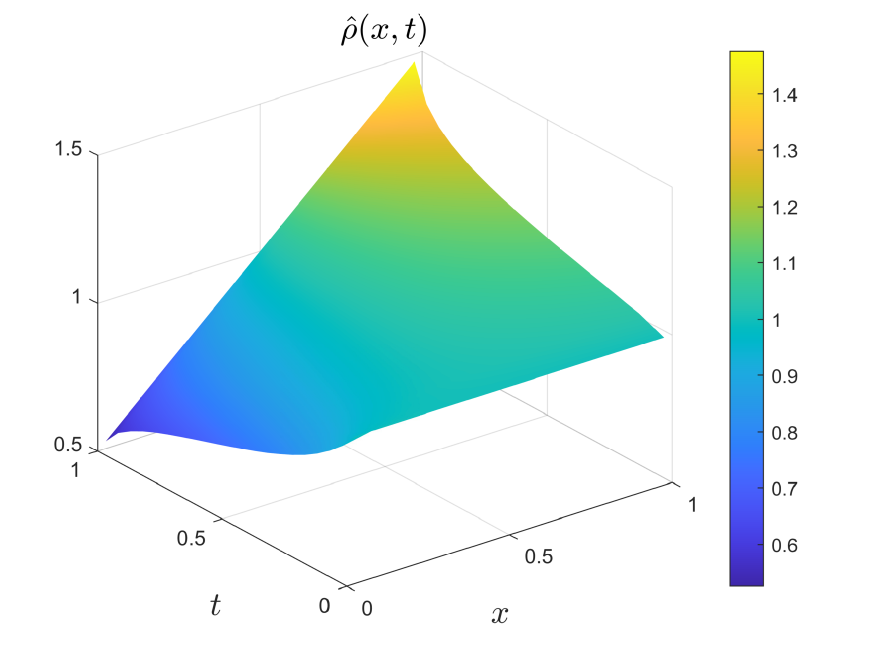}
\includegraphics[width=1.5in,height=1.25in]{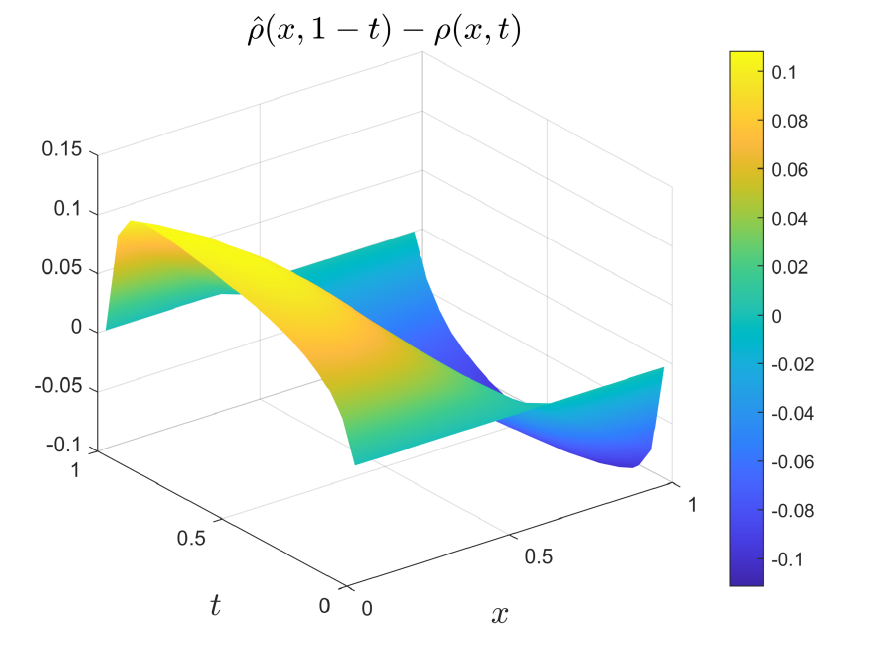}\\
\caption{Performance of the one-dimensional  OT problem for {\it Test 5.2.1} in \S \ref{S:Perf}. Left to right: plots of the density $\rho(x, t)$, $\hat \rho(x, t)$, and $\hat \rho(x, 1-t)-\rho(x, t)$. The first row: the integer-order OT problem. The second through fourth rows: the corresponding plots for fractional OT problem with $\alpha =0.9$, $0.8$, and $0.6$, respectively, in  \eqref{Model:set}.}
\label{figure1:OT}
\end{figure}

\paragraph{Test 5.2.1}
Consider  two densities $\hat \rho(x, t)$ of $\rho_1 \rightarrow \rho_0$ and  $\rho(x, t)$ of $\rho_0 \rightarrow \rho_1$.
The path-reversibility of the integer-order OT problem \cite{Vil} guarantees  the agreement of  $ \rho(x, t)$ and  $\hat \rho(x, T-t)$, which is the backward-in-time counterpart of $\hat \rho(x, t)$. Motivated by this, we are in the position to examine whether this nice property still holds for its fractional analogue defined in \eqref{Model:OT}.
 Let $\Omega \times [0, T] = [0, 1]^2$  and we choose the uniform spatial mesh size  and the temporal step size $\Delta x =\Delta t = 1/40$ in Algorithm \ref{algorithm}. We present the plots of the densities $\rho(x, t)$ of $\rho_0 \rightarrow \rho_1$, $\hat \rho(x, t)$ of $\rho_1 \rightarrow \rho_0$,
 and  $\hat \rho(x, 1-t)-\rho(x, t)$ with $\rho_0$ and $\rho_1$ defined in \S \ref{S:Converge} in Figure \ref{figure1:OT}, where the first row is for the integer-order OT problem, and the second through fourth rows are for fractional OT problem \eqref{Model:OT} with the fractional order  $\alpha=0.9$, $0.8$, and $0.6$, respectively.
We observe  from  Figure \ref{figure1:OT} that \textbf{(i)} the density $\rho(x, t)$ of integer-order OT  problem tends to coincide with $\hat \rho(x, 1- t)$, which  indicates that integer-order OT problem is path-reversible and this observation is consistent with the discussions above  \cite{Vil}. Nevertheless, \textbf{(ii)}  $\rho(x, t)$ of fractional OT problem exhibits more evident deviations from the  path $\hat \rho(x, 1-t)$   for smaller fractional order $\alpha$, which implies that the fractional OT problem might lose the path-reversibility property  exhibited by its integer-order analogue.

\begin{figure}[h]
	\setlength{\abovecaptionskip}{0pt}
	\centering
\includegraphics[width=1.5in,height=1.25in]{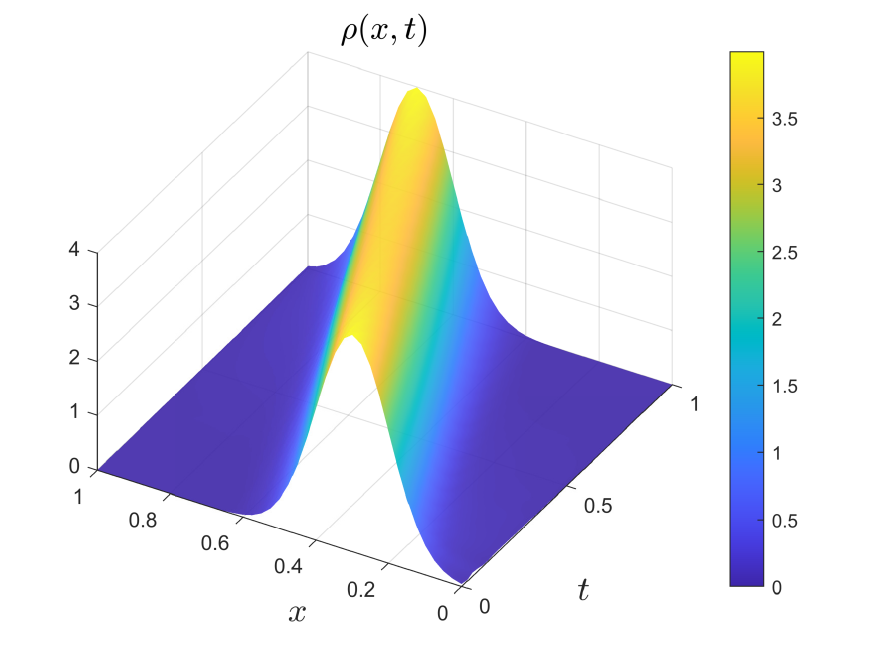}
\includegraphics[width=1.25in,height=1.25in]{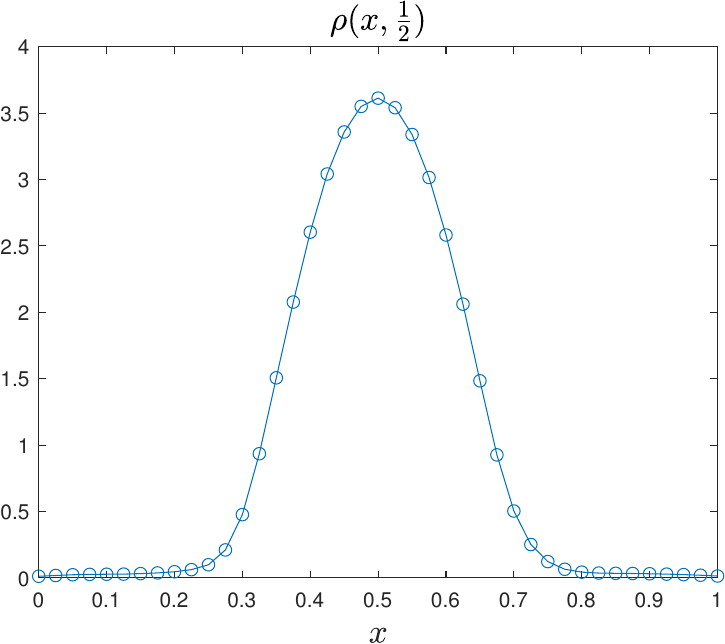} \hspace{0.1in}
\includegraphics[width=1.25in,height=1.25in]{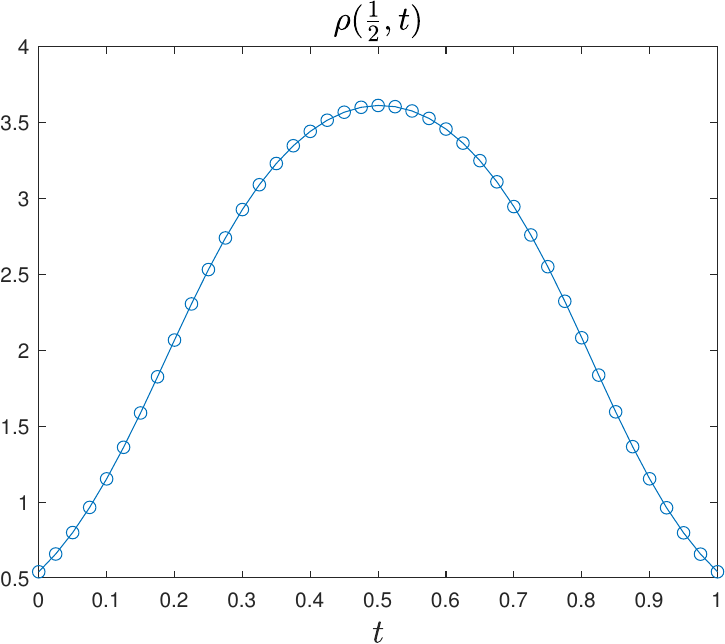}\\
\includegraphics[width=1.5in,height=1.25in]{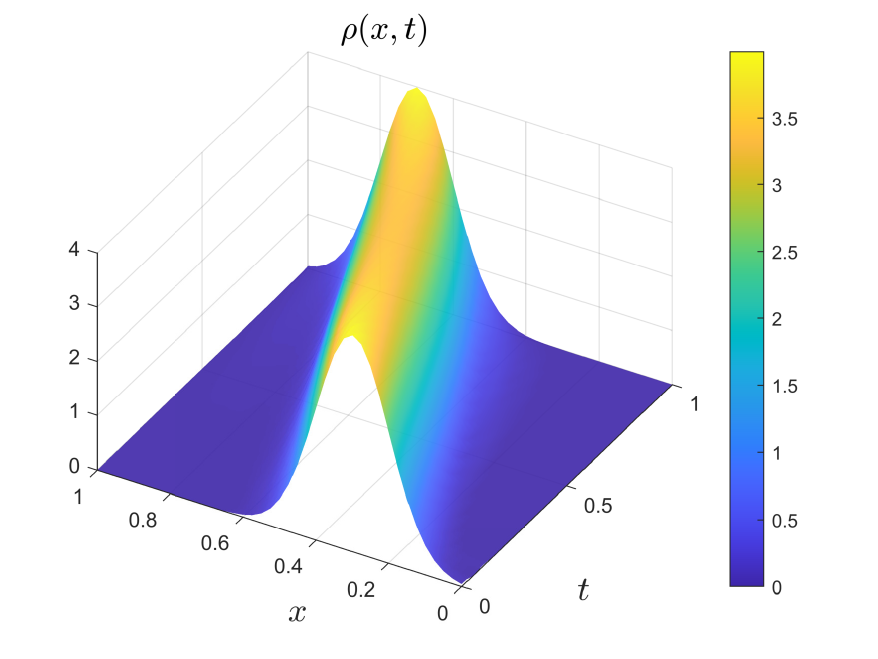}
\includegraphics[width=1.25in,height=1.25in]{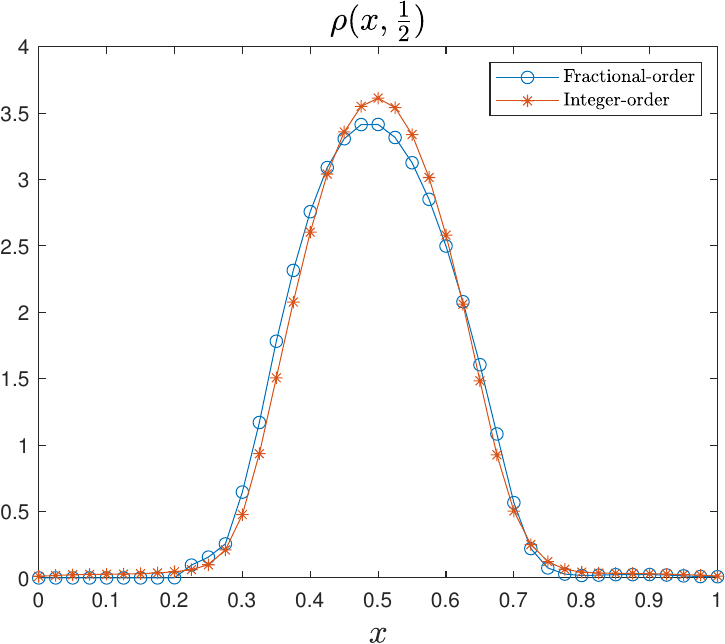}\hspace{0.1in}
\includegraphics[width=1.25in,height=1.25in]{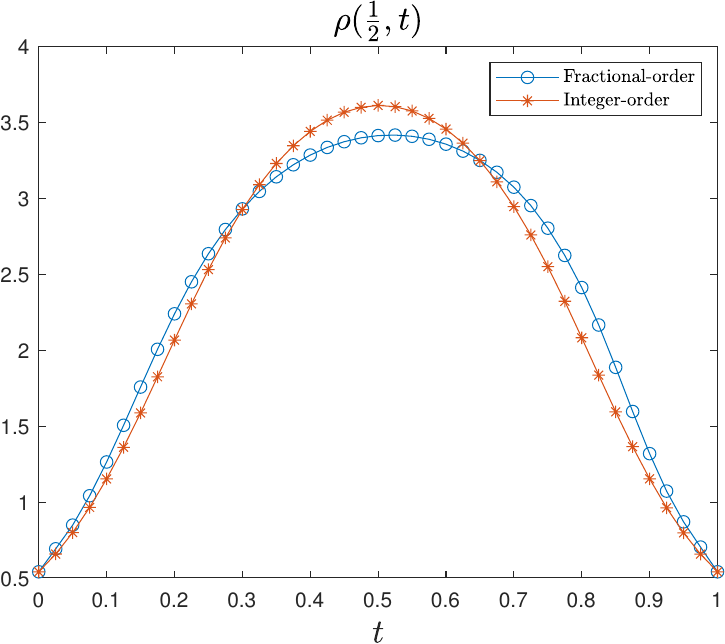}\\
\includegraphics[width=1.5in,height=1.25in]{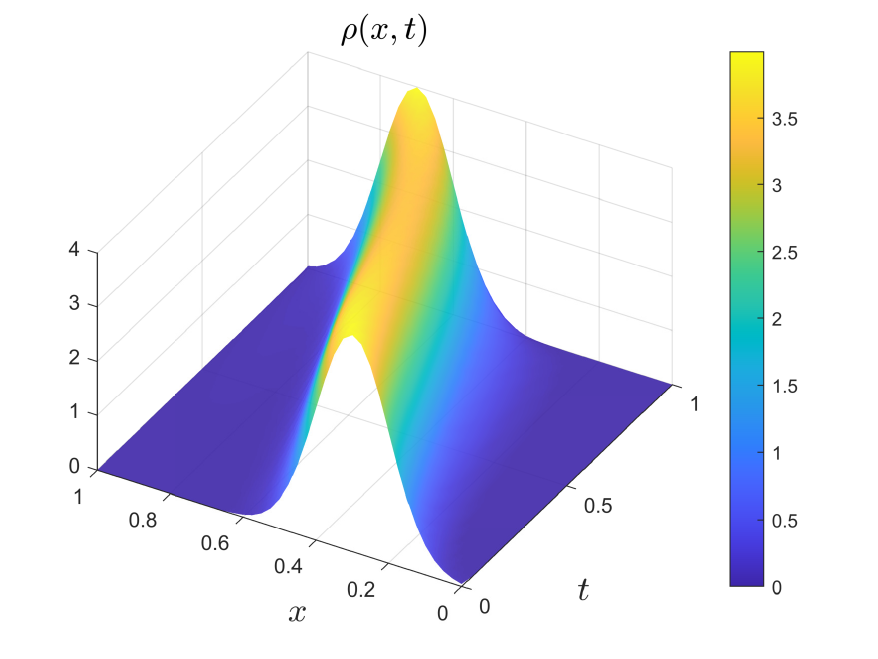}
\includegraphics[width=1.25in,height=1.25in]{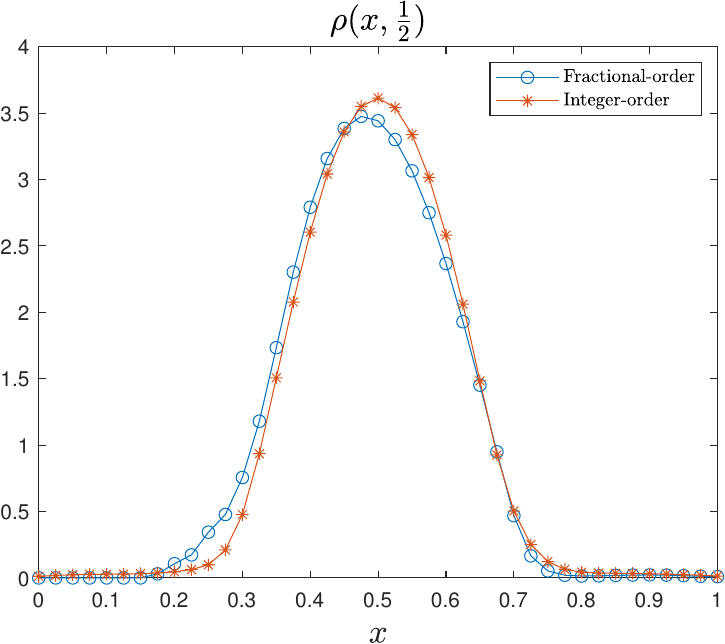}\hspace{0.1in}
\includegraphics[width=1.25in,height=1.25in]{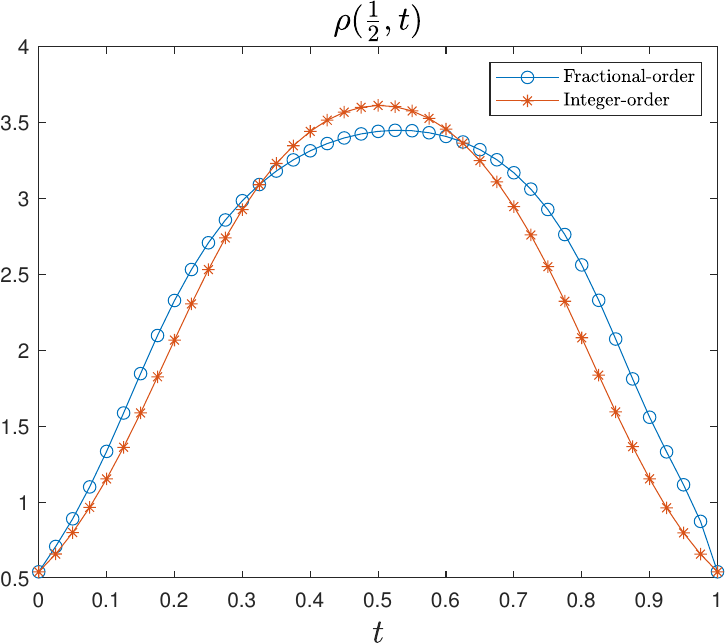}\\
\includegraphics[width=1.5in,height=1.25in]{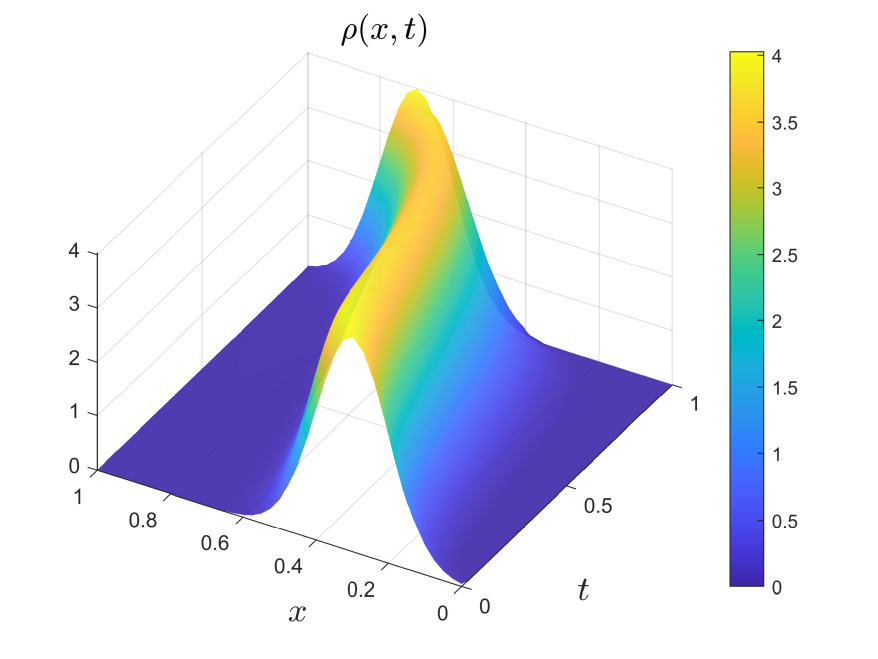}
\includegraphics[width=1.25in,height=1.25in]{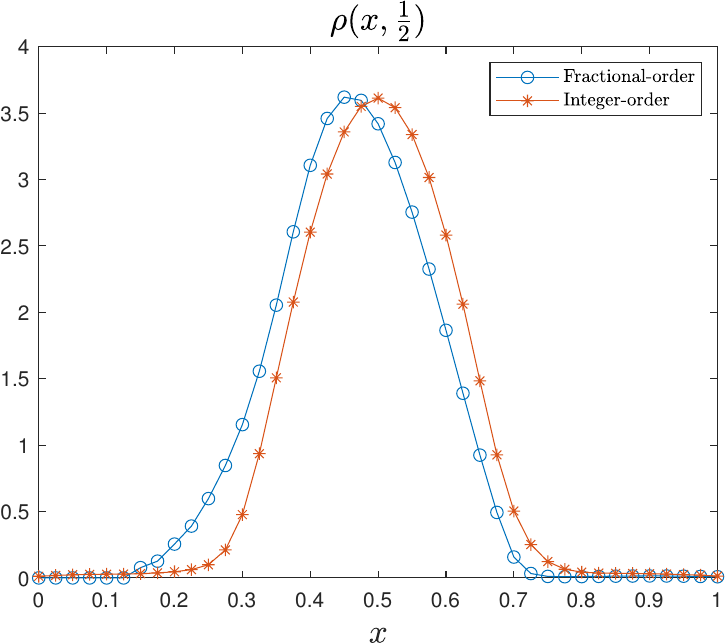}\hspace{0.1in}
\includegraphics[width=1.25in,height=1.25in]{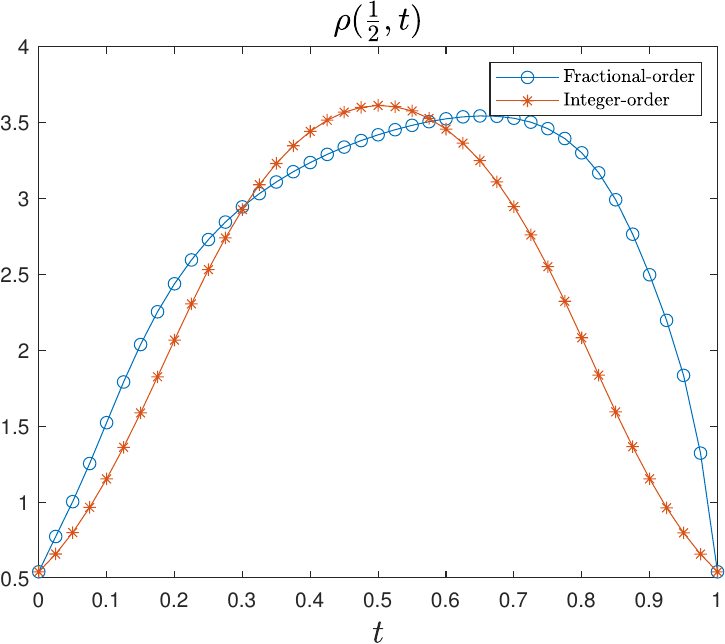}\\

\caption{Performance of the one-dimensional  OT problems for {\it Test 5.2.2} in \S \ref{S:Perf}: left column:  the density $\rho(x, t)$; middle column:  $\rho(x, t)$ at $t=\f{1}{2}$; right column:  $\rho(x, t)$ at $x=\f{1}{2}$. The first row: the integer-order OT problem; the second through fourth rows: the corresponding plots for fractional OT problem with $\alpha =0.9$, $0.8$, and $0.6$ respectively, in  \eqref{Model:set}.}
\label{figure1d1:FOT}
\end{figure}

\paragraph{Test 5.2.2}
We present the plots of the density $\rho(x, t)$   with $\rho(x, t)$ at $t=\f{1}{2}$  and $\rho(x, t)$ at $x=\f{1}{2}$  for the OT problems in Figure \ref{figure1d1:FOT} by choosing the same data as before except that $\rho_0$ and $\rho_1$ are Gaussian distribution densities  with mean $\mu_0 =0.3$, $\mu_1=0.7$ and the standard deviation  $\sigma_0=0.1$, $\sigma_1=0.1$, respectively.
We observe from Figure \ref{figure1d1:FOT} that \textbf{(i)} the density $\rho$ of the fractional OT problem exhibits the asymmetrical structure and demonstrates distortion along both the $x$ and $t$ directions in contrast with the Gaussian shape of that governed by the integer-order PDE. In addition, \textbf{(ii)} fractional OT problem generates the density $\rho$ with slower  propagation speed compared with that of the integer-order OT problem. Furthermore, \textbf{(iii)} the fractional order $\alpha$ in \eqref{Model:set} turns to characterize the propagation speed of the density $\rho$ of the fractional OT problem.  Specifically, the density propagates comparatively faster under larger fractional order $\alpha$ and approximates the solution of the integer-order OT problem as the fractional order $\alpha$ approaches 1. This is consistent with the fact that $\p_t^\alpha \rho \rightarrow \p_t \rho$ as $\alpha \rightarrow 1^{-}$  in \eqref{Model:set} for the density $\rho$ with the proper regularity \cite{JinZhou,Pod}.

\begin{figure}[h]
	\setlength{\abovecaptionskip}{0pt}
	\centering
\includegraphics[width=1.5in,height=1.25in]{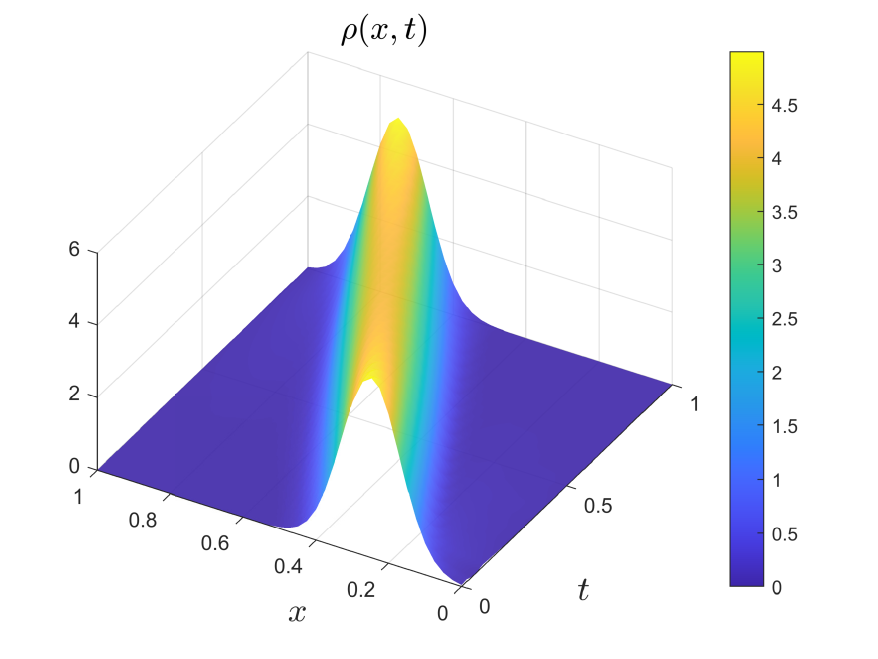}
\includegraphics[width=1.25in,height=1.25in]{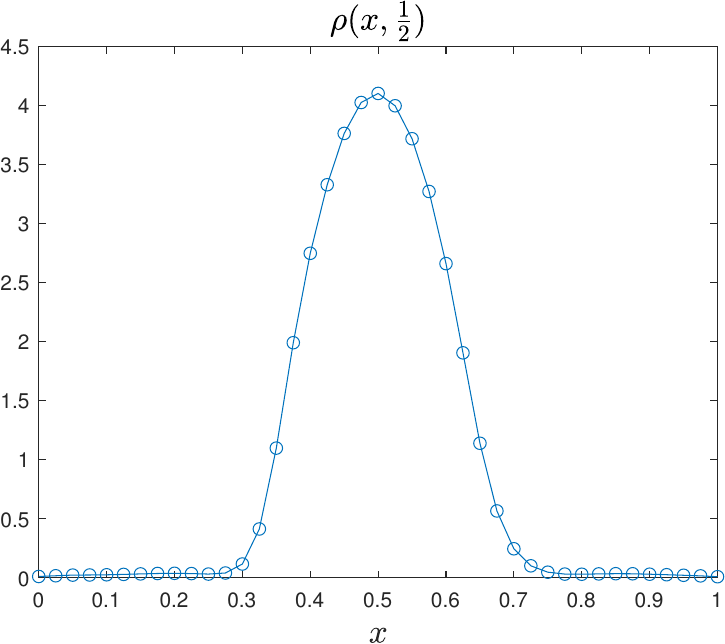} \hspace{0.1in}
\includegraphics[width=1.25in,height=1.25in]{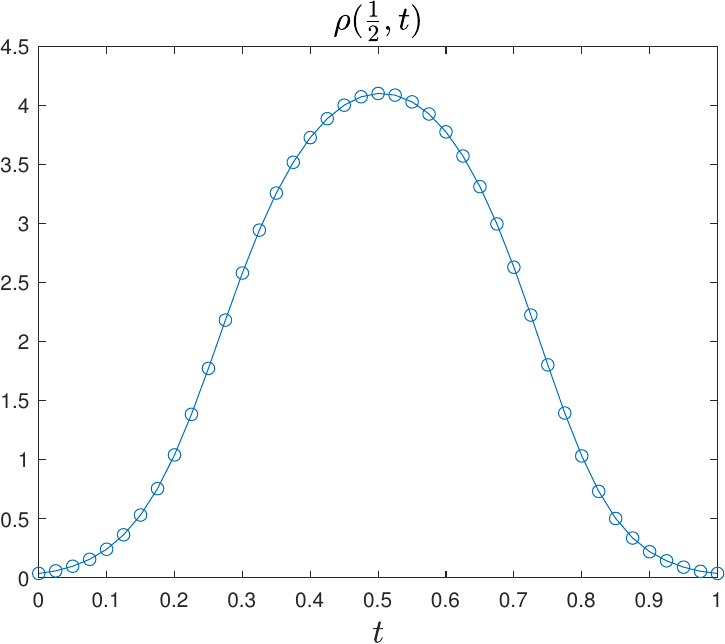}\\
\includegraphics[width=1.5in,height=1.25in]{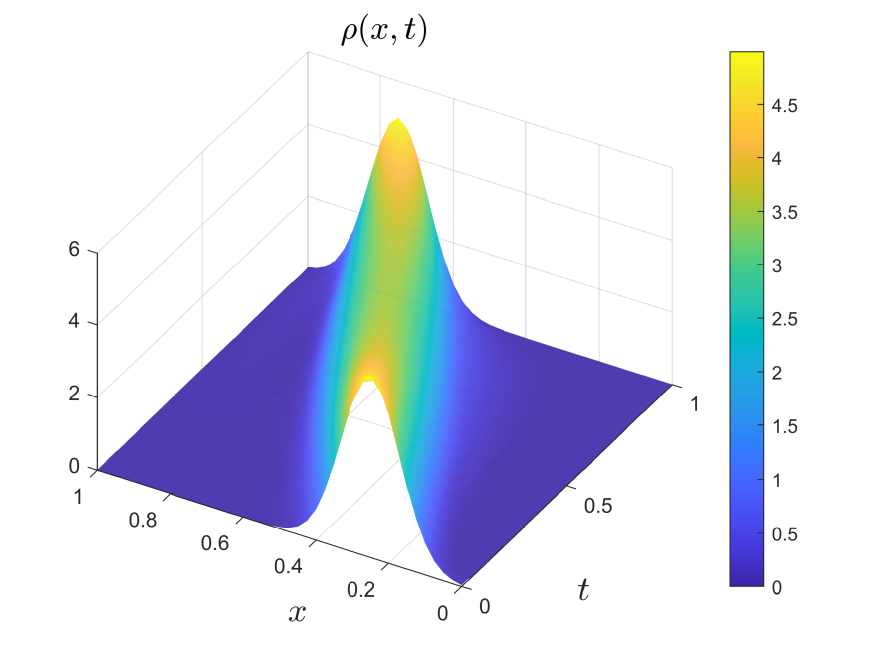}
\includegraphics[width=1.25in,height=1.25in]{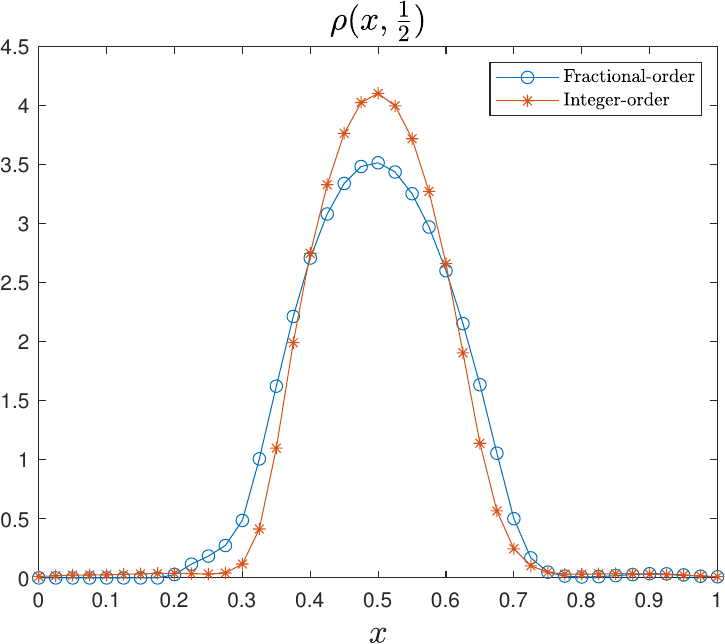} \hspace{0.1in}
\includegraphics[width=1.25in,height=1.25in]{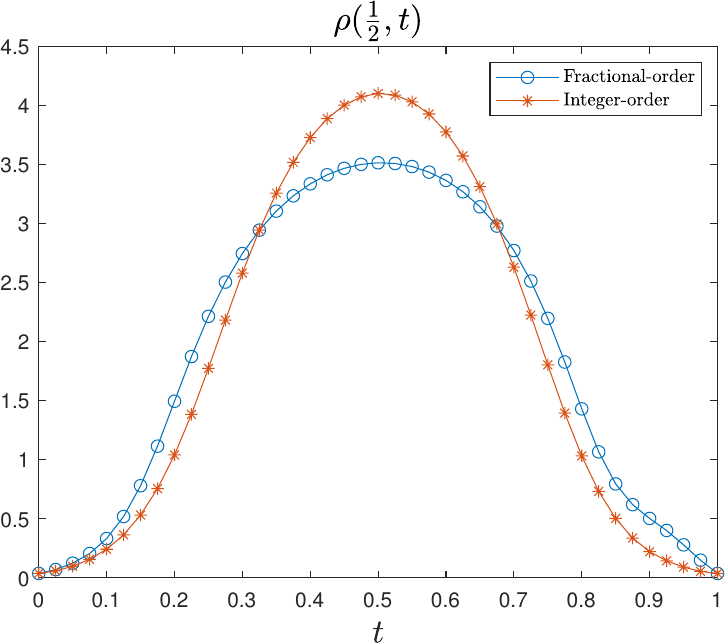}\\
\includegraphics[width=1.5in,height=1.25in]{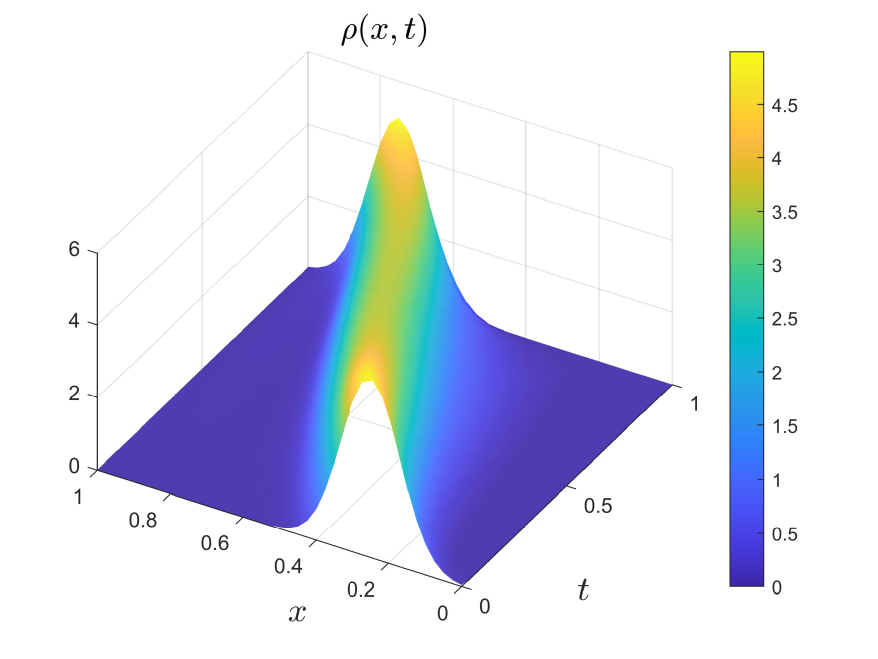}
\includegraphics[width=1.25in,height=1.25in]{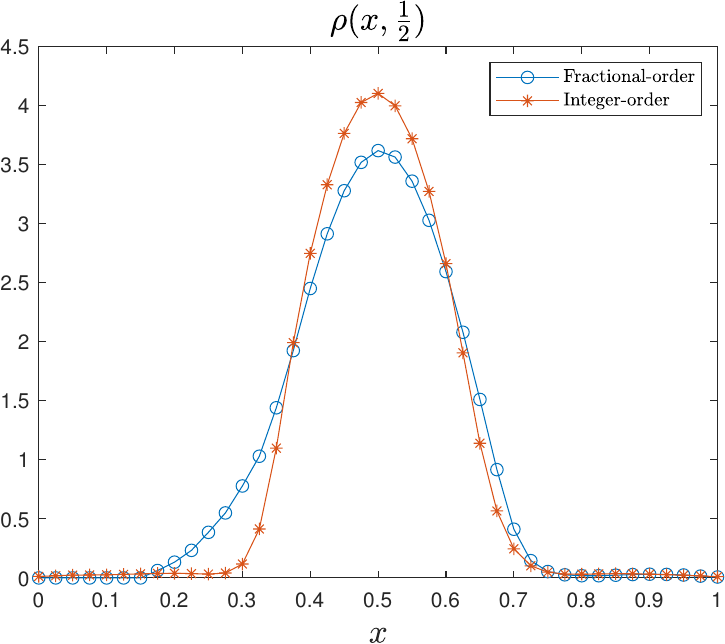} \hspace{0.1in}
\includegraphics[width=1.25in,height=1.25in]{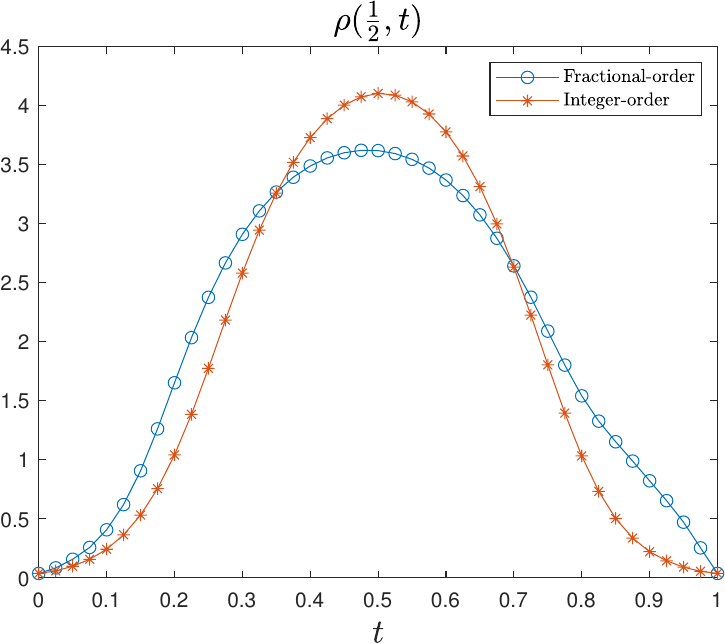}\\
\includegraphics[width=1.5in,height=1.25in]{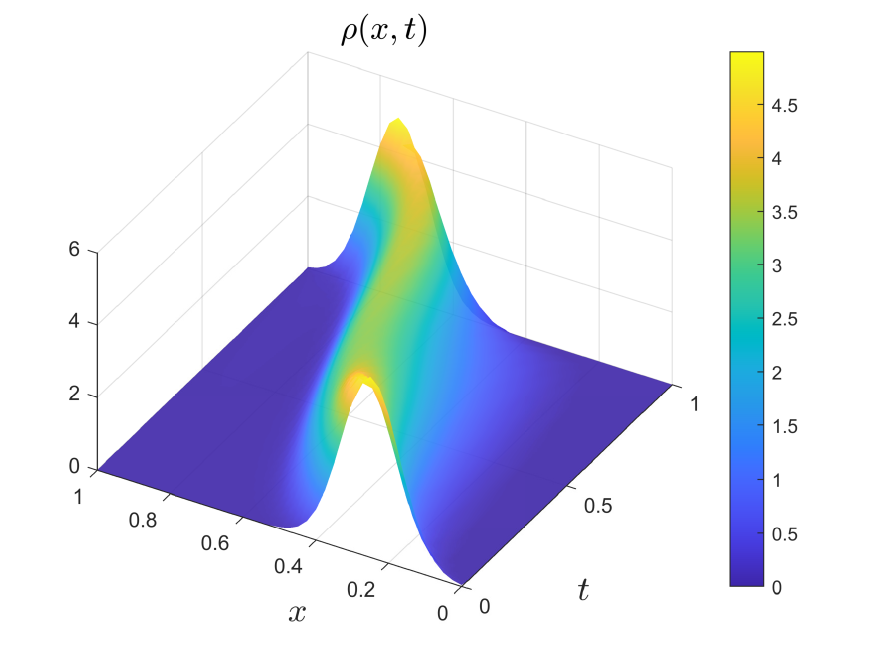}
\includegraphics[width=1.25in,height=1.25in]{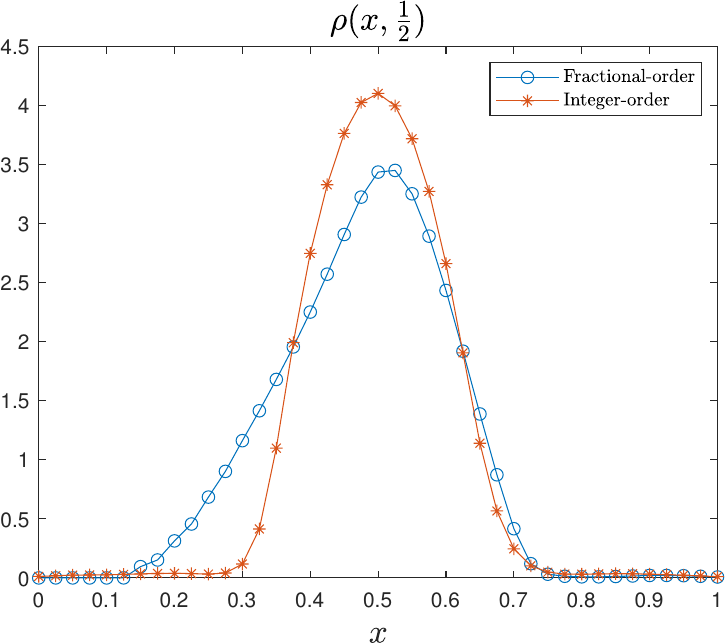} \hspace{0.1in}
\includegraphics[width=1.25in,height=1.25in]{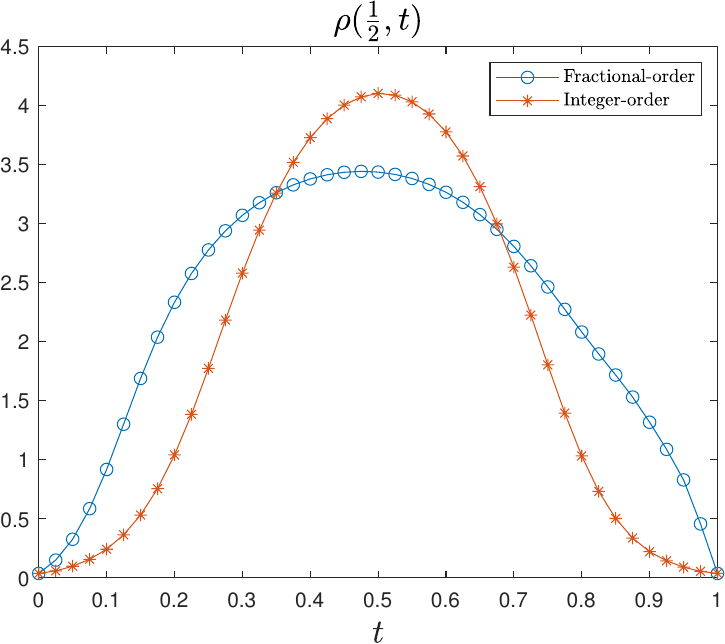}\\
\caption{Performance of the one-dimensional  OT problem for {\it Test 5.2.3} in \S \ref{S:Perf}: left column:  the density $\rho(x, t)$; middle column:  $\rho(x, t)$ at $t=\f{1}{2}$; right column:  $\rho(x, t)$ at $x=\f{1}{2}$. The first row: the integer-order OT problem; the second through fourth rows: the corresponding plots for fractional OT problem with $\alpha =0.9$, $0.8$, and $0.6$ respectively, in  \eqref{Model:set}.}
\label{figure1d2:FOT}
\end{figure}

 \paragraph{Test 5.2.3}
  We only modify the mean values and  standard deviations of the Gaussian distribution densities as follows  while keeping all the other data in {\it Test 5.2.2} unchanged:
  $$(\mu_0, \mu_1) = (0.25, 0.75), \quad \sigma_0 = \sigma_1 = 0.08.$$
  We present the plots of the density $\rho(x, t)$  with $\rho(x, t)$ at $t=\f{1}{2}$  and $\rho(x, t)$ at $x=\f{1}{2}$  for the  fractional OT problem \eqref{Model:set} with $\alpha = 0.9$, $0.8$, and $0.6$ respectively, in comparison with those for the integer-order OT problem in  Figure \ref{figure1d2:FOT}, from which we could obtain similar observations as in {\it Test 5.2.2}.

\subsection{Performance of the two-dimensional OT}\label{S:Perf2d}
Let $\Omega \times [0, T] = [0, 1]^3$. We choose the initial and terminal densities to be Gaussian distribution densities with mean $\bm {\mu_0} =(0.3, 0.3)$, $\bm {\mu_1} =  (0.7, 0.7)$ and the standard deviation $ {\sigma_0} =  {\sigma_1} = 0.1$, respectively, and follow the same data as those for Figure \ref{figure1:OT}. 

 We present the snapshots of the density $\rho$  at $t =0.1$, $0.3$, $0.5$, $0.7$, and $0.9$  in Figure \ref{figure2d1:FOT}, where the first row is  for integer-order OT problem, and the second through fourth rows are  for fractional OT problem with $\alpha =0.9$, $0.8$, and $0.6$, respectively. We observe from Figure \ref{figure2d1:FOT} that the density of the integer-order OT problem is symmetric around its mean, while its asymmetrical shape for the fractional OT problem is clearly visible, with the  tail located at the lower-left corner broader than the upper-right one.  This phenomenon becomes even more pronounced as the fractional order decreases, describing the effects of the density far lagging behind its mean position. These observations coincide with what we found in \S \ref{S:Perf}.

\begin{figure}[h]
\setlength{\abovecaptionskip}{0pt}
\centering
\includegraphics[width=1in,height=0.8in]{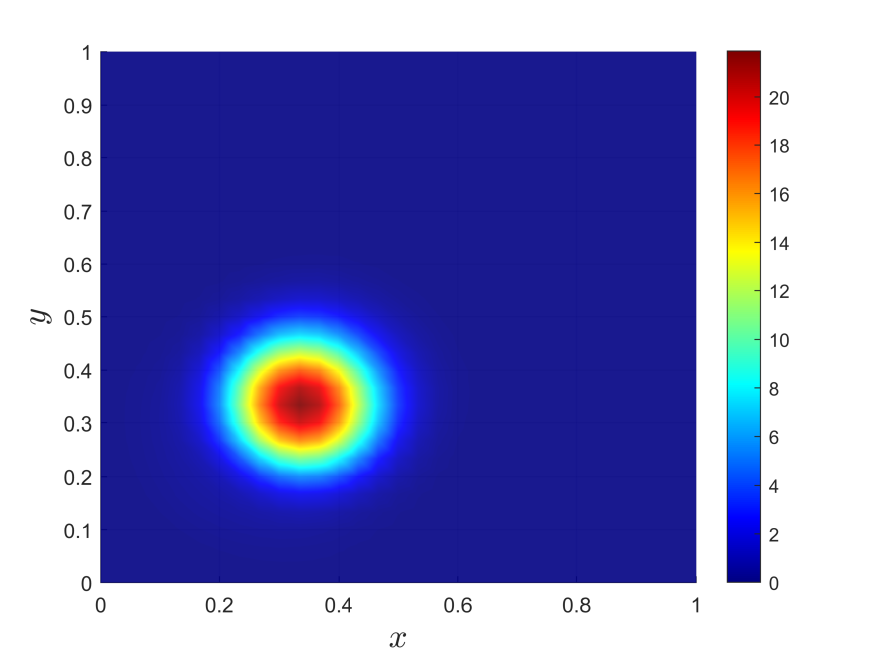}\hspace{-0.1in}
\includegraphics[width=1in,height=0.8in]{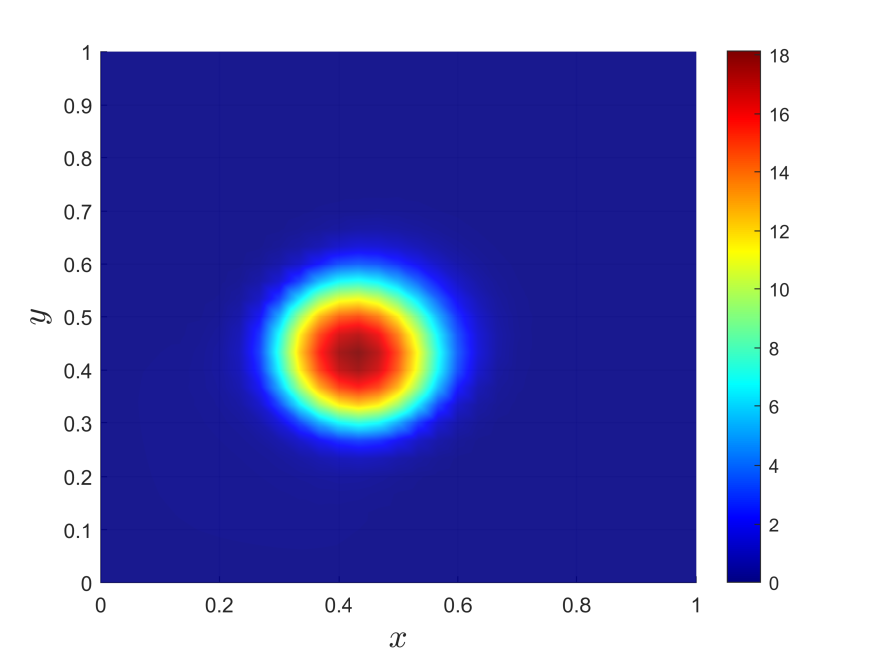}\hspace{-0.1in}
\includegraphics[width=1in,height=0.8in]{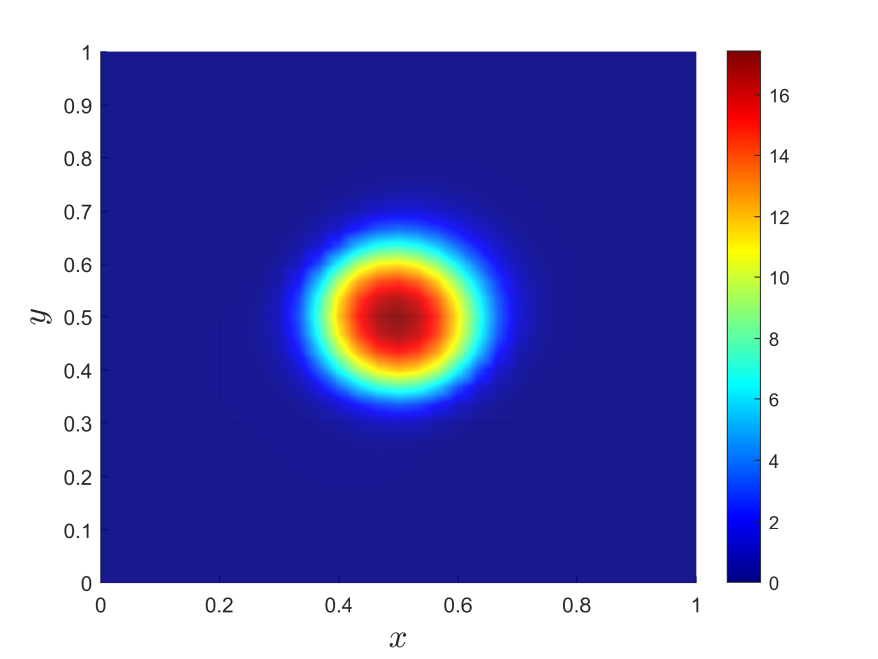}\hspace{-0.1in}
\includegraphics[width=1in,height=0.8in]{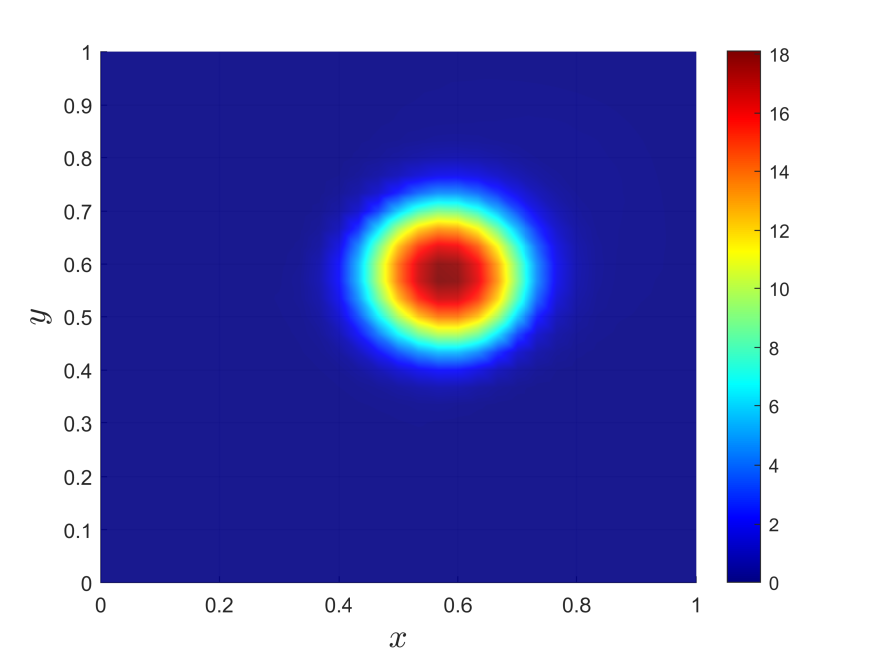}\hspace{-0.1in}
\includegraphics[width=1in,height=0.8in]{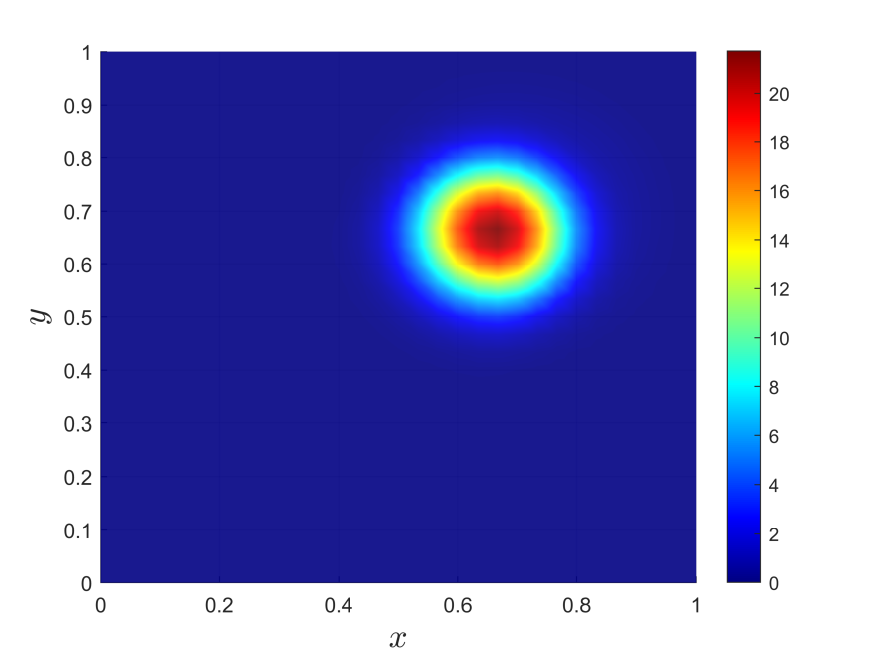}\\
\includegraphics[width=1in,height=0.8in]{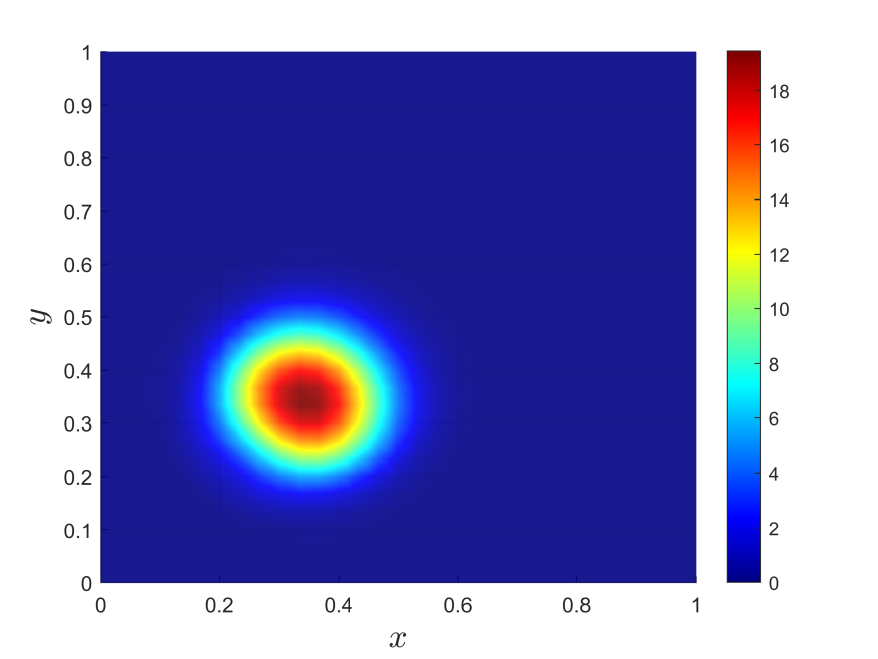}\hspace{-0.1in}
\includegraphics[width=1in,height=0.8in]{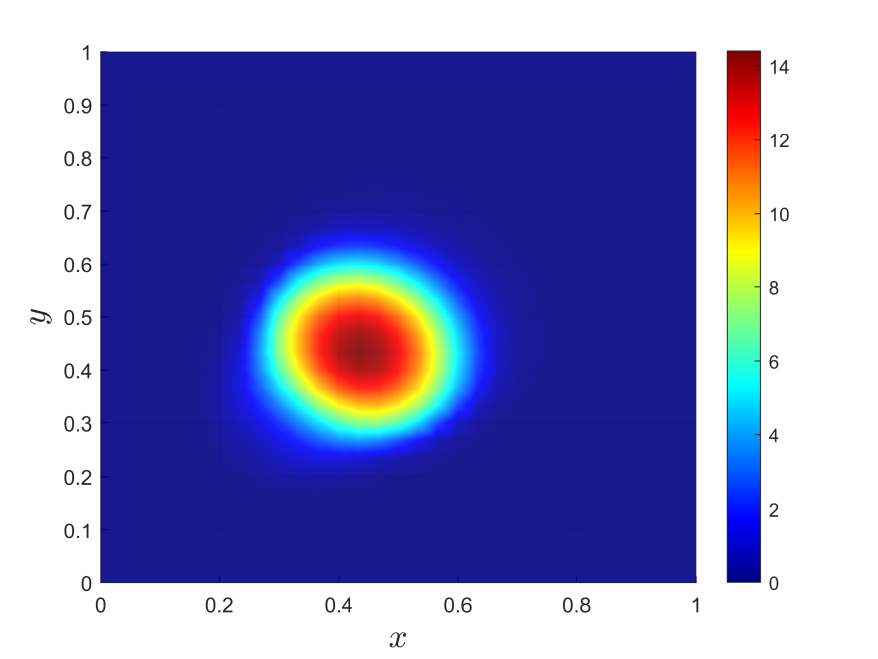}\hspace{-0.1in}
\includegraphics[width=1in,height=0.8in]{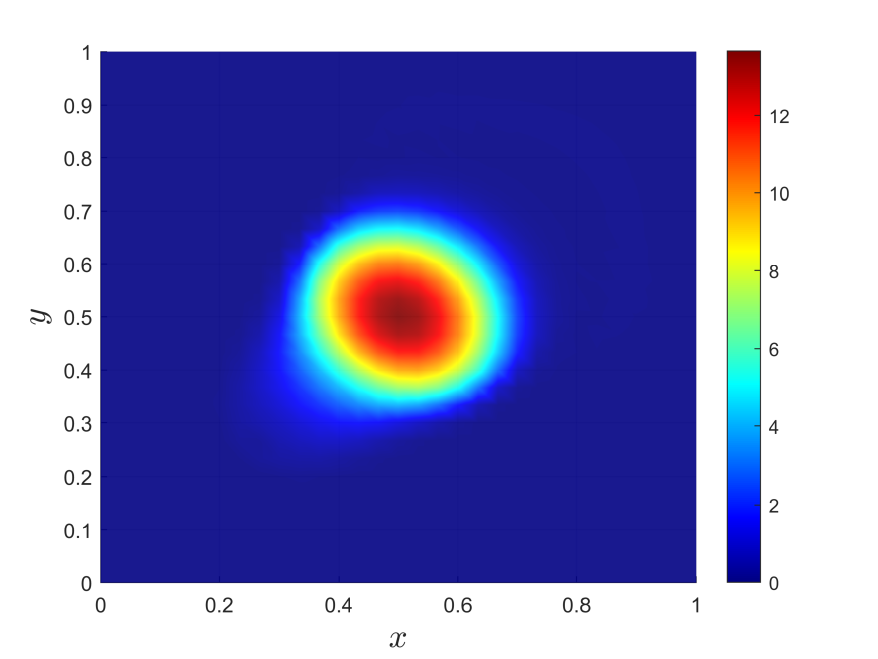}\hspace{-0.1in}
\includegraphics[width=1in,height=0.8in]{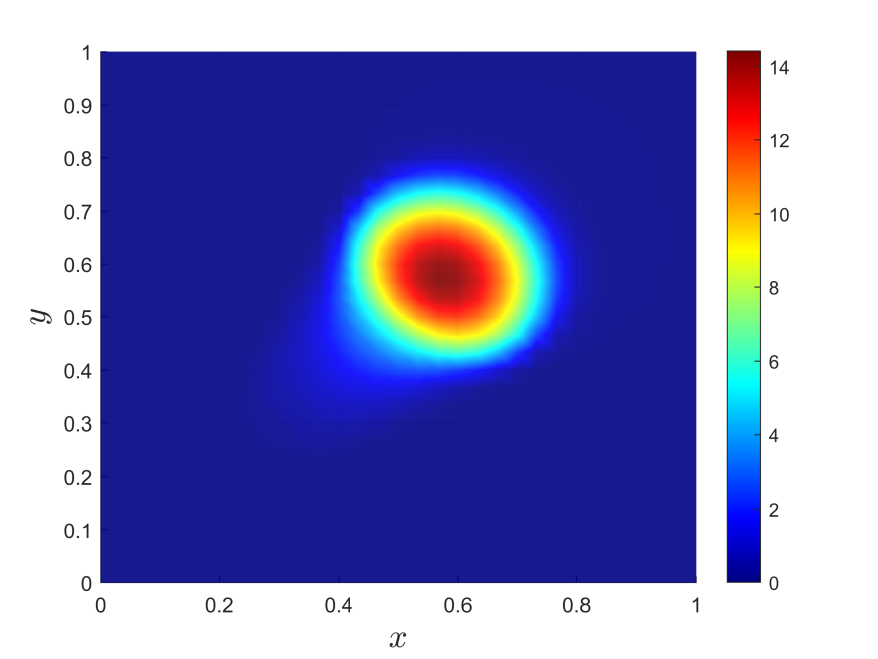}\hspace{-0.1in}
\includegraphics[width=1in,height=0.8in]{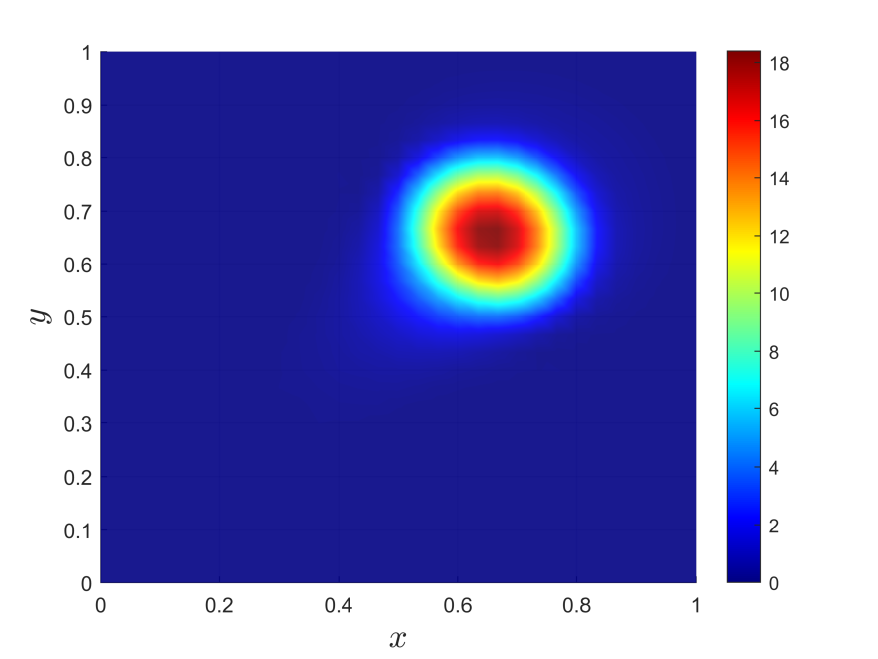}\\
\includegraphics[width=1in,height=0.8in]{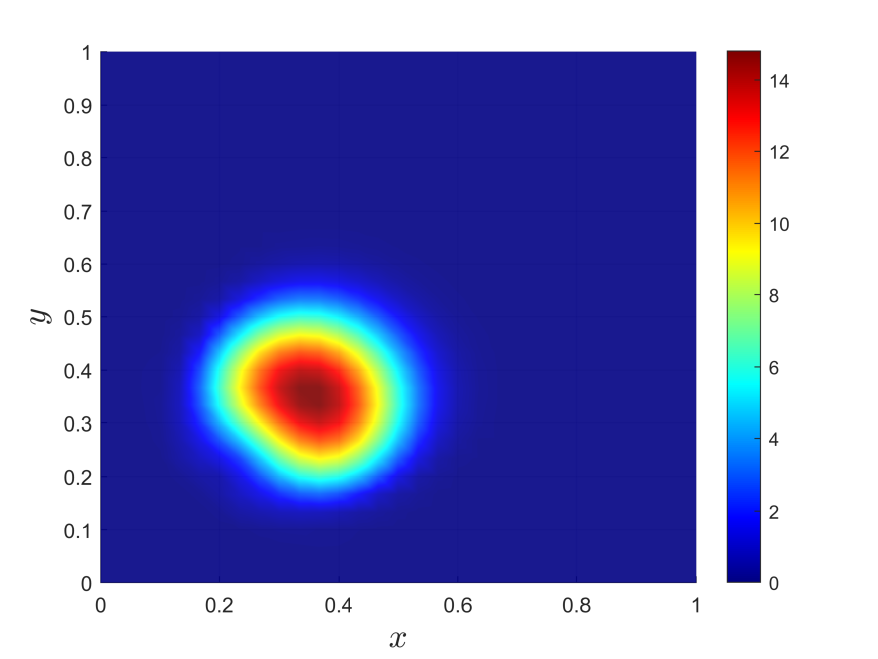}\hspace{-0.1in}
\includegraphics[width=1in,height=0.8in]{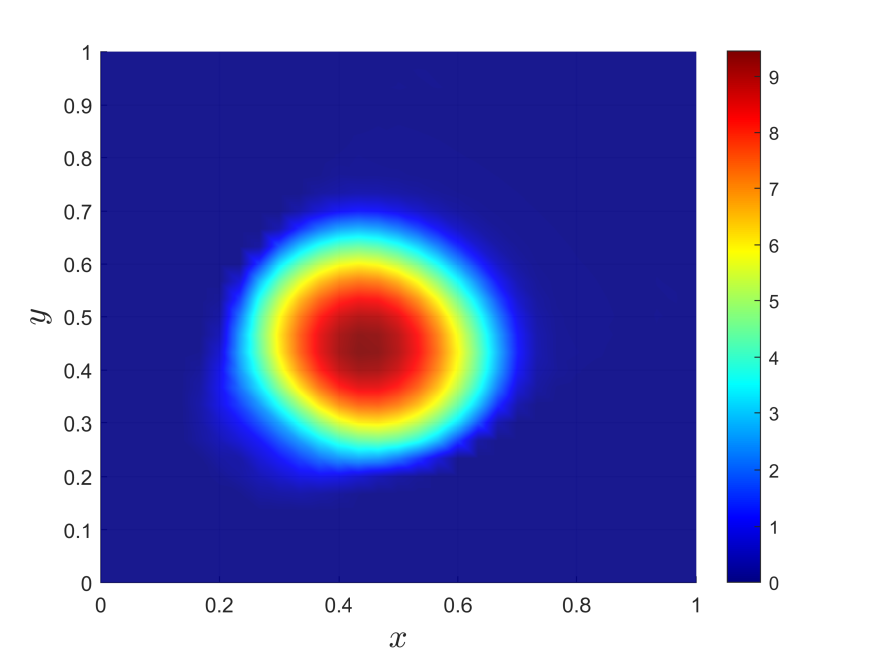}\hspace{-0.1in}
\includegraphics[width=1in,height=0.8in]{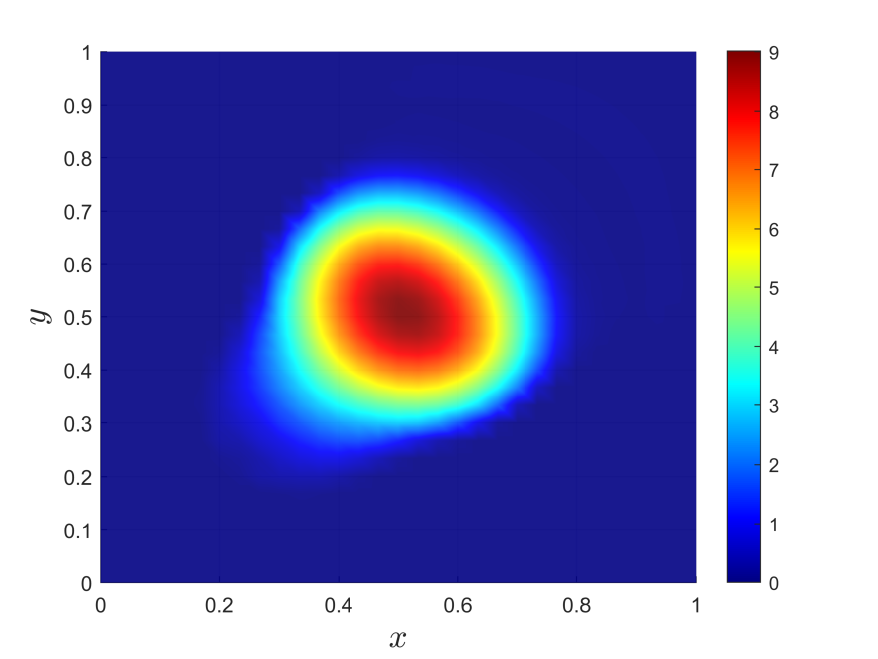}\hspace{-0.1in}
\includegraphics[width=1in,height=0.8in]{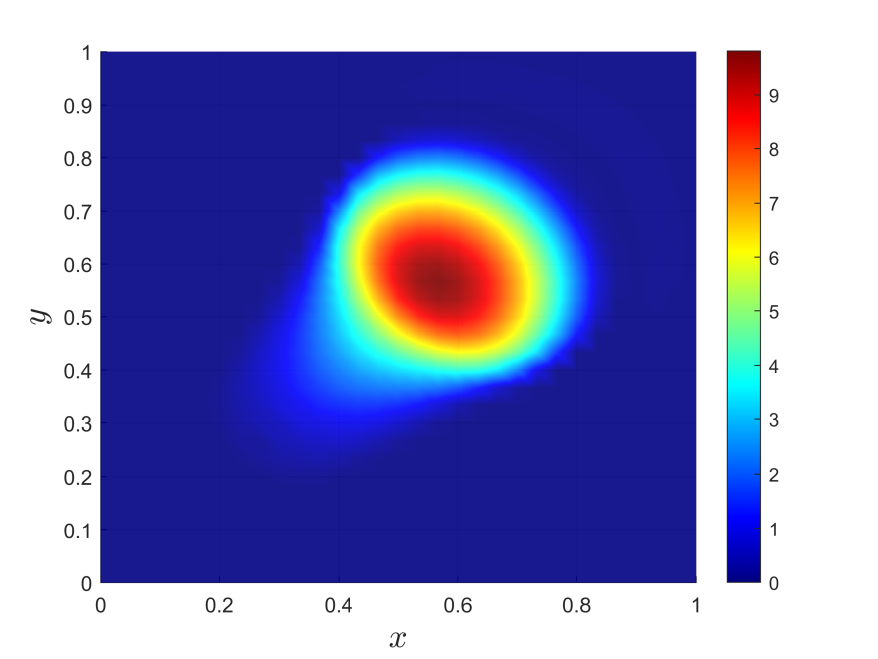}\hspace{-0.1in}
\includegraphics[width=1in,height=0.8in]{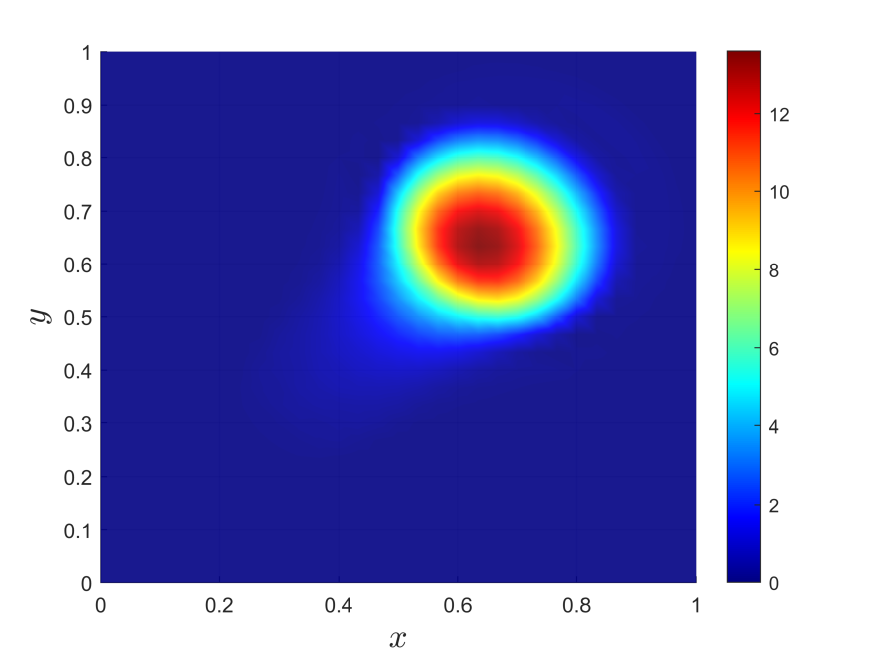}\\
\includegraphics[width=1in,height=0.8in]{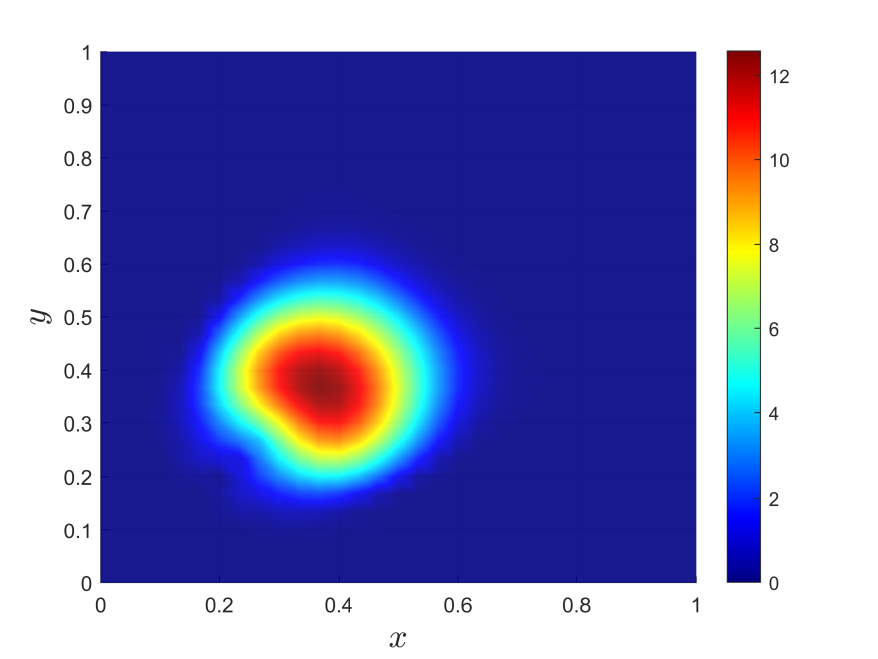}\hspace{-0.1in}
\includegraphics[width=1in,height=0.8in]{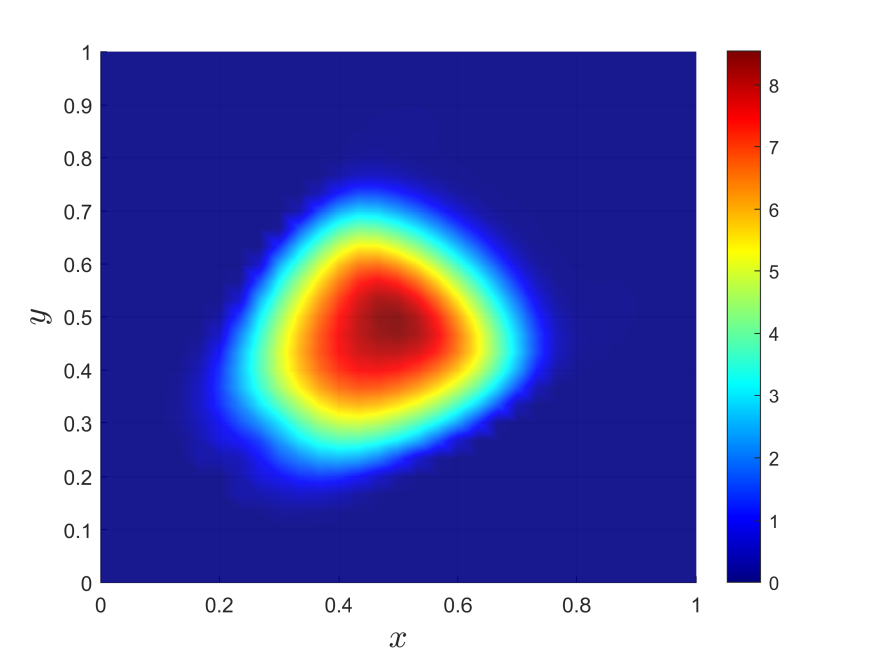}\hspace{-0.1in}
\includegraphics[width=1in,height=0.8in]{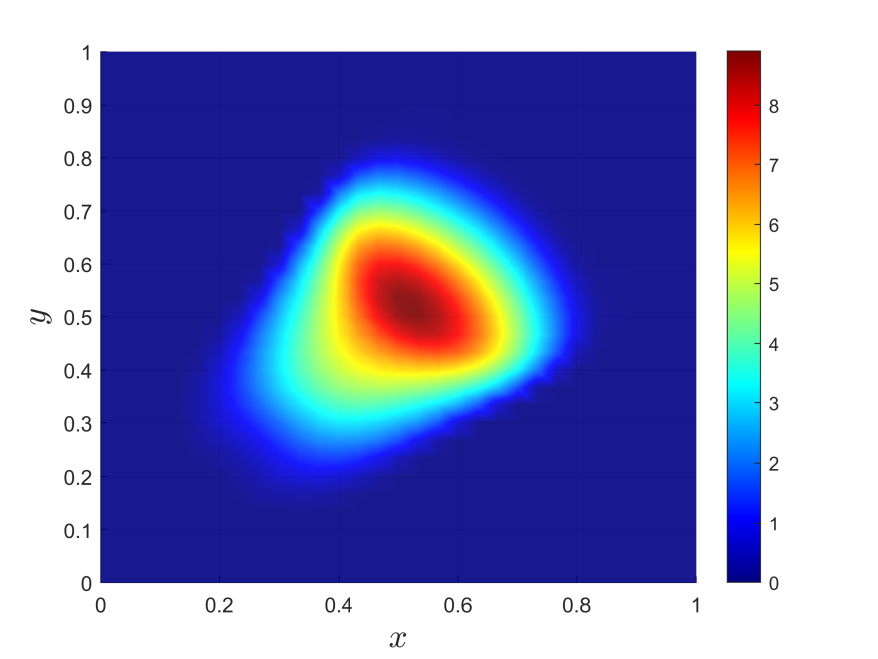}\hspace{-0.1in}
\includegraphics[width=1in,height=0.8in]{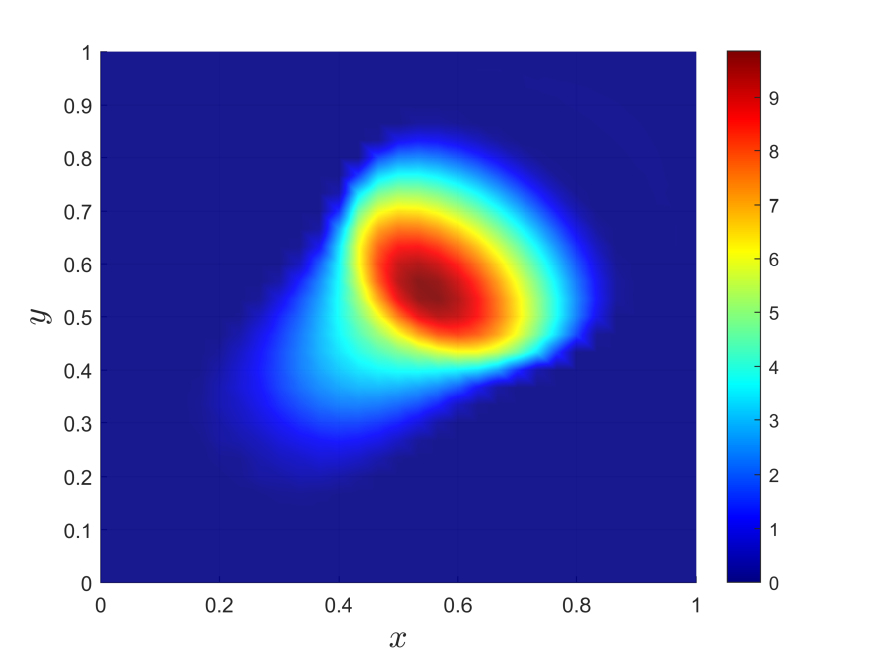}\hspace{-0.1in}
\includegraphics[width=1in,height=0.8in]{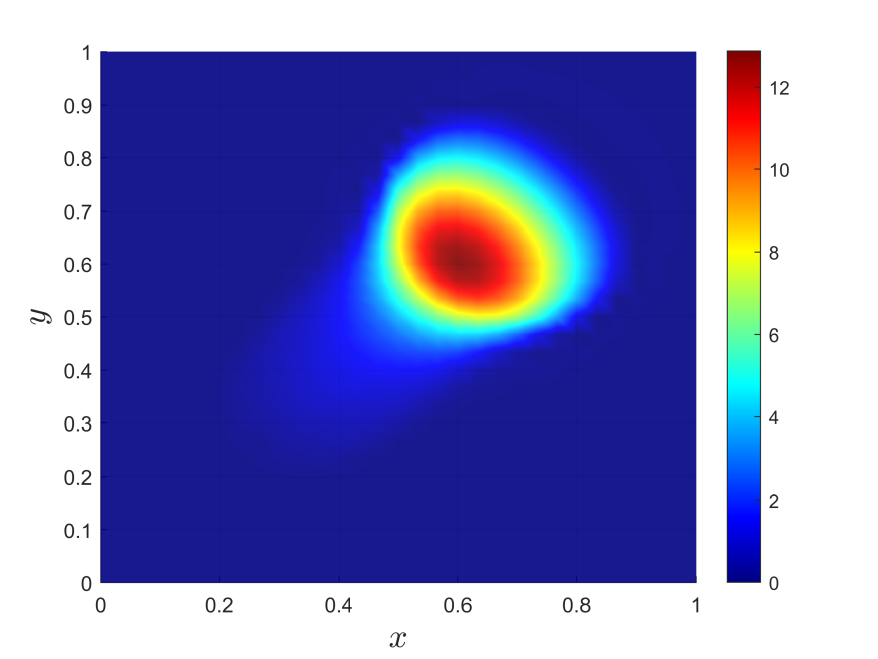}
\caption{Performance of the two-dimensional  OT problem in \S \ref{S:Perf2d}. Left to right: Snapshot of $\rho$ at $t=0.1$,  $0.3$, $0.5$, $0.7$, and $0.9$.
The first row: the integer-order OT problem; the second through fourth rows: the corresponding plots for fractional OT problem with $\alpha =0.9$, $0.8$, and $0.6$ respectively, in  \eqref{Model:set}.}
\label{figure2d1:FOT}
\end{figure}
\subsection{MFP with obstacles}\label{MFP:Obs}

Let $[0, T] = [0, 1]$. We consider a transportation domain $\Omega = [-\f{1}{2}, \f{1}{2}]^2$ in which there is a spatial obstacle that the mass cannot cross and it takes extra efforts for masses (agents) to pass through the obstacle region. One possible application of this problem is the case of a subway gate that a mass of individuals has to cross to reach a final destination. This problem yields the irregular spatial domain which greatly complicates the numerical implementation.
\begin{figure}[h]
\setlength{\abovecaptionskip}{0pt}
\centering
\includegraphics[width=1in,height=1in]{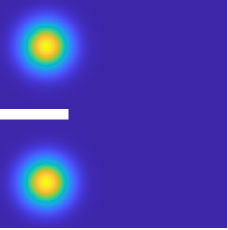} \hspace{0.4in}
\includegraphics[width=1in,height=1in]{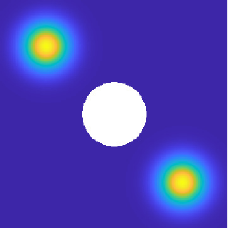}
\caption{Initial density $\rho_0$, terminal density $\rho_1$ and obstacle region $\Omega_0$ highlighted as white regions. Left: $\bm {\mu_0} = (-0.3, 0.3)$ and $\bm {\mu_1} = (-0.3, -0.3)$. Right: $\bm {\mu_0} = (-0.3, 0.3)$ and $\bm {\mu_1}= (0.3, -0.3)$.}
\label{Maze:OT1}
\end{figure}

\begin{figure}[h]
\setlength{\abovecaptionskip}{0pt}
\centering
\includegraphics[width=0.8in,height=0.8in]{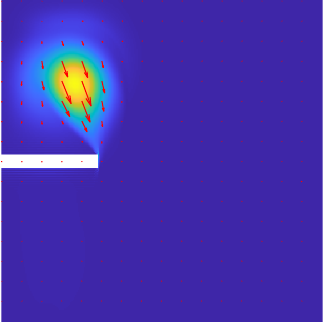}
\includegraphics[width=0.8in,height=0.8in]{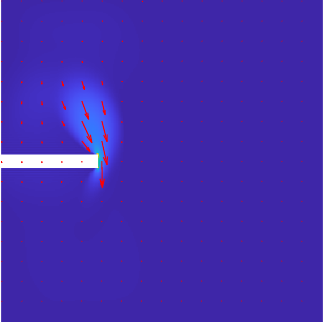}
\includegraphics[width=0.8in,height=0.8in]{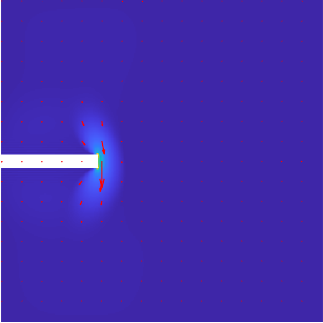}
\includegraphics[width=0.8in,height=0.8in]{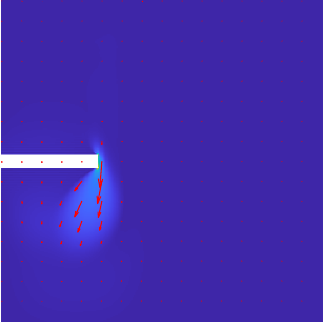}
\includegraphics[width=0.8in,height=0.8in]{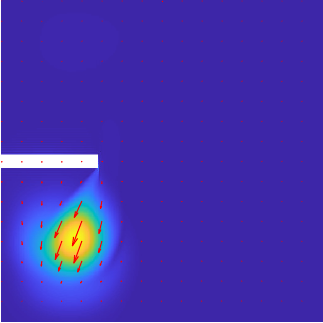} \\[0.1in]
\includegraphics[width=0.8in,height=0.8in]{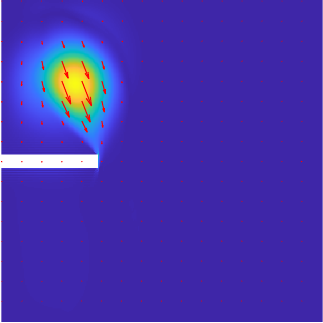}
\includegraphics[width=0.8in,height=0.8in]{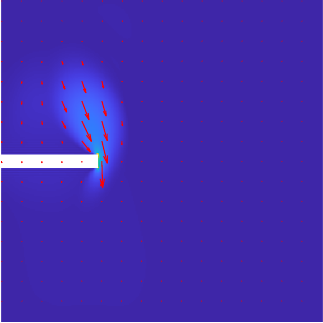}
\includegraphics[width=0.8in,height=0.8in]{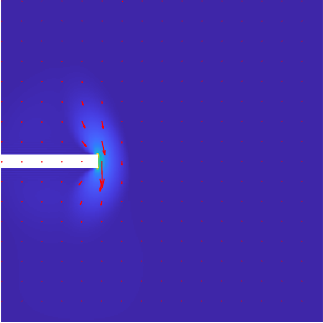}
\includegraphics[width=0.8in,height=0.8in]{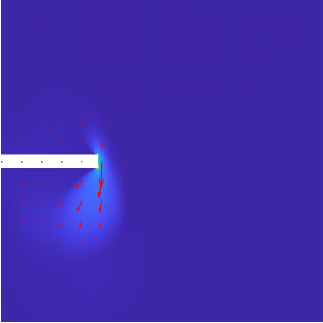}
\includegraphics[width=0.8in,height=0.8in]{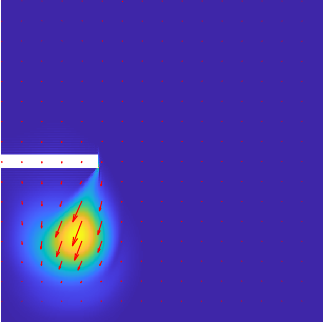} \\[0.1in]
\includegraphics[width=0.8in,height=0.8in]{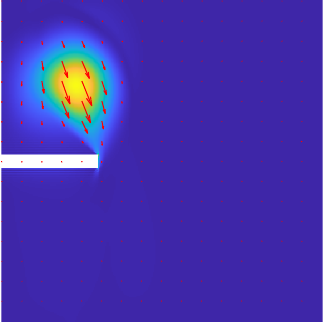}
\includegraphics[width=0.8in,height=0.8in]{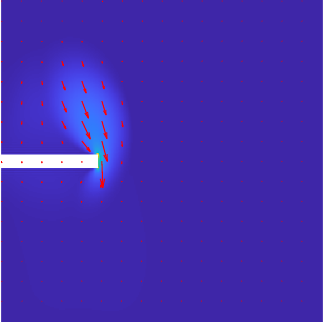}
\includegraphics[width=0.8in,height=0.8in]{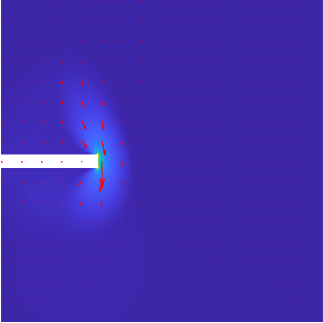}
\includegraphics[width=0.8in,height=0.8in]{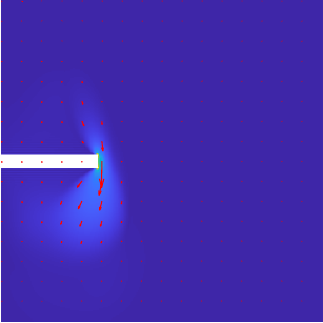}
\includegraphics[width=0.8in,height=0.8in]{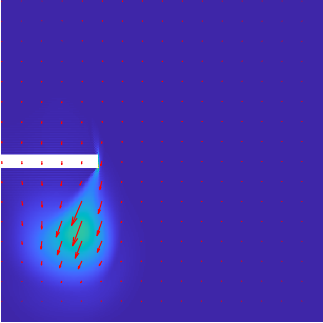} \\[0.1in]
\includegraphics[width=0.8in,height=0.8in]{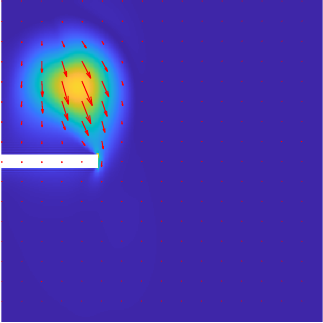}
\includegraphics[width=0.8in,height=0.8in]{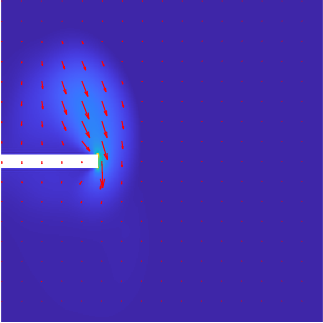}
\includegraphics[width=0.8in,height=0.8in]{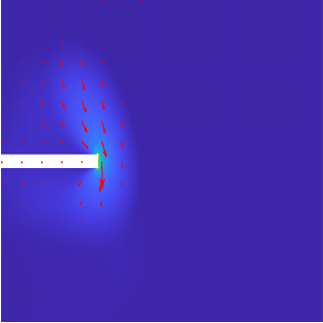}
\includegraphics[width=0.8in,height=0.8in]{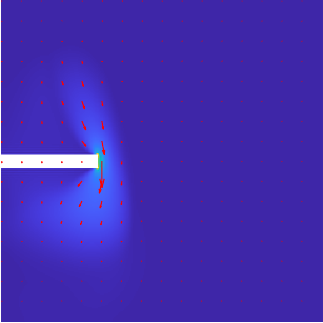}
\includegraphics[width=0.8in,height=0.8in]{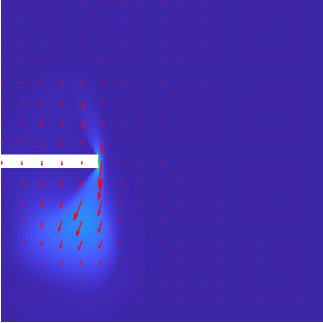} \\
\caption{Performance of the two-dimensional  MFP problem at $t=0.1$, $0.3$, $0.5$, $0.7$, and $0.9$  in \S \ref{MFP:Obs}.
The first row: the integer-order MFP problem; the second through fourth rows: the corresponding plots for fractional MFP problem with $\alpha =0.9$, $0.8$, and $0.6$ respectively, in  \eqref{Model:set}.}
\label{Maze:FOT1}
\end{figure}

As discussed in \S \ref{frac:MFP}, one potential method of handling the irregular domain is to set $Q(\bm x) = \mathbbm {1}_{\Omega_0}(\bm x)$ in the MFP problem \eqref{Model:MFP}--\eqref{F},   an indicator function of the obstacle region $\Omega_0$
and to choose a very large parameter $\lambda_Q$ (e.g.,  $\lambda_Q = 8 \times 10^4$) and $\lambda_R =0$ in \eqref{F}.  We choose the initial and terminal densities to be Gaussian distribution densities with mean $\bm {\mu_0} $, $\bm {\mu_1} $ and the standard deviation $ {\sigma_0} =  {\sigma_1} = 7 \times 10^{-2}$, respectively. Different choices of $\rho_0$, $\rho_1$, and $Q$ are shown in Figure \ref{Maze:OT1}.

\begin{figure}[!htbp]
\setlength{\abovecaptionskip}{0pt}
\centering
\includegraphics[width=0.8in,height=0.8in]{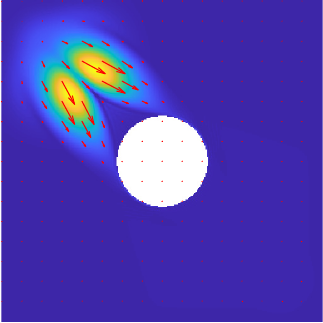}
\includegraphics[width=0.8in,height=0.8in]{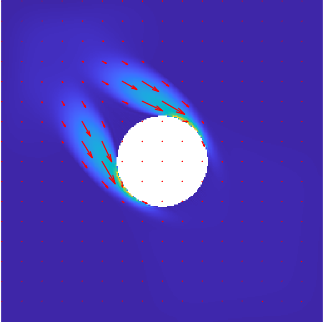}
\includegraphics[width=0.8in,height=0.8in]{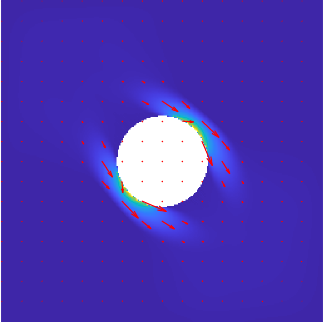}
\includegraphics[width=0.8in,height=0.8in]{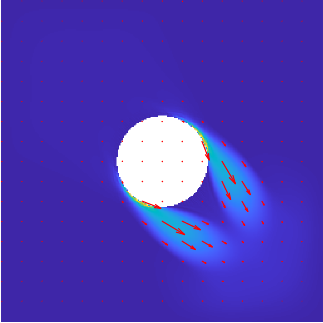}
\includegraphics[width=0.8in,height=0.8in]{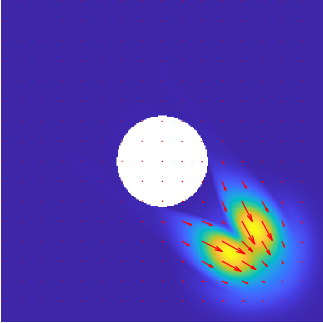} \\[0.1in]
\includegraphics[width=0.8in,height=0.8in]{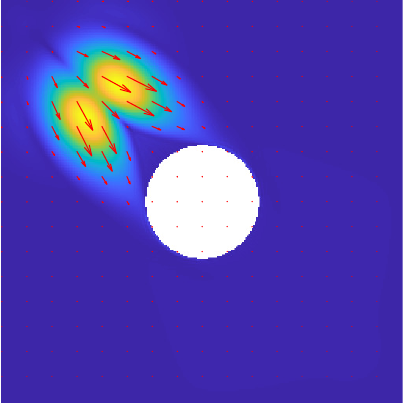}
\includegraphics[width=0.8in,height=0.8in]{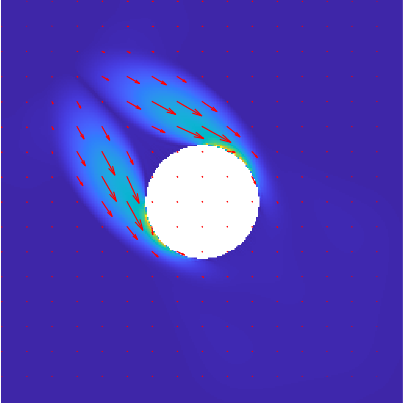}
\includegraphics[width=0.8in,height=0.8in]{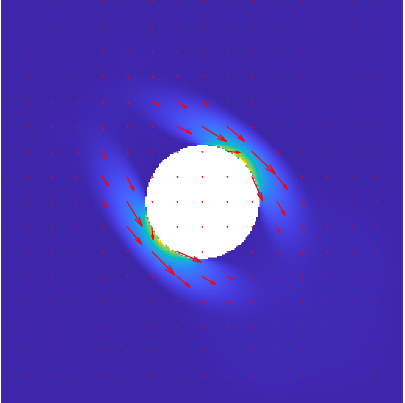}
\includegraphics[width=0.8in,height=0.8in]{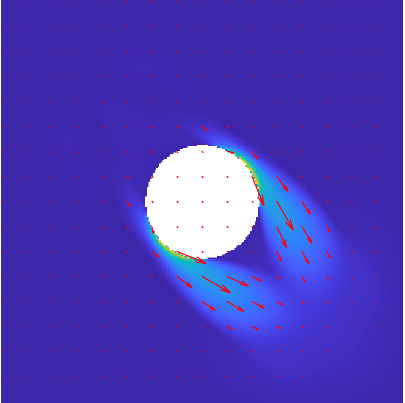}
\includegraphics[width=0.8in,height=0.8in]{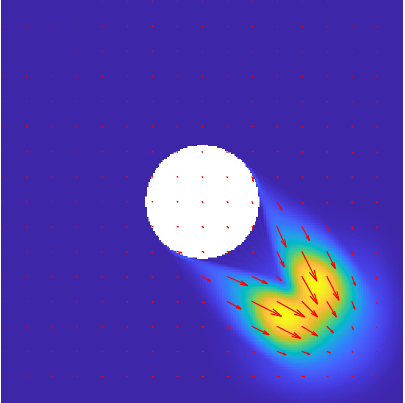} \\[0.1in]
\includegraphics[width=0.8in,height=0.8in]{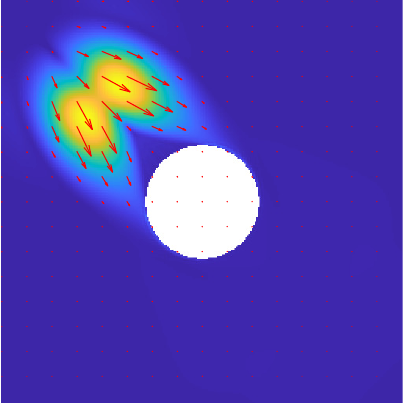}
\includegraphics[width=0.8in,height=0.8in]{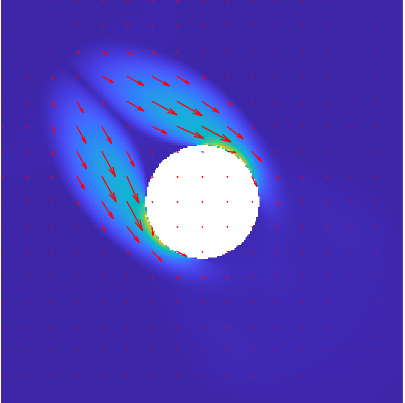}
\includegraphics[width=0.8in,height=0.8in]{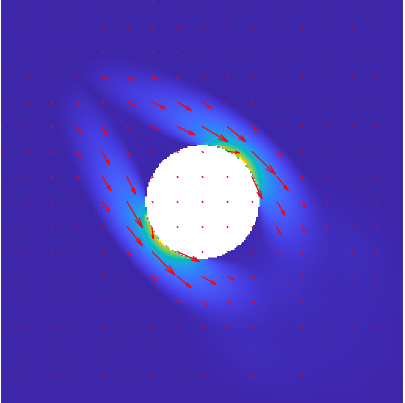}
\includegraphics[width=0.8in,height=0.8in]{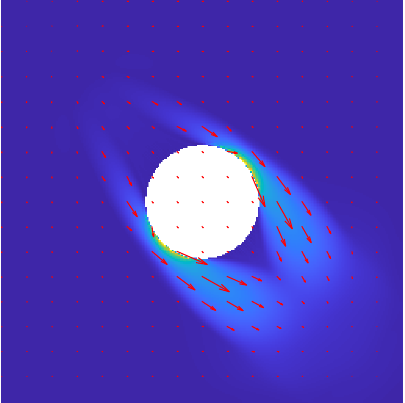}
\includegraphics[width=0.8in,height=0.8in]{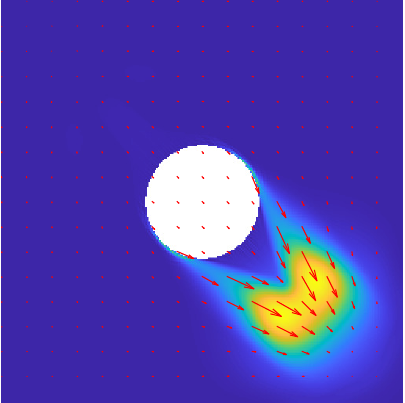} \\[0.1in]
\includegraphics[width=0.8in,height=0.8in]{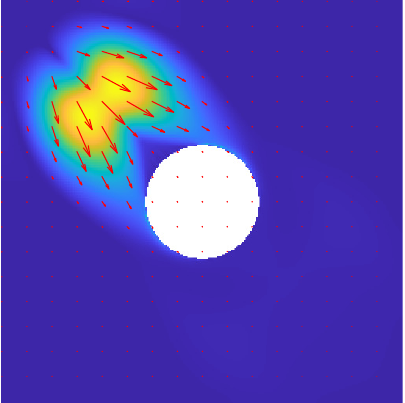}
\includegraphics[width=0.8in,height=0.8in]{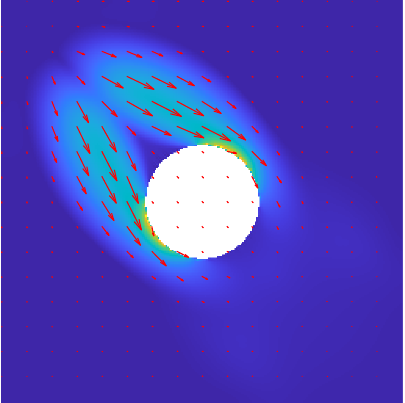}
\includegraphics[width=0.8in,height=0.8in]{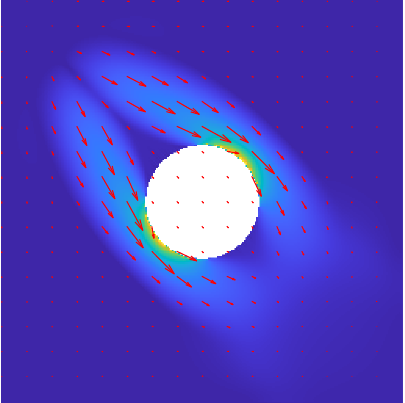}
\includegraphics[width=0.8in,height=0.8in]{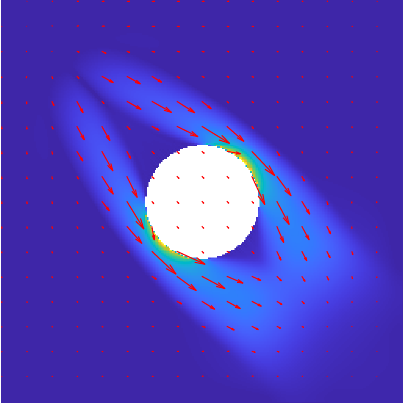}
\includegraphics[width=0.8in,height=0.8in]{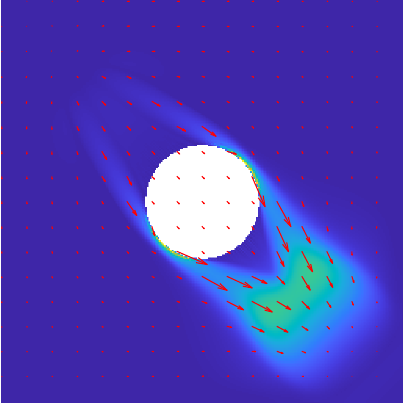} \\
\caption{Performance of the two-dimensional  MFP problem at $t=0.1$, $0.3$, $0.5$, $0.7$, and $0.9$  in \S \ref{MFP:Obs}.
The first row: the integer-order MFP problem; the second through fourth rows: the corresponding plots for fractional MFP problem with $\alpha =0.9$, $0.8$, and $0.6$ respectively, in  \eqref{Model:set}.}
\label{Maze:FOT2}
\end{figure}

The snapshots of the mass evolution shown in Figures \ref{Maze:FOT1}--\ref{Maze:FOT2} demonstrate that  the mass governed by both the integer-order PDE and the fractional PDEs with different fractional orders  circumvents the obstacles very well, which meets our expectation. In addition, the mass governed by the fractional PDE exhibits heavy tails and propagates slower compared with its integer-order counterparts, and these observations are consistent with the discussions in \S \ref{S:Perf2d}.

\subsection{MFP between images}\label{MFP:image}

Our last example concerns with OT and MFP problems \eqref{Model:MFP} between images, which shows the effectiveness and flexibility of our proposed Algorithm  \ref{algorithm}.
The data are as follows: $\Omega =[0,1]^2 $, $[0, T] = [0, 1]$,  the initial and terminal densities $\rho_0$, $\rho_1$, and  the interaction penalty
$$Q(\bm{x})=\left\{
\begin{array}{ll}
0, & \rho_0(\bm{x}) \neq 0 \text { or } \rho_1(\bm{x}) \neq 0, \\ [0.1in]
1, & \text { otherwise }
\end{array}\right.$$
 are shown in  Figure \ref{Image:OT1}.

\begin{figure}[!htbp]
\setlength{\abovecaptionskip}{0pt}
\centering
\includegraphics[width=1in,height=1in]{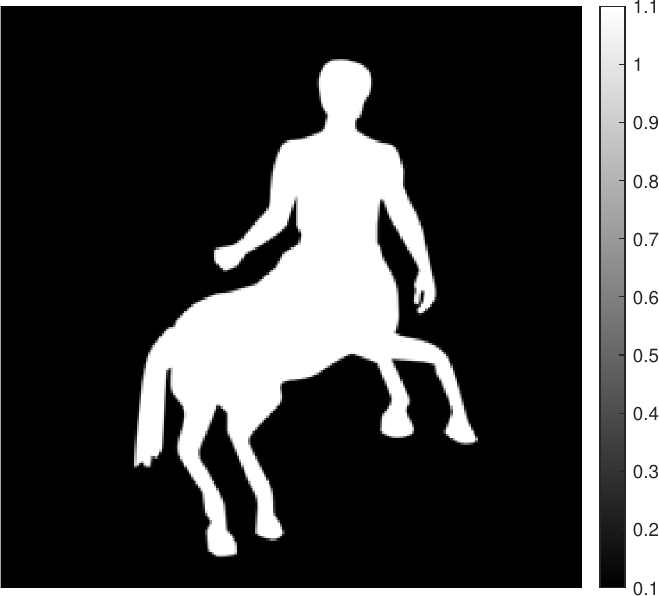} \hspace{0.2in}
\includegraphics[width=1in,height=1in]{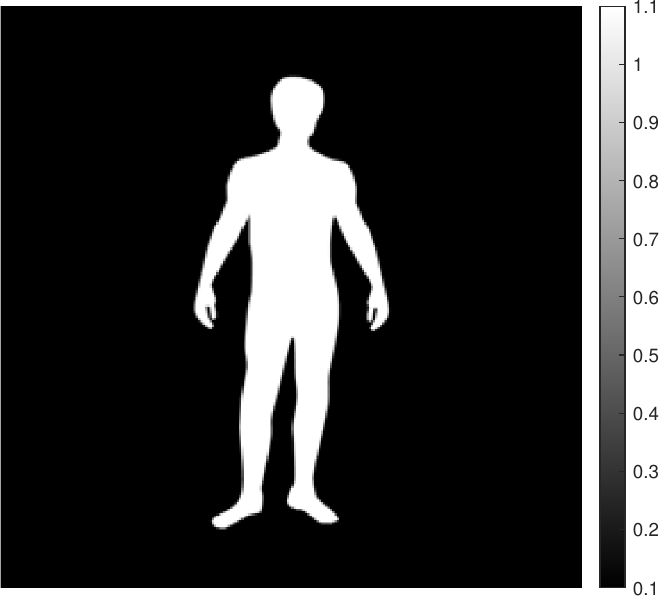} \hspace{0.2in}
\includegraphics[width=1in,height=1in]{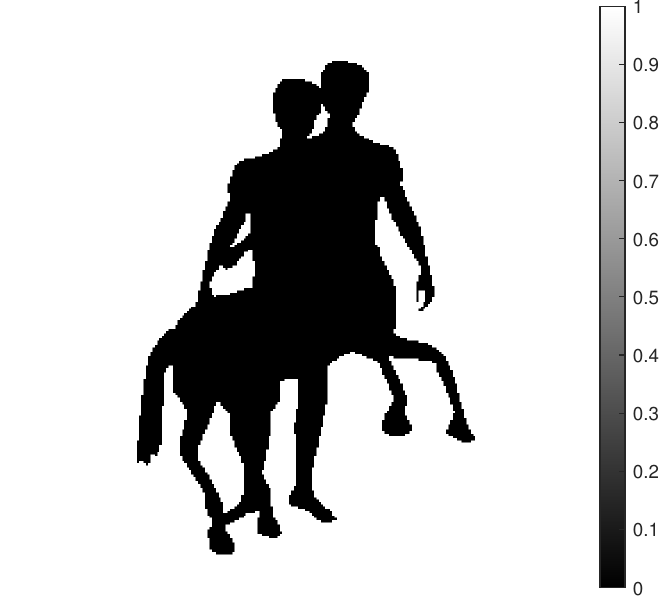}\hspace{0.2in}
\caption{Left to right: initial density $\rho_0$, final density $\rho_1$, and interaction penalty $Q$.}
\label{Image:OT1}
\end{figure}

\begin{figure}[!htbp]
\setlength{\abovecaptionskip}{0pt}
\centering
\includegraphics[scale=0.35]{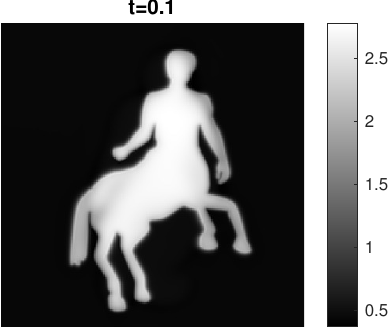}
\includegraphics[scale=0.35]{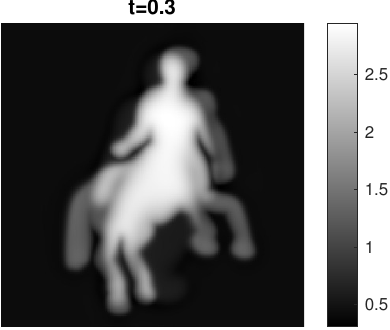}
\includegraphics[scale=0.35]{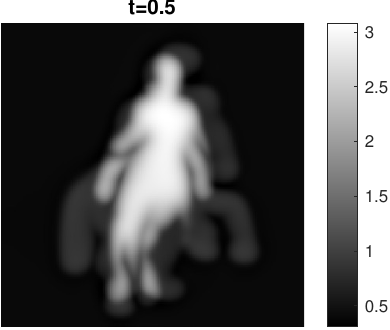}
\includegraphics[scale=0.35]{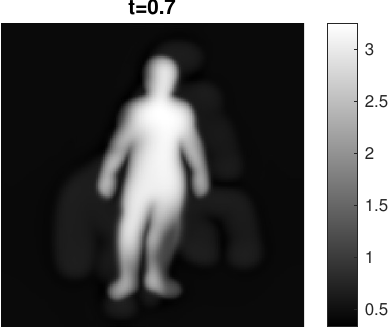}
\includegraphics[scale=0.35]{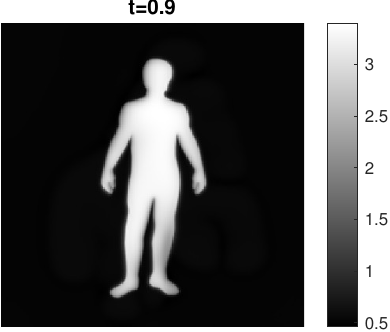}\\[0.1in]
\includegraphics[scale=0.37]{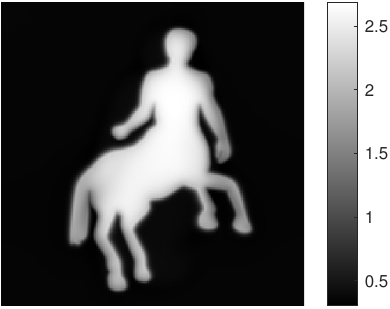}
\includegraphics[scale=0.35]{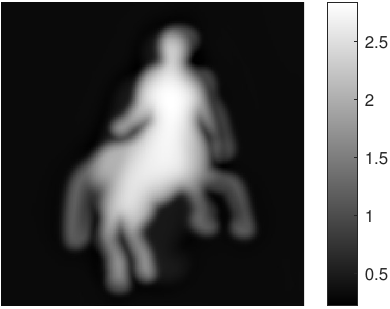}
\includegraphics[scale=0.35]{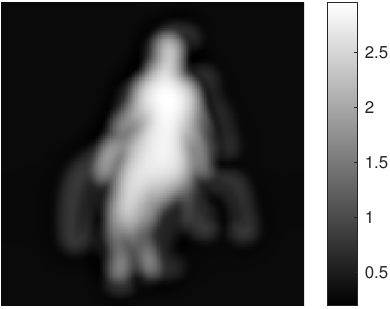}
\includegraphics[scale=0.35]{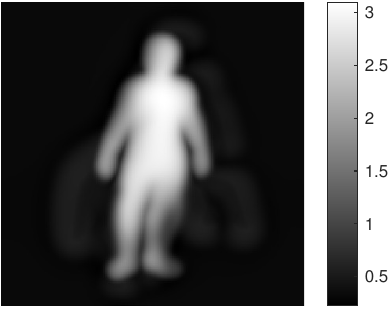}
\includegraphics[scale=0.35]{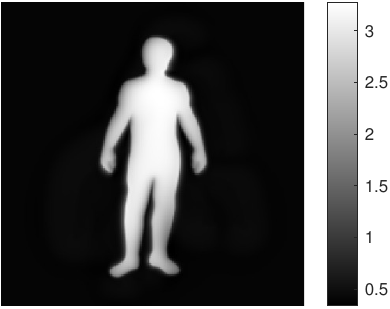}\\[0.1in]
\includegraphics[scale=0.35]{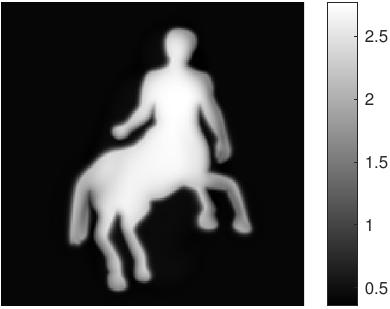}
\includegraphics[scale=0.35]{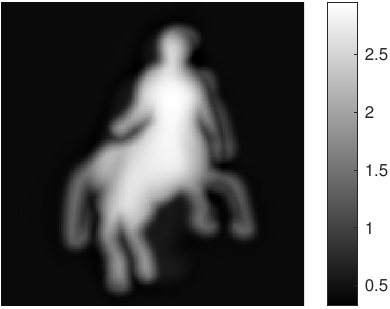}
\includegraphics[scale=0.35]{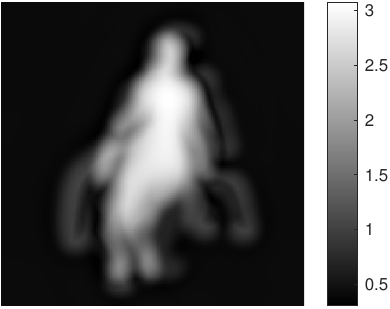}
\includegraphics[scale=0.35]{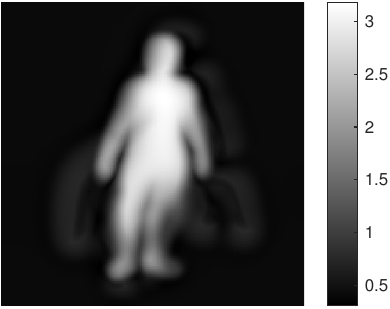}
\includegraphics[scale=0.35]{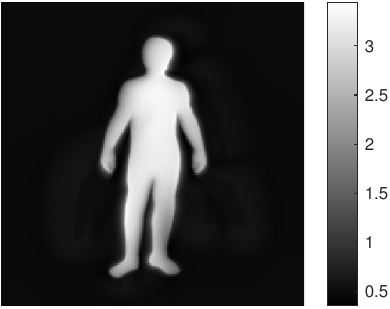}\\[0.1in]
\includegraphics[scale=0.35]{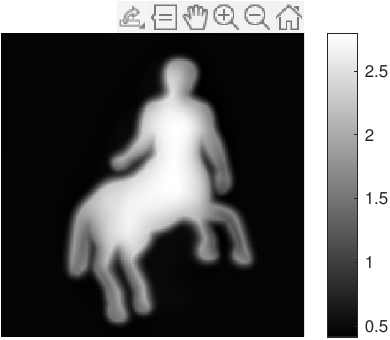}
\includegraphics[scale=0.35]{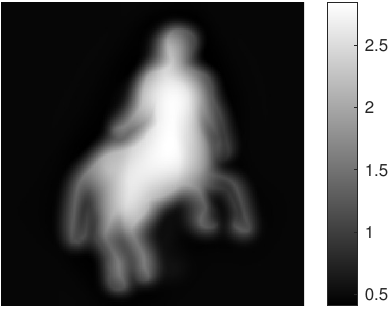}
\includegraphics[scale=0.35]{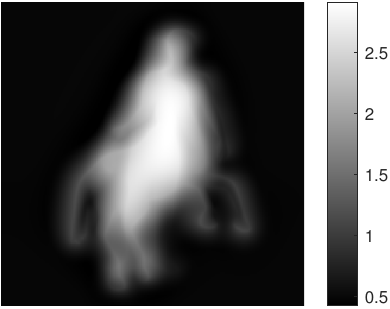}
\includegraphics[scale=0.35]{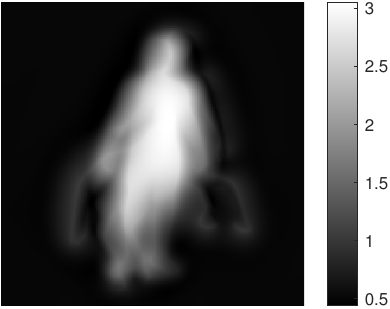}
\includegraphics[scale=0.35]{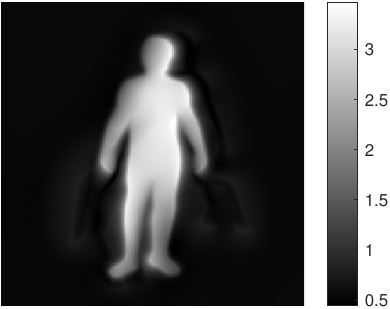}\\[0.1in]
\caption{Case 1: Performance of the two-dimensional OT problem at $t=0.1$, $0.3$, $0.5$, $0.7$, and $0.9$  in \S \ref{MFP:image}.
The first row: the integer-order OT problem; the second through fourth rows: the corresponding plots for fractional OT problem with $\alpha =0.9$, $0.8$, and $0.6$ respectively, in  \eqref{Model:set}.}
\label{Image:FOT1}
\end{figure}

\begin{figure}[h]
\setlength{\abovecaptionskip}{0pt}
\centering
\includegraphics[scale=0.35]{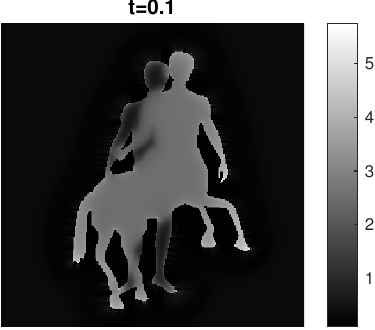}
\includegraphics[scale=0.35]{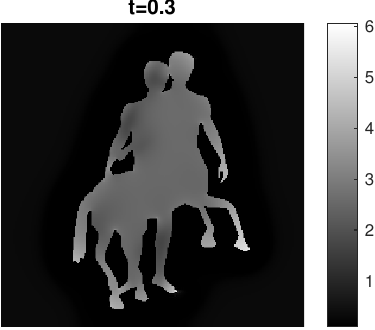}
\includegraphics[scale=0.35]{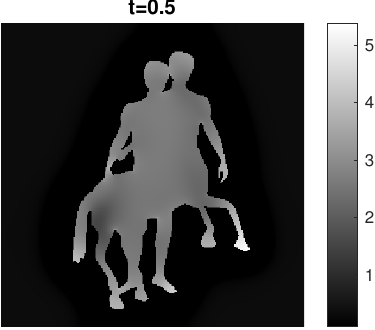}
\includegraphics[scale=0.35]{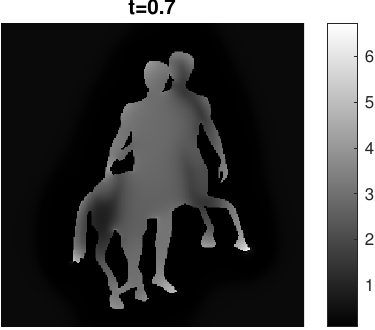}
\includegraphics[scale=0.35]{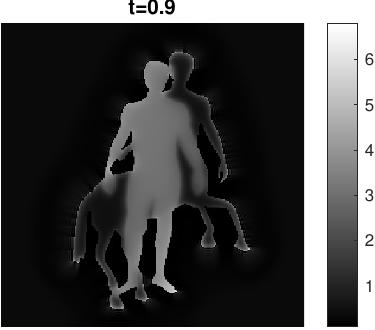} \\[0.1in]
\includegraphics[scale=0.35]{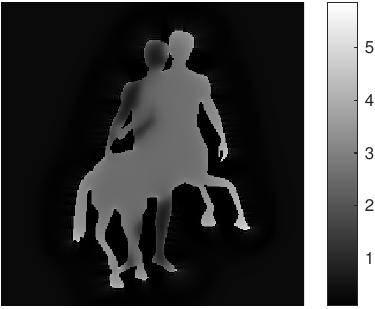}
\includegraphics[scale=0.35]{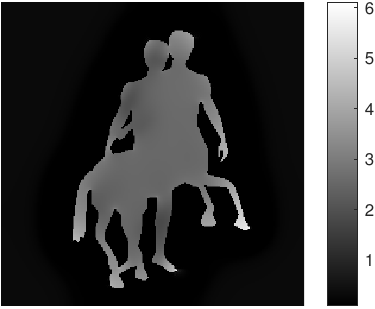}
\includegraphics[scale=0.35]{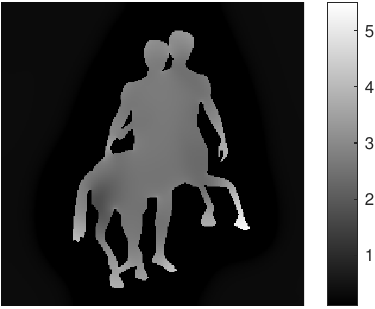}
\includegraphics[scale=0.35]{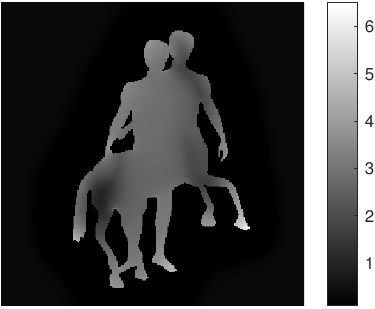}
\includegraphics[scale=0.35]{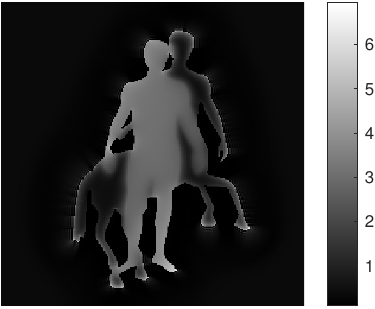} \\[0.1in]
\includegraphics[scale=0.35]{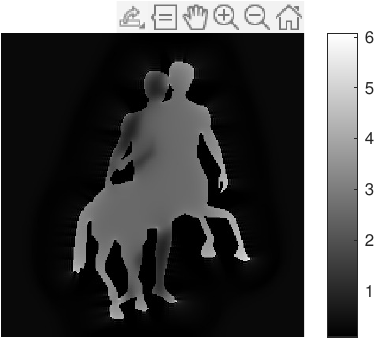}
\includegraphics[scale=0.35]{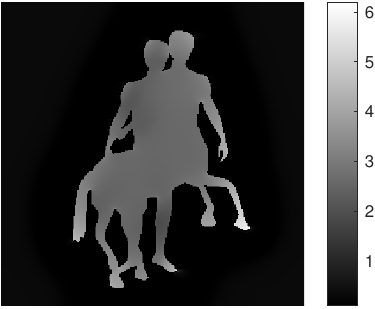}
\includegraphics[scale=0.35]{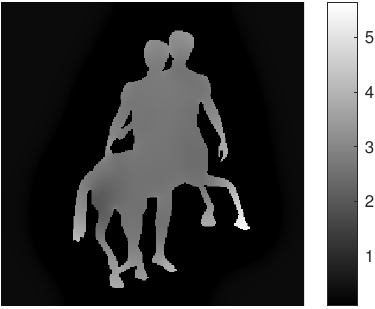}
\includegraphics[scale=0.35]{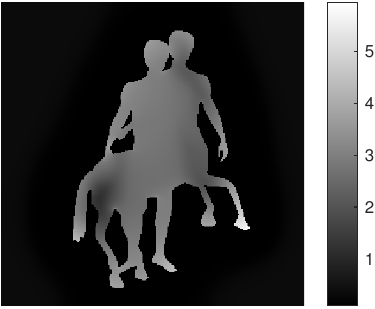}
\includegraphics[scale=0.35]{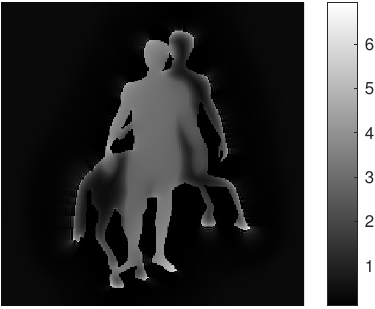}\\[0.1in]
\includegraphics[scale=0.35]{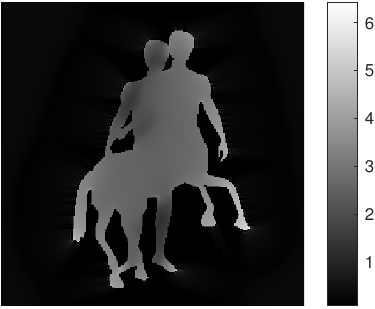}
\includegraphics[scale=0.35]{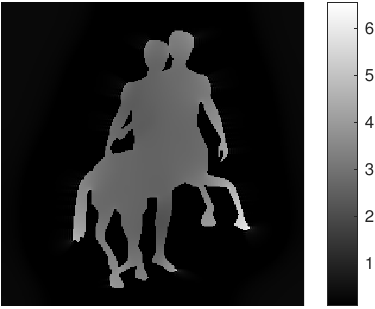}
\includegraphics[scale=0.35]{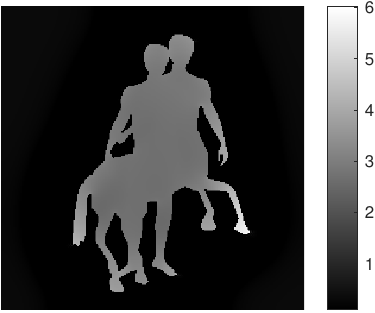}
\includegraphics[scale=0.35]{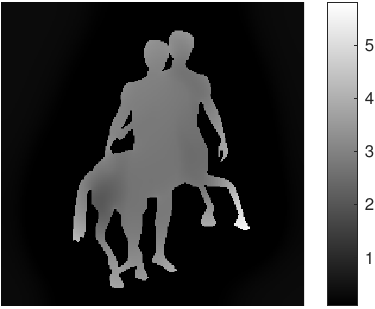}
\includegraphics[scale=0.35]{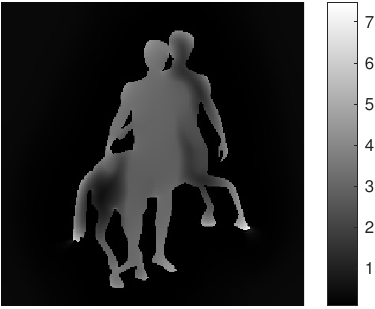}\\[0.1in]
\caption{Case 2: Performance of the two-dimensional MFP problem at $t=0.1$, $0.3$, $0.5$, $0.7$, and $0.9$   in \S \ref{MFP:image}.
The first row: the integer-order MFP problem; the second through fourth rows: the corresponding plots for fractional MFP problem with $\alpha =0.9$, $0.8$, and $0.6$ respectively, in  \eqref{Model:set}.}
\label{Image:FOT2}
\end{figure}

We apply the proposed algorithm to the following problems:

\begin{itemize}
  \item (Case 1)  OT problem \eqref{Model:OT}
  \item  (Case 2) MFP problem \eqref{Model:MFP} with
    $   R (\rho) : = \frac{\rho^2}{2}$,  $ \lambda_R     = 0.01$ and $ \lambda_Q =0.1$ in \eqref{F}
\end{itemize}

 Snapshots of the density evolution  are presented in Figures \ref{Image:FOT1}--\ref{Image:FOT2}, from which we have the following observations: \textbf{(i)} Case 1 (OT) produces a more free density evolution path while the interaction cost in Case 2 restricts the mass evolution within the dark region of the interaction penalty $Q$ (cf.  Figure \ref{Image:OT1}), which is consistent with the fact that MFP sets an additional moving  preference for the density. For both two cases, \textbf{(ii)} the density of fractional MFP problem ($\alpha=0.6$) at $t=0.7$s exhibits similar behavior as that of fractional MFP problem ($\alpha=0.9$)   and that of integer-order MFP problem at $t=0.5$s, which indicates that the density of fractional MFP problem evolves slower than that of integer-order MFP problem, and exhibits comparatively faster evolution speed as the fractional order increases, as we have found in Section \ref{S:Perf}.

\section*{Acknowledgments}

Y. Li's work is partially supported by the National Science Foundation under Grant DMS-2012291 and by the Start-up funding of W. Li at University of South Carolina; W. Li's work is partially supported by AFOSR YIP award No. FA9550-23-1-008, NSF DMS-2245097, and NSF RTG: 2038080; H. Wang's work is partially supported by the National Science Foundation under Grants DMS-2012291 and DMS 2245097.

	\end{document}